\documentclass[10pt]{article}
\usepackage{graphicx} % Required for inserting images
\usepackage[letterpaper]{geometry}
\usepackage{amsmath,amssymb,amsthm}
\usepackage{caption}
\usepackage{subcaption}
\usepackage{hyperref}
\usepackage[doi=false,isbn=false,url=false,eprint=false,maxbibnames=9]{biblatex}
\usepackage{algorithm}
\usepackage{algpseudocode}
\usepackage{tikz}
\usetikzlibrary{positioning}
\usetikzlibrary{shapes.geometric}
\usetikzlibrary{patterns,plotmarks}
\usepackage{xcolor}
\usepackage{subcaption}
\newcommand{\mc}{\mathcal}
\newcommand{\mb}{\mathbf}
\newcommand{\R}{\mathbb{R}}
\newcommand{\E}{\mathbb{E}}

\newcommand{\N}{\mathbb{N}}

\newcommand{\var}{\mathsf{Var}}

\newcommand{\domain}{\mathsf{dom}}
\newcommand{\diag}{\mathsf{diag}}
\newcommand{\agent}{\mathsf{Agent}}
\newcommand{\sign}{\mathsf{sign}}

\newcommand{\setspan}{\mathsf{Span}}

\DeclareMathOperator*{\argmin}{argmin}

\newcommand{\child}{\mathsf{Chi}}
\newcommand{\Descendant}{\mathsf{Des}}

\usepackage{authblk}

\algrenewcommand\algorithmicrequire{\textbf{Input:}}
\algrenewcommand\algorithmicensure{\textbf{Output:}}

\newtheorem{theorem}{Theorem}[section]
\newtheorem{lemma}{Lemma}[section]
\newtheorem{definition}{Definition}[section]
\newtheorem{corollary}{Corollary}[section]
\newtheorem{assumption}{Assumption}[section]

\title{Multi-Timescale Primal Dual Hybrid Gradient with Application to Distributed Optimization}

\author[1]{Junhui Zhang}
\author[1,2]{Patrick Jaillet}
\affil[1]{Operations Research Center, MIT}
\affil[2]{Department of Electrical Engineering and Computer Science, MIT}

\date{}
\begin{document}
\maketitle
\begin{abstract}
We propose two variants of the Primal Dual Hybrid Gradient (PDHG) algorithm for saddle point problems with block decomposable duals, hereafter called Multi-Timescale PDHG (MT-PDHG) and its accelerated variant (AMT-PDHG). Through novel mixtures of Bregman divergence and multi-timescale extrapolations, our MT-PDHG and AMT-PDHG converge under arbitrary updating rates for different dual blocks while remaining fully deterministic and robust to extreme delays in dual updates. 

We further apply our (A)MT-PDHG, augmented with the gradient sliding techniques introduced in \cite{lan_communication-efficient_2020,lan_gradient_2016}, to distributed optimization. The flexibility in choosing different updating rates for different blocks allows a more refined control over the communication rounds between different pairs of agents, thereby improving the efficiencies in settings with heterogeneity in local objectives and communication costs. Moreover, with careful choices of penalty levels, our algorithms show linear and thus optimal dependency on function similarities, a measure of how similar the gradients of local objectives are. This provides a positive answer to the open question whether such dependency is achievable for non-smooth objectives \cite{ArjevaniShamir2015Communication}.
\end{abstract}
% \tableofcontents
\section{Introduction}

We study the following saddle point problem with block decomposable dual variables and objectives:
\begin{equation}\label{eq:objective_pd_block}\tag{$\mc P_b$}
    \min_{X\in \mc X} \max_{y_s\in \R^{n_s},~s\in [S]} F(X)+\sum_{s=1}^S\left( \langle K_sX,y_s\rangle  -  R_s^*(y_s)\right),
\end{equation}
where $\mc X\subset \R^{d}$ is a non-empty convex set and $F:\mc X\to \R$ is a convex function such that there exists $\mu,M\geq 0$, 
\begin{equation}\label{eq:prop-F}
    \frac{\mu}{2}\|X-X'\|^2\leq F(X) - F(X')-\langle F' (X'),X-X')\rangle\leq M\|X-X'\|,\quad\forall X,X'\in \mc X,
\end{equation}
where $F':\mc X\to \R^{d}$ is a subgradient oracle, i.e. for each $X'\in \mc X$, $F'(X')\in \partial F(X')$ is a subgradient. For instance, when $\|F'\|_*\leq M'$, then $M = 2M'$ holds.\footnote{This is because $F(X') \geq  F(X)+ \langle F'(X), X'-X\rangle$, and so using Cauchy–Schwarz inequality 
\begin{displaymath}
    F(X) - F(X') - \langle F'(X'), X-X'\rangle \leq \langle F'(X) - F'(X'), X-X'\rangle \leq 2M \|X-X'\|.
\end{displaymath}} For each $s\in [S]$, $K_s\in \R^{n_s\times d}$ is a matrix, $R_s:\R^{n_s}\to \overline{\R}$ is a proper, convex, and lower-semicontinuous function, and $R_s^*:\R^{n_s}\to \overline{\R}$ is its Fenchel conjugate, defined as $R_s^*(y_s) = \sup_{y_s'\in \R^{n_s}} \langle y_s,y_s'\rangle -R_s(y'_s)$. 

Saddle point problems of the form \eqref{eq:objective_pd_block} with one block ($S=1$) have been widely studied due to its applications in various problems including linear programming \cite{applegate_faster_2023}, distributed optimization \cite{lan_communication-efficient_2020}, and inverse problems such as image denoising \cite{chambolle_ergodic_2016,chambolle_first-order_2011}, to name a few. For $S>1$, stochastic algorithms which update random subsets of dual blocks have been proposed \cite{Chambolle2018Stochastic,ShaiTong2013}. 

In this work, building upon the Primal Dual Hybrid Gradient (PDHG) algorithms\cite{chambolle_ergodic_2016,chambolle_first-order_2011}, we propose Multi-Timescale Primal Dual Hybrid Gradient (MT-PDHG) and its accelerated variant (AMT-PDHG) where dual blocks are updated periodically at potentially 
different rates. Our algorithms achieve the following three properties.
\begin{enumerate}
    \item \underline{\textit{Flexible}}. (A)MT-PDHG converge under arbitrary updating rates for the dual blocks. 
    \item \underline{\textit{Deterministic}}. The updating schedules for the dual blocks are deterministic and periodic. 
    \item \underline{\textit{Robust}}. The convergence rates depend on the average, instead of the maximum, of the updating rates, thereby are robust to extreme delays (i.e. large updating rates). 
\end{enumerate}

 The motivation behind this work is distributed optimization, a branch of optimization where multiple agents, each having access to only partial information about the (global) objective, work together to solve the global problem. As an example, in distributed empirical risk minimization for machine learning, the global objective function is the sum of local loss functions, each depending on the local dataset which is only available to one agent \cite{AssranAytekinFeyzmahdavianJohanssonRabbat2020,Boyd2011,Zhu2016,ArjevaniShamir2015Communication,Jakub2016}. Examples of other applications include power system control \cite{Molzahn2017,Nedic2018}, multi-robot system control \cite{Bullo2009,Nedic2018,Jaleel2020,Shorinwa2024}, and signal processing \cite{BertsekasTsitsiklis1989,Dan2002,Rabbat2004}. More precisely, we study the following distributed optimization problem
\begin{equation}\label{eq:objective_f}\tag{$\mc P_{d}$}
    \min_{x\in \overline{\mc X}} \sum_{v\in V} f_v(x)
\end{equation}
where $V = [m]$ represents $m$ agents, $\overline{\mc X}\subset \R^{\overline{d}}$ is a nonempty, closed convex set, and $x^*\in \overline{\mc X}$ is an optimal solution to \eqref{eq:objective_f}. 

\textbf{From \eqref{eq:objective_pd_block} to lifted space reformulation of \eqref{eq:objective_f}.} We consider a lifted space reformulation of \eqref{eq:objective_f} -- \eqref{eq:objective_F} in Section \ref{sec:do-setup} -- where agent $v$ maintains and updates $x_v$, a local version of the decision variable $x$. The objective then becomes $\sum_{v\in V} f_v(x_v)$, with additional consensus constraints encouraging $x_v\approx x_{v'}$ for all $v,v'\in V$. The classical PDHG algorithm, when implemented in the distributed setting, requires  communication only when dual variables are updated.

\textbf{From \eqref{eq:objective_f} and \eqref{eq:objective_F} to algorithm design goals for \eqref{eq:objective_pd_block}.} We target at distributed optimization under: (1) heterogeneity in local objectives and communication costs, and (2) communication bottlenecks with periodic connectivity constraints.

As an example of heterogeneity, in distributed empirical risk minimization, local loss functions are different due to randomness in the local datasets and variation of data distributions \cite{Zhu2016,Kovalev2022,pmlr-v32-shamir14,pmlr-v37-zhangb15,pmlr-v151-tian22b,ArjevaniShamir2015Communication}; the costs of communication could depend on factors such as the distance between agents, methods of communication, and amounts of data sent \cite{Berahas2019,BertsekasTsitsiklis1989,Scaman2019,Tyurin2024,Santosh2020}. The heterogeneity makes it desirable to design algorithms allowing more \underline{\textit{flexible}} control over the numbers of communication rounds among different subsets of agents, which translates to the numbers of updates applied to each dual blocks. Existing block coordinate descent algorithms allow some level of control on this by using different selection probabilities for different blocks. However, stochastic algorithms could be impractical or inefficient due to factors such as unpredictability, random memory access \cite{Sun2017}, sampling overhead \cite{Csiba2018}, and physical constraints on connectivity between agents \cite{Angelia2015}. 

We consider a ``rate based design'', where communication is scheduled at different rates between different pairs of agents, \underline{\textit{deterministically}}. In practice, the rates could be either picked by users and/or subject to physical, bandwidth, and/or energy constraints \cite{Zeng2021}. This brings in one challenge: how to ensure the convergence of the algorithms under periodic communication at different rates. Another challenge, and also our third desideratum, is to ensure that the performance of the algorithms should not be influenced by stragglers (dual blocks which are updated infrequently), meaning that the algorithms should be \underline{\textit{robust}} to extreme values in the updating rates.

\textbf{Results.} Our (A)MT-PDHG achieves the above three goals. To ensure convergence, we use careful mixtures of Bregman divergence and novel multi-timescale extrapolation. For the distributed optimization problem considered, to find an $\epsilon$-suboptimal solution, the complexities of our algorithms achieve optimal dependency on $\epsilon$: MT-PDHG needs $O(\overline{r}A/\epsilon)$ communication rounds and $O(\overline{r}/\epsilon^2)$ subgradient steps for Lipchitz objectives, and AMT-PDHG needs $O(\overline{r}A/\sqrt{\epsilon\mu})$ communication rounds and $O(\overline{r}/(\epsilon\mu))$ subgradient steps if the objectives are also $\mu$-strongly convex. Here, $\overline{r}$ measures the ``average rate of updates'' for dual blocks, and $A$ measures similarities between (subgradients of) local functions. In fact, the linear dependency of communication rounds on $A$ is optimal \cite{ArjevaniShamir2015Communication}, thereby providing a positive answer to the open question whether such dependency is achievable for non-smooth objectives \cite{ArjevaniShamir2015Communication}.

Numerical experiments for linear programming and support vector machine problem with regularized hinge losses confirm the effectiveness of our algorithms and demonstrate the above dependence on $\overline{r}$ and $A$.

\subsection{Related works}

\textbf{Primal-Dual Hybrid Gradient and its block variant.} Our algorithms are built upon the Primal-Dual Hybrid Gradient (PDHG) algorithms \cite{chambolle_first-order_2011,chambolle_ergodic_2016}. For problems with block-decomposable duals, block-coordinate descent type of variants have been proposed, such as the Stochastic Primal-Dual Coordinate (SPDC) \cite{pmlr-v37-zhanga15} and the Stochastic-PDHG (S-PDHG)\cite{Alacaoglu2022,Chambolle2018Stochastic}, to name a few. These algorithms update \textit{random} subsets of the dual blocks at each iteration, where all blocks have strictly positive probability of being selected. Although the $O(1/k)$ rate of convergence still holds,  due to the randomness, the convergence is only shown for the \textit{expected} objective value suboptimality (for SPDC) or \textit{expected} duality gap (for S-PDHG), and could have unsatisfactory performance due to large variance. Moreover, updating random subsets could potentially be inefficient -- due to reasons such as random memory access \cite{Sun2017} and potential overhead in computing sampling distributions \cite{Csiba2018} -- or even impossible due to physical constraints -- such as in distributed settings, blocks available for updates are subject to  network connectivity and communication constraints \cite{Angelia2015}.

Although deterministic block coordinate descent for convex optimization has been shown to converge, such as under the cyclic updating rule \cite{Sun2017,wright_coordinate_2015}, to the best of our knowledge, the \textit{multi-timescale} updating rule we propose is the first \textit{deterministic} block updating rule for PDHG with separable duals, such that different blocks could be updated different numbers of times, and the duality gap converges deterministically at the optimal rate.

As the dual blocks are not updated at each global iteration, they introduce ''outdated'' information to the dynamics. Most existing asynchronous optimization algorithms achieve convergence rates which depend on the \textit{maximum delay} \cite{NiuRechtReWright2011,Liu2015Asynchronous,Peng2016ARock}, with the exception of \cite{Cohen2021,Attia2025}, whose convergence rates depend on the average and the quantiles of delays, respectively. However, \cite{Cohen2021,Attia2025} are designed for \textit{convex optimization}. As a comparison, our (A)MT-PDHG are designed for \textit{saddle point problems} and have convergence rates which depend on the \textit{averages instead of maximum} of the updating rates, making them robust against extreme delays in the dual updates.

\textbf{Non-smooth distributed optimization.} Since the seminal works \cite{BertsekasTsitsiklis1989,TsitsiklisBertsekasAthans1986}, numerous algorithms have been proposed for non-smooth distributed optimization under various settings, and we refer readers to surveys such as \cite{AssranAytekinFeyzmahdavianJohanssonRabbat2020,NedicOlshevskyRabbat2018,konecny2016federatedoptimizationdistributedmachine}. For the function class of Lipschitz, non-smooth, convex objectives, most of these algorithms fall into the following two categories: subgradient based and dual based \cite{lan_communication-efficient_2020}. Subgradient based algorithms such as the incremental gradient method \cite{bertsekas_incremental_2011}, decentralized subgradient method \cite{NedicOzdaglar2009}, and the dual averaging \cite{Duchi2012DA} usually require $O(1/\epsilon^2)$ rounds of communication, each followed by one gradient step. Within the function class, this achieves the optimal subgradient oracle complexity, but is suboptimal with respect to the communication rounds: as proven by \cite{ArjevaniShamir2015Communication,Scaman2019}, the communication rounds needed is $O(1/\epsilon)$. 

As a comparison, dual based algorithms, which dualize the consensus constraints, usually have better communication complexity: $O(1/\epsilon)$ rounds are needed for distributed ADMM \cite{Aybat2018,wei2013o1kconvergenceasynchronousdistributed} and the decentralized communication sliding (DCS) \cite{lan_communication-efficient_2020}, as examples. However, each round of communication is followed by optimization of Lagrangians or proximal updates, performed locally by each agent. To make the overall algorithm first-order, \cite{lan_communication-efficient_2020} proposes the Communication Sliding (CS) procedure, which approximates the proximal updates through $O(1/\epsilon)$ steps of (local) mirror descent, thereby achieving the $O(1/\epsilon^2)$ subgradient oracle complexity. The CS procedure has roots in  the gradient sliding technique \cite{lan_gradient_2016}, which can save gradient computation for the smooth component when the objective involves a smooth and a non-smooth component.

For the class of strongly convex objectives, DCS can be accelerated, needing $O(1/\sqrt{\epsilon})$ rounds of communication and $O(1/\epsilon)$ gradient steps in total, both achieving the theoretical optimal \cite{lan_communication-efficient_2020}. In this work, due to the different time scales, we generalize the CS procedure for problems involving a mixture of Bregman divergences. In addition, we point out that for problems with \textit{smooth} objectives, Local SGD -- which applies gradient steps locally but communicates only once in a while -- has been studied under various settings \cite{Zinkevich2010,stich2019local,Woodworth2020}.

\textbf{Lower bounds on communication.} In \cite{ArjevaniShamir2015Communication}, it is shown that for distributed convex optimization, $O(1/\epsilon)$ rounds of communication are needed for $1$-Lipschitz objectives, and $O(1/\sqrt{\mu\epsilon})$ rounds are needed when the objectives are also $\mu$-strongly convex. These lower bounds are achieved by splitting a ``chain like'' objective into two, each given to one agent. \cite{Scaman2019} extends these results to a decentralized, network setting and shows the dependence of the lower bounds on the network diameter and communication delay. \cite{Tyurin2024} provides lower bound and (nearly) optimal algorithm for a different setup, where distributed agents have stochastic first order oracles to the same smooth nonconvex objective, but computation and communication speeds are bounded and different for different edges and agents. Apart from the round complexity, \cite{Santosh2020} shows a dimension-dependent lower bound on the bit-complexity of communication.  

In addition, motivated by distributed training in machine learning, communication lower and upper bounds have been established using function similarities \cite{ArjevaniShamir2015Communication,pmlr-v32-shamir14,Zhu2016,pmlr-v151-tian22b,Kovalev2022,pmlr-v119-karimireddy20a,pmlr-v119-koloskova20a}: for instance, in (distributed) empirical risk minimization, the local loss functions have the same functional form but use different subsets of data, thereby inheriting the similarity in data. In \cite{ArjevaniShamir2015Communication}, function similarities are measured using norms of the differences in (sub)gradients (and Hessians if exist), and a communication round lower bound linear in this measure is shown for convex Lipschitz objectives and strongly convex objectives. Known algorithms that take advantage of function similarities usually require additional assumptions such as strong convexity and smoothness \cite{Zhu2016,pmlr-v32-shamir14,pmlr-v37-zhangb15,pmlr-v151-tian22b,Kovalev2022,pmlr-v119-karimireddy20a,pmlr-v119-koloskova20a}. As pointed out in \cite{ArjevaniShamir2015Communication}, there is no known algorithm which achieve these communication round lower bounds for non-smooth convex objectives. In this work, we formalize the notion of function similarity for non-smooth convex objectives (Definition \ref{def:func_simi}), and show that the communication round complexity for our (A)MT-PDHG indeed achieve these lower bounds, thereby answering \cite{ArjevaniShamir2015Communication}'s open question positively. 

\subsection{Contributions}
We propose multi-timescale PDHG for saddle point problems with block-decomposable duals, where different dual blocks are updated at different rates. 
\begin{itemize}
    \item To ensure convergence with arbitrary updating rates, we propose novel multi-timescale extrapolation steps for the dual updates \eqref{eq:novel-dual} and mixtures of Bregman divergence for the primal updates \eqref{eq:primal-md-obj}. The duality gaps of our algorithms converge at the optimal rates: $O(1/k)$ and $O(1/k^2)$ for general and strongly convex objectives, respectively, despite the potentially outdated information in the dual. 
    \item To the best of our knowledge, for the saddle point problems considered, our multi-timescale PDHG algorithms are the first \textit{deterministic} updating mechanisms which still \textit{allow different (arbitrary) updating rates for different dual blocks}.
    \item We quantify how updating rates influence convergence: the duality gaps show linear dependencies on \textit{weighted averages} of the updating rates (Corollary \ref{cor:main}) and its square (Corollary \ref{cor:main_acc}) for general and strongly convex objectives, respectively. Thus, our algorithms are \textit{robust to the maximum updating rates}.\footnote{We call the inverses of frequencies the rates. Thus, the largest updating rate corresponds to the smallest frequency.}
\end{itemize}

We apply our algorithms to convex non-smooth distributed optimization. More specifically, we propose relaxing the consensus constraints via \textit{generic convex, block-decomposable penalty functions} (Lemma \ref{lm:delta-subopt} and Corollary \ref{cor:delta_subopt}), which generalize the \textit{characteristic-function penalty} used in prior work. This yields saddle-point formulations \eqref{eq:objective_F} with block-decomposable dual variables, enabling our multi-timescale PDHG method, which is particularly well suited to settings with heterogeneous communication costs among agents. 

Moreover, we show that with proper choices of the penalties, the communication round complexities of our algorithms have linear, and thus
\textit{optimal}, dependency on similarities between gradients of the local objectives. This provides positive answers to the open question whether the theoretical communication round lower bounds proposed in \cite{ArjevaniShamir2015Communication} can be attained.

\subsection{Roadmap}

In Section \ref{sec:setup}, we present additional details and assumptions about the saddle point problem \eqref{eq:objective_pd_block}. Then, in Section \ref{sec:update-proc}, we propose our multi-timescale PDHG for \eqref{eq:objective_pd_block}, and in Sections \ref{sec:convergence-convex} and \ref{sec:convergence-acc}, we present the convergence properties of our algorithms with and without strong convexity, respectively.

In Section \ref{sec:app-do}, we apply our multi-timescale PDHG to distributed optimization problems: we first show that \eqref{eq:objective_f} can be reformulated in the lifted space in the form of \eqref{eq:objective_pd_block} (Section \ref{sec:do-setup}) and propose a set of conditions on the consensus constraints (Section \ref{sec:agents}) and penalties (Section \ref{sec:two_measures}). The requirements on the penalties motivate the definition of function similarities (Definition \ref{def:func_simi}). Then, we describe how multi-timescale PDHG can be applied and the communication involved (Section \ref{sec:do-pdhg}), and provide the convergence results in Sections \ref{sec:do-convergence} and \ref{sec:do-convergence-acc}.

Numerical experiments are provided in Section \ref{sec:exp}.

\section{Setup}\label{sec:setup}

For the problem \eqref{eq:objective_pd_block}, we assume $\R^d$ is equipped with a norm $\|\cdot\|$ not necessarily generated by the Euclidean inner product, and we equip $\mc X$ with a distance generating function\footnote{
For a convex closed set $\mc S$, a function $w:\mc S\to \R$ is a distance generating function \cite{Ghadimi2012} with modulus $\nu>0$ w.r.t. $\|\cdot\|$ if $w$ is continuously differentiable and 
\begin{displaymath}
    \langle x-z,\nabla w(x) -\nabla w(z)\rangle \geq \nu \|x-z\|^2,\quad \forall x,z\in \mc S.
\end{displaymath}
 The Bregman divergence generated by $w$ is defined as $D_w(x,z):=w(x) - w(z) - \langle \nabla w(z),x-z\rangle$. 
} $w_X:\mc X\to \R$ with modulus $1$.

We further denote $\mc Y_s:=\domain(R_s^*) = \{y_s\in \R^{n_s},~R_s^*(y_s)<\infty\}$ as the domain of $R_s^*$, and define $R:\R^n\to \overline{\R}$ as $R(Y) = \sum_{s=1}^S R_s(y_s)$. It's easy to see that $R^*(Y) = \sum_{s=1}^S R_s^*(y_s)$ and $\mc Y:=\domain(R^*) = \prod_{s=1}^S \domain(R_s^*)$. Similarly, for each $s\in [S]$, we assume that $\R^{n_s}$ is equipped with a norm $\|\cdot\|$ not necessarily generated by the Euclidean inner product, and we equip $\domain(R_s^*)$ with distance generating function $w_{y_s}:\domain(R^*_s)\to \R$ with modulus $1$, and $w_Y(Y): = \sum_{s=1}^S w_{y_s}(y_s)$. 

For convenience, we define $K:\R^d\to \R^n$ as $(KX)_s = K_sX$ for each $s\in [S]$. Thus, \eqref{eq:objective_pd_block} can be compactly written as
\begin{equation}\label{eq:compact}\tag{$\mc P_c$}
    \min_{X\in \mc X} \max_{Y\in \R^n} F(X) +\langle KX,Y\rangle  - R^*(Y).
\end{equation}
In addition, since $R_s$ is convex and lower-semicontinuous, $R_s = R_s^{**}$ (Theorem 11.1 \cite{rockafellarVariational}), and importantly, 
\begin{equation*} 
    R_s(K_sX) = \sup_{y_s\in \R^{n_s}} \langle K_sX,y_s\rangle - R_s^*(y_s).
\end{equation*}
Thus, \eqref{eq:objective_pd_block} is also equivalent to the following primal-only formulation:
\begin{equation}\label{eq:objective_pd_block-primal}\tag{$\mc P_{p}$}
        \min_{X\in {\mc X}}F(X) + \sum_{s=1}^S R_s(K_sX).
\end{equation}

In the rest of the work, when there is no confusion on the domain of the function, we abbreviate $w_{X},w_Y, w_{y_s}$ as $w$ and the associated Bregman divergence as $D$. We also make the following assumption. 
\begin{assumption}\label{assumption:bregman_easy}
For any $\overline{y}_s\in \domain(R^*_s)$, $g\in \R^{n_s}$ the following problem can be solved exactly:
\begin{displaymath}
    \min_{y_s\in \domain(R^*_s)} R_s^*(y_s) +\langle g,y_s\rangle+ D(y_s,\overline{y}_s).
\end{displaymath}
For any $g\in \R^d$, the following problem can be solved exactly:
\begin{displaymath}
    \min_{X\in \mc X} \langle g,X\rangle + w(X).
\end{displaymath}
\end{assumption}

\textbf{Performance measure.} We use the duality gap to measure the performance of $Z = (X,Y)$. More precisely, we define $G:\mc Z\times \mc Z\to\overline{\R}$ where $\mc Z = \mc X\times \R^n$ as
    \begin{equation}\label{eq:def_gap}
        G(X,Y;X',Y'):=  \langle KX,Y'\rangle + F(X) - R^*(Y')- \left\{\langle KX',Y\rangle + F(X') - R^*(Y)\right\},
\end{equation}
and we will provide upper bounds on $G(Z):=\sup_{Z'\in \mc Z} G(Z;Z')$ where $Z$ is the output of our algorithms.

\section{Multi-timescale primal-dual updates}\label{sec:mupdhg}
To solve the saddle point problem \eqref{eq:compact}, one of the most popular algorithms is the Primal-Dual Hybrid Gradient (PDHG) algorithm\cite{chambolle_first-order_2011,lan_communication-efficient_2020,Chambolle2018Stochastic}, which updates the primal variable $X$ and the dual variable $Y$ iteratively. At iteration $k=0,1,\ldots$,
\begin{subequations}\label{eq:pdhg}
\begin{align}
    \widetilde{X}^{k} & = X^{k-1} + \alpha_k(X^{k-1}-X^{k-2})\label{eq:pdhg-x-extra}\\
    Y^{k} &=\argmin_{Y\in \R^n} R^*(Y)+ \langle -K\widetilde{X}^{k},Y\rangle + \tau_k D(Y,Y^{k-1})\label{eq:pdhg-y}\\
    X^{k} &= \argmin_{X\in \mc X} F(X)+
    \langle K^*Y^k,X\rangle + \eta_k D(X,X^{k-1}),\label{eq:pdhg-x}
\end{align}
\end{subequations}
where $\tau_k,\eta_k>0$ are parameters depending on the operator norm of $K$, and $\alpha_k\in [0,1]$. The convergence rates of PDHG under various assumptions of the objective functions $F$ and $R$ have been well established \cite{chambolle_first-order_2011}. For instance, for general convex $F$ and $R$, the duality gap (of the ergodic mean) converges at the rate of $O(1/k)$; further, if $F$ is strongly convex, the accelerated rate of $O(1/k^2)$ can be achieved. 

\textbf{Approximation to \eqref{eq:pdhg-x} through gradient sliding and auxiliary primal sequence.} For a generic (convex, potentially non-linear) objective function $F$, it's sometimes unreasonable to assume that one can find the exact minimizer in the update \eqref{eq:pdhg-x}. Indeed, in Assumption \ref{assumption:bregman_easy}, we only assume that $\langle g,X\rangle + w(X)$ can be exactly minimized. Thus, \cite{lan_gradient_2016,lan_communication-efficient_2020} propose and use gradient sliding techniques, which approximate the minimizer through multiple iterations of mirror descent. In addition, to ensure the convergence of the overall PDHG algorithm with inexact updates, an auxiliary sequence $\widehat{X}^k$ is constructed, and the right hand side of \eqref{eq:pdhg-x-extra} is replaced with $X^{k-1} + \alpha_k(\widehat{X}^{k-1}-X^{k-2})$.

\textbf{Block-decomposable dual updates.} In \eqref{eq:objective_pd_block}, both $R^*(Y) = \sum_{s=1}^S R^*_s(y_s)$ and $D(Y,Y^{k-1}) = \sum_{s=1}^S D(y_s,y_s^{k-1})$ are block-decomposable. Thus, the dual update \eqref{eq:pdhg-y} is also block-decomposable: 
\begin{equation}\label{eq:pdhg-dual-block}
    y_s^{k} =\argmin_{y\in \R^{n_s}} R_s^*(y_s)+ \langle -K_s\widetilde{X}^{k},y_s\rangle + \tau_k D(y_s,y_s^{k-1}),\quad s=1,\ldots,S.
\end{equation}
This makes block-coordinate descent type of algorithms possible. As an example, S-PDHG \cite{Chambolle2018Stochastic} updates a random subset of the dual blocks at each iteration. The flexibility in choosing the sampling distribution allows one to control the frequency of updates of different blocks. However, due to the randomness, the $O(1/k)$ rate of convergence is shown only for the expected duality gap. 

To maintain the \textit{deterministic convergence guarantee} as well as the \textit{flexibility in choosing the number of updates applied to each dual block}, we propose a multi-timescale updating mechanism for \eqref{eq:objective_pd_block}, where different dual blocks are updated at potentially different rates. More precisely, denoting the global time using $k=0,1,\ldots$, then the dual block $y_s$ is updated only at iteration $k=0,r_s,2r_s,\ldots$ for some positive integer $r_s$ and remains fixed for all other iterations. Due to this multi-timescale mechanism, two challenges arise. 
\begin{enumerate}
    \item \underline{Information delay in $y_s^k$.} In the primal update \eqref{eq:pdhg-x-extra}, the term $K^*Y^k = \sum_{s=1}^S K^*_sy_s^k$ depends on $y_s^k$, yet $y_s^{k} = y_s^{\lfloor k/r_s\rfloor\times r_s}$ contains information only till iteration $\lfloor k/r_s\rfloor\times r_s\leq k$, which could be ``outdated''. To mitigate the negative effect of ``information delay'', we propose using \underline{\textit{mixtures of Bregman divergences}} in \eqref{eq:pdhg-x} to control how fast the primal sequence varies (\eqref{eq:primal-md-obj}). 
    \item \underline{Multi-timescale information aggregation of $X^k$.} The dual updates \eqref{eq:pdhg-y} and \eqref{eq:pdhg-dual-block} are defined for each global time $k$, and $\widetilde{X}^k$ depends on $X^{k-2}$ and $X^{k-1}$ only. In the multi-timescale setting, one can expect that $\widetilde{X}^k$ should depend on primal sequences over longer intervals (the length of which could depend on $r_s$), and could be different for different dual blocks. We propose \underline{\textit{multi-timescale extrapolation}} which aggregates $X^{(k-2r_s):(k-1)}$ for block $s$ (\eqref{eq:novel-dual}).  
\end{enumerate}

In the rest of this section, we state the exact updating procedure in Section \ref{sec:update-proc}, and provide convergence results in Sections \ref{sec:convergence-convex} and \ref{sec:convergence-acc} for general convex and strongly convex $F$ respectively.

\subsection{Multi-timescale updating procedure}\label{sec:update-proc}

In our multi-timescale updating procedure, we use $k=0,1,\ldots,N$ to denote the global time. Motivated by gradient sliding procedures in \cite{lan_communication-efficient_2020}, we also allow approximate minimizers to the primal updates, and keep track of the pair $(X^k, \widehat{X}^k)$, which is assumed to satisfy the condition \eqref{eq:primal-approx-prop}. This can be achieved using a generalized gradient sliding technique (Appendix \ref{sec:generalized-CS}).

 \textbf{Initialization.} We assume access to initializations $X^{init}\in \mc X$ and $y_s^{init}\in \mc Y_s$ for all $s\in [S]$, and we initialize $X^{k'} = \widehat{X}^{k'}=X^{init}$ and $y_s^{k'} = y_s^{init}$ for all $k'<0$. 

{\color{blue}\textbf{Dual updates.}} We associate each dual $y_s$ with a rate $r_s\in \N$ and a local time $i_s=0,1,\ldots,N_s-1$, such that $N+1 = r_sN_s$. For $k=0,1,\ldots,N$, $y_s$ remains dormant unless $k = r_si_s$ for some $i_s\in \{0,1,\ldots,N_s-1\}$, where $y^{i_s}_s$ is computed as follows (which replaces \eqref{eq:pdhg-x-extra} and \eqref{eq:pdhg-y}):
\begin{subequations}\label{eq:novel-dual}
    \begin{align}
        \widetilde{X}_{s}^{i_s} &=\alpha_{s,i_s}\left(\sum_{k'=r_si_s-r_s}^{r_si_s-1}\theta_{k'}(\widehat{X}^{k'} - X^{k'-r_s})\right) + \sum_{k'=r_si_s-r_s}^{r_si_s-1}\theta_{k'+r_s}X^{k'},\label{eq:update_x_tilde}\\
    y_s^{i_s} &= \argmin_{y_s\in \R^{n_s}} \langle - \frac{1}{\sum_{k'=r_si_s}^{r_si_s+r_s-1}\theta_{k'}}K_{s}\widetilde{X}_{s}^{i_s},y_s\rangle + R^*_s(y_s) + \tau_{s,i_s}D(y_s ,y^{i_s-1}_s).\label{eq:update_y}
    \end{align}
\end{subequations}

Further, we denote $\overline{y}_s^k=y_s^{\lfloor k/r_s\rfloor}$, i.e. value of the dual block $y_s$ at the global time $k$, and we abbreviate $\overline{Y}^k = (\overline{y}_s^k)_{s\in [S]}$.

{\color{red}\textbf{Primal updates.}} First, for each $k$, we define 
\begin{equation}\label{eq:primal-md-obj}
    \Phi^k(X):=F(X) + \langle K^* \overline{Y}^k,X\rangle + \sum_{s=1}^S \eta_{k,s}D(X,X^{k-r_s}).
\end{equation}
All primal variables are updated at each global time as \textit{approximate} minimizer to \eqref{eq:primal-md-obj}, such that there exists some $\delta_k:\mc X\to \R$ which could depend on $X^k,X^{k-1}$ and other algorithm parameters, the following is satisfied: 
\begin{equation}\label{eq:primal-approx-prop}
    \Phi^k(\widehat{X}^k) \leq \Phi^k(X) -(\frac{\mu}{C}+\sum_{s=1}^S\eta_{k,s})D(X,X^k)+\delta_k(X),\quad \forall X\in \mc X,
\end{equation}
where we assume \eqref{eq:prop-F} holds for some $\mu\geq 0$, and $D(X,X')\leq \frac{C}{2}\|X-X'\|^2$ for some $0<C\leq \infty$.

In case $\Phi^k$ can be minimized exactly, say $F(X) = \langle g_F,X\rangle$ for some $g_F\in \R^d$ is a linear function, we can take $\widehat{X}^k=X^k=\argmin_{X\in \mc X} \Phi^k(X)$ to be the exact minimizer, and then we can take $\delta_k(X) = 0$ (Lemma \ref{lm:generalized_lm5}). Nevertheless, in \eqref{eq:primal-approx-prop}, we also allow inexact minimizer. For generic convex objectives, this can be found through a generalization of the gradient sliding technique in \cite{lan_communication-efficient_2020}: we extend this technique from $S=1$ to $S\geq 1$ (see Appendix \ref{sec:generalized-CS} for more details). More concretely, following Corollary \ref{cor:lambda_beta}, the following results hold.
\begin{corollary}\label{cor:cs-prop-aux}
 Consider the following updates using $GS$, the generalized gradient sliding procedure in Algorithm \ref{alg:CS-procedure}, where $T_k\in \N$ and $\eta_k = \sum_{s=1}^S \eta_{k,s}>0$, where $\eta_{k,s}\geq 0$, 
\begin{equation}\label{eq:update-x-cs}
    (X^k,\widehat{X}^k) = GS(F,{\mc X},D,T_k,(\eta_{k,s})_{s\in [S]},(X^{k-r_s})_{s\in [S]},K^* \overline{Y}^k,X^{k-1}). 
\end{equation}

Assume that \eqref{eq:prop-F} holds with some $\mu\geq 0$, then with $\lambda_t = t+1$ and $\beta_t =\frac{t}{2}$ for $t\geq 1$, \eqref{eq:primal-approx-prop} holds with 
\begin{equation*}
    \delta_k(X) = \frac{2\eta_k}{T_k(T_k+3)}\left(D(X,X^{k-1})  
     - D(X,X^k) \right)+ \frac{4M^2}{\eta_k(T_k+3)}.
\end{equation*}

Further, if \eqref{eq:prop-F} holds with some $\mu> 0$, and $D(X,X')\leq \frac{C}{2}\|X-X'\|^2$ for some $C<\infty$, then with $\lambda_t = t$ and $\beta_t = \frac{(t+1)\mu}{2\eta_k C} +\frac{t-1}{2}$, \eqref{eq:primal-approx-prop} holds with 
\begin{equation*}
\delta_k(X) = \frac{2M^2/\eta_k}{T_k(T_k+1)}\sum_{t=1}^T\frac{\lambda_t}{\beta_t}.
\end{equation*}
\end{corollary}

\textbf{Final outputs.} Denoting $\widehat{Z}^k = (\widehat{X}^k,\overline{Y}^k)$, then the output is 
\begin{equation}\label{eq:ergodic}
    Z^N = (\sum_{k=0}^N\theta_k)^{-1}\sum_{k=0}^N \theta_k\widehat{Z}^k.
\end{equation}

\begin{figure}[h!]
\centering
\resizebox{.8\linewidth}{!}{
\begin{tikzpicture}[x=1cm,y=0.55cm,>=stealth,thick]

% colors
\definecolor{myblue}{RGB}{0,120,160}
\definecolor{myred}{RGB}{220,0,0}
\definecolor{myorange}{RGB}{240,140,60}
\definecolor{myyellow}{RGB}{255, 221, 54}
\definecolor{mygreen}{RGB}{80,170,60}

% horizontal arrows
\draw[myblue, line width=1.2pt, ->] (0,3.6) -- (10,3.6);
\draw[myblue, line width=1.2pt, ->] (0,2.4) -- (10,2.4);
\draw[myblue, line width=1.2pt, ->] (0,1.2) -- (10,1.2);
\draw[myblue, line width=1.2pt, ->] (0,0) -- (10,0);

% labels on the right
\node[right] at (10.5,3.6) {$(X^k,\widehat{X}^k),~k=0,\ldots,17$};
\node[right] at (10.5,2.4) {$y_1^{i_1},~r_1 = 3,~i_1 = 0,\ldots,5$};
\node[right] at (10.5,1.2) {$y_2^{i_1},~r_2 = 3,~i_2 = 0,\ldots,5$};
\node[right] at (10.5,0) {$y_3^{i_3},~r_3 = 6,~i_3 = 0,\ldots,2$};

\foreach \x in {0,...,17}
  {\draw[myred, only marks, mark=x, mark size=4pt] plot coordinates {(0.5*\x+0.6,3.6)};
  }

\foreach \x in {0,...,5}
  {\node[myorange,rectangle, draw,font=\tiny,fill=myorange!30,text = black ] at (1.5*\x+0.6,2.4) {$y_1^{\x}$};}

\foreach \x in {0,...,5}
  {\node[mygreen,rectangle, draw,font=\tiny,fill=mygreen!30,text = black ] at (1.5*\x+0.6,1.2) {$y_2^{\x}$};}

\foreach \x in {0,...,2}
{\node[myyellow,rectangle, draw,font=\tiny,fill=myyellow!30,text = black ] at (3*\x+0.6,0) {$y_3^{\x}$};}
\end{tikzpicture}
}
\caption{Updates for $S=3$, $r_1=r_2 = 3$ and $r_3 = 6$. Each marker represents one update: $(X^k, \widehat{X}^k)$ is updated at each global time $k=0,1,\ldots,17$. If generalized gradient sliding is used, then this involves $T_k$ iterations of mirror descent updates at iteration $k$. $y_1$ is updated at each local time $i_1 = 0,\ldots,5$, i.e. global time $k = 0,3,6,\ldots,15$, and similarly for $y_2$ and $y_3$. }
\end{figure}

\begin{algorithm}
\caption{(Accelerated) Multi-timescale PDHG}\label{alg:main-algo}
\begin{algorithmic}
\Require $\{\alpha_{s,i_s}\},\{\theta_k\},\{\eta_{k,s}\},\{\tau_{s,i_s}\},\{r_s\},X^{init},Y^{init}$
\Ensure Primal dual pair $Z^N$
\State Initialize $(X^{k'},\widehat{X}^{k'},Y^{k'})\gets (X^{init},X^{init},Y^{init})$ for all $k'< 0$
\For{$k = 0,1,\ldots,N$}\Comment{implicitly $i_s = \lfloor k/r_s\rfloor$ for all $s\in [S]$}
\For{$s\in [S]$ such that $k = 0 ~(\text{ mod } r_s)$} 
\State {\color{blue}{Dual update:}} compute $\widetilde{X}_{s}^{i_s}$ using \eqref{eq:update_x_tilde}, then update $y_s^{i_s}$ using \eqref{eq:update_y}
\EndFor
\State {\color{red}{Primal update:}} compute $K^*\overline{Y}^k$ where $\overline{y}_s^k = y_s^{\lfloor k/r_s\rfloor}$ for $s\in [S]$, update $(X^k,\widehat{X}^k)$ satisfying \eqref{eq:primal-approx-prop} 
\EndFor 
\State Compute $Z^N$ using \eqref{eq:ergodic}. 
\end{algorithmic}
\end{algorithm}

\subsection{Convergence of the multi-timescale updates}\label{sec:convergence-convex}
In this section, we assume $F$ is a generic convex function satisfying \eqref{eq:prop-F} for some $\mu\geq 0$. We propose conditions on the parameters of  Algorithm \ref{alg:main-algo} which ensure the convergence of the resulting multi-timescale updating procedure. 
The proof of convergence of Algorithm \ref{alg:main-algo} follows a similar type of argument as the proof of convergence of PDHG \cite{chambolle_first-order_2011} and the Decentralized Communication Sliding \cite{lan_communication-efficient_2020}: the primal updates \eqref{eq:update-x-cs} and dual updates control the following two terms (\eqref{eq:primal_bound}, \eqref{eq:dual_bound}): 
\begin{align*}
    &\left\{\sum_{k=0}^N \langle K^*\overline{Y}^k,\widehat{X}^k - X\rangle +F(\widehat{X}^k) - F(X)\right\} \\
    &\qquad+ \left\{ \sum_{i_s=0}^{N_s-1}\langle -K_{s}\widetilde{X}_s^{i_s},y_s^{i_s} - y_s\rangle + r_s(R^*_s(y_s^{i_s}) - R^*_s(y_s))\right\}. 
\end{align*}

The above sum (approximately) matches the gap $\sum_{k=0}^N G(\widehat{Z}^k;Z)$ up to an additive term (\eqref{eq:mix_terms})
\begin{displaymath}
    \sum_{s=1}^S \sum_{k=0}^{N}\langle \widehat{X}^{k}-\widetilde{X}_{s}^{\lfloor k/r_s\rfloor},K^*_{s}(y_s-\overline{y}_s^{k})\rangle = \sum_{i_s=0}^{N_s-1}\langle\sum_{i=0}^{r_s-1} \widehat{X}^{r_si_s+i}-\widetilde{X}_{s}^{i_s},K^*_{s}(y_s-y_s^{i_s})\rangle.
\end{displaymath}
Notice that due to the different timescales for the duals, we bound the above terms \textit{at dual time scales}: instead of controlling $\langle \widehat{X}^{k}-\widetilde{X}_{s}^{\lfloor k/r_s\rfloor},K^*_{s}(y_s-y_s^{\lfloor k/r_s\rfloor})\rangle$ for each $k$, we control the cumulative term (sum from $k=r_si_s$ to $k = r_s(i_1+1)-1$). With our choice of the $\widetilde{X}_s^{i_s}$ and the mixture terms used in primal proximal updates, the following results hold. 

\begin{theorem}\label{thm:main}
Assume that \eqref{eq:prop-F} holds with some $M,\mu \geq 0$, then with the following choice of parameters: $\alpha_{s,i_s}=\alpha=1$, $\theta_k = 1$; $\eta_{k,s} = \eta\rho_s$ where $\rho_s\geq 0$ and $\sum_{s=1}^S\rho_s = 1$; $\tau_s = \frac{2\widetilde{\kappa}_s^2}{\rho_s\eta}$ where $\widetilde{\kappa}_{s}:=\sup_{\|y_s\|\leq 1}\|K_{s}^*y_s\|_{*}$, denote $\overline{r} = \sum_{s=1}^S r_s\rho_s$, we have $\forall Z\in \mc Z = \mc X\times \R^n$, 
\begin{align*}
(N+1)\cdot G(Z^k;Z)&\leq \eta\overline{r}D(X,X^{init}) - \eta  D(X,X^N) +\sum_{k=0}^N \delta_k(X)\\
    &\quad +  \sum_{s=1}^S \tau_{s}r_s\cdot\left\{\frac{3}{2}D(y_s,y_s^{init}) - \frac{1}{2}D(y_s,y_s^{N_s-1})\right\}.
\end{align*}
\end{theorem}

\begin{proof}[Proof of Theorem \ref{thm:main}]
 
\textbf{Primal update properties.} Summing \eqref{eq:primal-approx-prop} over $k$, and defining $\eta_{k,s} = 0$ for all $k< 0$ and $k\geq N+1$, we have
\begin{align}\label{eq:primal_bound}
    &\quad \sum_{k=0}^N\langle K^* \overline{Y}^k,\widehat{X}^k-X\rangle + F(\widehat{X}^k) - F(X)\nonumber \\
    &\leq \sum_{k=0-\max\{r_s\}}^N \left( \sum_{s=1}^S\eta_{k+r_s,s} \right)D(X,X^{k})-\sum_{k=0}^N\sum_{s=1}^S  \eta_{k,s} D(\widehat{X}^k,X^{k-r_s})+ \sum_{k=0}^N \left(\delta_k(X)-\eta_kD(X,X^k)\right)\nonumber\\
    &\leq \eta \overline{r}D(X,X^{init})-  \eta  D(X,X^{N})-\sum_{k=0}^N\sum_{s=1}^S  \eta_{k,s} D(\widehat{X}^k,X^{k-r_s})+ \sum_{k=0}^N \delta_k(X),
\end{align}
where the last step is because $X^{k} =X^{init} $ for all $k< 0$, and by our choice that $\eta_{k,s} = \eta \rho_s$ for $k=0,1,\ldots,N$, 
\begin{displaymath}
    \sum_{k=0-\max\{r_s\}}^{-1} \sum_{s=1}^S\eta_{k+r_s,s}  \leq \eta \sum_{s=1}^S r_s\rho_s=  \eta \overline{r},\quad \sum_{s=1}^S\eta_{k'+r_s,s}  \leq \eta( \sum_{s=1}^S \rho_s) = \eta,~k'=0,\ldots,N-1.
\end{displaymath}

\textbf{Dual update properties.} By the updating rule for the dual, we have by Proposition 2 in \cite{lan_communication-efficient_2020} for any $y_s\in \R^{n_s}$, 
\begin{equation*}
     \langle -\frac{1}{r_s}K_{s}\widetilde{X}_{s}^{i_s},y_s^{i_s}-y_s\rangle + R^*_s(y_s^{i_s})  - R^*_s(y_s) \leq \tau_{s,i_s} \left(D(y_s,y_s^{i_s-1}) - D(y_s,y_s^{i_s})-D(y^{i_s}_s,y_s^{i_s-1})\right).
\end{equation*}
Thus, with $\tau_{s,i_s} = \tau_s$ for all $i_s$, summing over the above, we get 
\begin{align}\label{eq:dual_bound}
     &\quad \sum_{i_s=0}^{N_s-1}\langle -\frac{1}{r_s}K_{s}\widetilde{X}_{s}^{i_s},y_s^{i_s}-y_s\rangle + R^*_s(y_s^{i_s})  - R^*_s(y_s)\nonumber\\
     &\leq \tau_{s} \left (D(y_s,y_s^{init}) - D_{w^y_s}(y_s,y_s^{N_s-1})\right) - \tau_{s} \cdot\sum_{i_s=0}^{N_s-1} D(y^{i_s}_s,y_s^{i_s-1}).
\end{align}

\textbf{Gap properties.} Recall that for each $s\in [S]$, $\overline{y}_s^k = y_s^{\lfloor k/r_s\rfloor}$, thus we have
\begin{align}\label{eq:mix_terms}
 \sum_{k=0}^{N}\left\{\langle \widehat{X}^{k},K_s^*y_s\rangle - \langle X,K^*_s\overline{y}_s^{k}\rangle\right\}&= \sum_{i_s=0}^{N_s-1}\left\{\sum_{i=0}^{r_s-1}\langle \widehat{X}^{r_si_s+i},K^*_{s}y_s\rangle - r_s\langle X,K^*_{s}y_s^{i_s}\rangle\right\}\nonumber\\
    &= \sum_{i_s=0}^{N_s-1}\langle\sum_{i=0}^{r_s-1} \widehat{X}^{r_si_s+i}-\widetilde{X}_{s}^{i_s},K^*_{s}(y_s-y_s^{i_s})\rangle \nonumber \\
    &\quad +\sum_{k=0}^N \langle \widehat{X}^{k}-X,K^*_{s}\overline{y}_s^{k}\rangle +\sum_{i_s=0}^{N_s-1} \langle K_{s}\widetilde{X}_{s}^{i_s}, y_s- y_s^{i_s}\rangle .
\end{align}
Recall that for $i_s=0,1,\ldots,N_s-1$,
\begin{displaymath}
    \widetilde{X}_{s}^{i_s} =\alpha(\sum_{k'=r_si_s-r_s}^{r_si_s-1}\widehat{X}^{k'} - \sum_{k'=r_si_s-2r_s}^{r_si_s-r_s-1}{X}^{k'}) + \sum_{k'=r_si_s-r_s}^{r_si_s-1}X^{k'}.
\end{displaymath}

We first bound the first term in \eqref{eq:mix_terms}. Notice that for $i_s=0,1,\ldots,N_s-1$, we have
\begin{align*}
    &\quad \langle\sum_{i=0}^{r_s-1} \widehat{X}^{r_si_s+i}-\widetilde{X}_{s}^{i_s},K^*_{s}(y_s-y_s^{i_s})\rangle\\
    & = \langle \sum_{i=0}^{r_s-1}(\widehat{X}^{r_si_s+i}-X^{r_s(i_s-1)+i}) - \alpha\sum_{i=0}^{r_s-1} (\widehat{X}^{r_s(i_s-1)+i}-X^{r_s(i_s-2)+i}),K^*_{s}(y_s-y_s^{i_s})\rangle\\
    & = \langle \sum_{i=0}^{r_s-1}(\widehat{X}^{r_si_s+i}-X^{r_s(i_s-1)+i}) ,K^*_{s}(y_s-y_s^{i_s})\rangle\\
    &\quad -\alpha \langle \sum_{i=0}^{r_s-1} (\widehat{X}^{r_s(i_s-1)+i}-X^{r_s(i_s-2)+i}),K^*_{s}(y_s-y_s^{i_s-1})\rangle\\
    & \quad + \alpha \langle \sum_{i=0}^{r_s-1} (\widehat{X}^{r_s(i_s-1)+i}-X^{r_s(i_s-2)+i}),K^*_{s}(y_s^{i_s}-y_s^{i_s-1})\rangle.
\end{align*}

Thus, with $\alpha = 1$, and recall that for $i_s=0$, $i=0,\ldots,r_s-1$, $\widehat{X}^{r_s(i_s-1)+i}-X^{r_s(i_s-2)+i} = X^{init} - X^{init} = \mb 0$, we have
\begin{align*}
   &\quad \sum_{i_s=0}^{N_s-1}\langle\sum_{i=0}^{r_s-1} \widehat{X}^{r_si_s+i}-\widetilde{X}_{s}^{i_s},K^*_{s,v}(y_s-y_s^{i_s})\rangle \\
   &= \langle \sum_{i=0}^{r_s-1}(\widehat{X}^{N-r_s+i}-X^{N-2r_s+i}) ,K^*_{s}(y_s-y_s^{N_s-1})\rangle\\
   &\quad +  \sum_{i_s=1}^{N_s-1}\langle \sum_{i=0}^{r_s-1} (\widehat{X}^{r_s(i_s-1)+i}-X^{r_s(i_s-2)+i}),K^*_{s}(y_s^{i_s}-y_s^{i_s-1})\rangle\\
   &\leq \sum_{i=0}^{r_s-1} \|\widehat{X}^{r_s(N_s-1)+i}-X^{r_s(N_s-2)+i}\|\cdot \| K^*_{s}(y_s-y_s^{N_s-1})\|_{*} \\
   &\quad +  \sum_{i_s=1}^{N_s-1}\sum_{i=0}^{r_s-1}   \|\widehat{X}^{r_s(i_s-1)+i}-X^{r_s(i_s-2)+i}\|\cdot \|K^*_{s}(y_s^{i_s}-y_s^{i_s-1})\|_{*}
\end{align*}

Thus, for any $\rho>0$, we have 
\begin{align}\label{eq:mix_term1}
    &\quad \sum_{i_s=0}^{N_s-1}\langle\sum_{i=0}^{r_s-1} \widehat{X}^{r_si_s+i}-\widetilde{X}_{s}^{i_s},K^*_{s,v}(y_s-y_s^{i_s})\rangle \nonumber\\
    &\leq \sum_{k=0}^{N} \frac{\rho}{2}\|\widehat{X}^{k}-X^{k-r_s}\|^2 + \frac{r_s\widetilde{\kappa}^2_{s}}{2\rho}(\sum_{i_s=1}^{N_s-1}\|y_s^{i_s}-y_s^{i_s-1}\|^2 +  \|y_s-y_s^{N_s-1}\|^2).
\end{align}

\textbf{Bounding the gap.}
Thus, with \eqref{eq:primal_bound}, \eqref{eq:dual_bound}, and \eqref{eq:mix_term1}, we have the following upper bound on the gap 
\begin{align*}
    &\quad \sum_{k=0}^N G(\widehat{X}^k,\overline{Y}^k;Z)\\
    &=\sum_{k=0}^N \left\{F(\widehat{X}^k) -R^*(Y)- F(X) + R^*(\overline{Y}^k)\right\} + \sum_{s=1}^S \sum_{k=0}^N\left\{\langle K_{s}\widehat{X}^k,Y_s\rangle - \langle K_sX,\overline{Y}^k\rangle\right\}\\
     &\leq \eta \overline{r}D(X,X^{init})- \eta  D(X,X^{N})-\sum_{k=0}^N\sum_{s=1}^S  \eta_{k,s} D(\widehat{X}^k,X^{k-r_s})+ \sum_{k=0}^N \delta_k(X)\\
    &\quad +  \sum_{s=1}^S \tau_{s}r_s\left\{D(y_s,y_s^{init}) - D(y_s,y_s^{N_s-1})-\sum_{i_s=0}^{N_s-1} D(y^{i_s}_s,y_s^{i_s-1})\right\}\\
    &\quad + \sum_{k=0}^{N} \sum_{s=1}^S\frac{\eta\rho_s}{2}\|\widehat{X}^{k}-X^{k-r_s}\|^2 + \sum_{s=1}^S\frac{r_s\widetilde{\kappa}^2_{s}}{2\eta\rho_s}(\sum_{i_s=1}^{N_s-1}\|y_s^{i_s}-y_s^{i_s-1}\|^2 +  \|y_s-y_s^{N_s-1}\|^2)
\end{align*}
where we take $\rho = \eta \rho_s$ in \eqref{eq:dual_bound}. Thus, with $\frac{\widetilde{\kappa}^2_{s}}{\rho_s\tau_s}\leq \frac{\eta}{2}$ for all $s\in [S]$, we have
\begin{align*}
    \sum_{k=0}^{N} G(\widehat{X}^k,\overline{Y}^k;Z)&\leq \eta\overline{r}D(X,X^{init}) - \eta  D(X,X^{N}) +\sum_{k=0}^N \delta_k(X)\\
    &\quad +  \sum_{s=1}^S \tau_{s}r_s\cdot\left\{\frac{3}{2}D(y_s,y_s^{init}) - \frac{1}{2}D(y_s,y_s^{N_s-1})\right\}.
\end{align*}

The result follows since $Z^N $ is the ergodic mean of $(\widehat{X}^k,\overline{Y}^k)$ and $F,R$ are convex.
\end{proof}

When the primal update \eqref{eq:pdhg-x} is approximated through the generalized communication sliding procedure, from Corollary \ref{cor:cs-prop-aux}, we have the following results. 
\begin{corollary}\label{cor:main-cs}
Under the conditions in Theorem \ref{thm:main}, and assume that $(X^k,\widehat{X}^k)$ are constructed using the generalized communication sliding \eqref{eq:update-x-cs} with $\lambda_t = t+1$, $\beta_t = t/2$, and $T_k = T\geq 1$. Then the following holds for all $Z\in \mc Z$ 
\begin{align*}
    &\quad (N+1) \cdot G(Z^N;Z)\\
     &\leq \eta\left\{\frac{3}{2}\overline{r}D(X,X^{init}) -  D(X,X^{N})\right\}\\
    &\quad + \frac{1}{\eta}\left\{ \sum_{s=1}^S \frac{\widetilde{\kappa}_s^2r_s}{\rho_s}\left\{3D(y_s,y_s^{init}) - D(y_s,y_s^{N_s-1})\right\} + \frac{4M^2(N+1)}{T+3}\right\}
\end{align*}
    
\end{corollary}

\begin{proof}[Proof of Corollary \ref{cor:main-cs}]
Recall that from Corollary \ref{cor:cs-prop-aux}, we have for $k=0,1,\ldots,N$, with $\eta_k=\eta$ and $T_k = T$
\begin{equation*}
    \delta_k(X) = \frac{2\eta}{T(T+3)}\left(D(X,X^{k-1})  
     - D(X,X^k)\right) + \frac{4M^2}{\eta(T+3)}.
\end{equation*}
Thus, summing over $k$, we have 
\begin{equation*}
     \sum_{k=0}^N\delta_k(X)=  \frac{2\eta}{T(T+3)}\left(D(X,X^{init})  
     - D(X,X^N)\right) + \frac{4M^2(N+1)}{\eta(T+3)}
\end{equation*}
The result follows from noticing that for any $T\geq 1$, $\frac{2}{T(T+3)}\leq \frac{1}{2} \leq \frac{\overline{r}}{2}$.

\end{proof}

\begin{corollary}\label{cor:main} 
For $\widehat{X}\in {\mc X}$, assume that the following are finite:
\begin{displaymath}
  D(\widehat{X},X^{init})\leq D^X<\infty,\quad  \sup_{y_s\in \domain(R^*_s)}D (y_s,y_s^{init})\leq D_s^y<\infty. 
\end{displaymath}
Under the conditions in Corollary \ref{cor:main-cs}, taking $\eta = (\sum_{s'=1}^S \widetilde{\kappa}_{s'}\sqrt{D_{s'}^y})\sqrt{\frac{8}{3D^X}}$, $\rho_s = \frac{\widetilde{\kappa}_s\sqrt{D_s^y}}{\sum_{s'=1}^S \widetilde{\kappa}_{s'}\sqrt{D_{s'}^y}}$, and $T \geq \lfloor \frac{4M^2(N+1)}{\overline{r}(\sum_{s=1}^S \widetilde{\kappa}_{s}\sqrt{D_{s}^y})^2}\rfloor$ where $\overline{r}:=\sum_{s=1}^S r_s\rho_s$, we have
\begin{equation}\label{eq:main3}
    \sup_{Y'\in \R^n}G(Z^N;\widehat{X},Y') \leq \frac{2\sqrt{6}\overline{r}\cdot (\sum_{s=1}^S \widetilde{\kappa}_{s}\sqrt{D_{s}^y})\cdot \sqrt{D^X}}{N+1}. 
\end{equation}

\end{corollary}
\begin{proof}[Proof of Corollary \ref{cor:main}]
From Corollary \ref{cor:main-cs}, we first notice that with $\rho_s = \frac{\widetilde{\kappa}_s\sqrt{D_s^y}}{\sum_{s'=1}^S \widetilde{\kappa}_{s'}\sqrt{D_{s'}^y}}$ 
\begin{displaymath}
    \sum_{s=1}^S \frac{\widetilde{\kappa}_s^2r_sD_s^y}{\rho_s} = \overline{r}(\sum_{s'=1}^S \widetilde{\kappa}_{s'}\sqrt{D_{s'}^y})^2. 
\end{displaymath}
Thus, we have
\begin{displaymath}
    T +3 \geq \frac{4M^2(N+1)}{\overline{r}(\sum_{s=1}^S \widetilde{\kappa}_{s}\sqrt{D_{s}^y})^2} \implies \frac{4M^2(N+1)}{T+3} \leq \sum_{s=1}^S \frac{\widetilde{\kappa}_s^2r_sD_s^y}{\rho_s}
\end{displaymath}
Thus, with the additional assumptions, we get 
\begin{align*}
    \sup_{Y'\in  \R^n}G(Z^N;\widehat{X},Y')&
     \leq (N+1)^{-1}\left\{\frac{3\eta D^X}{2}(\sum_{s=1}^S r_s\rho_s) + \frac{4}{\eta}(\sum_{s=1}^S \frac{\widetilde{\kappa}_s^2r_sD_s^y}{\rho_s})\right\}\\
     &= (N+1)^{-1}\left\{\frac{3\eta D^X}{2}\cdot \overline{r} + \frac{4}{\eta}\overline{r}(\sum_{s'=1}^S \widetilde{\kappa}_{s'}\sqrt{D_{s'}^y})^2\right\}\\
     & = \frac{2\sqrt{6}\overline{r}(\sum_{s'=1}^S \widetilde{\kappa}_{s'}\sqrt{D_{s'}^y})\cdot \sqrt{D^X}}{N+1}.
\end{align*}

\end{proof}

\textbf{Discussion on the complexities.} Corollary \ref{cor:main} implies that to find an $\epsilon$-suboptimal solution, one can take 
\begin{displaymath}
    N = O(\frac{\overline{r}\cdot (\sum_{s=1}^S \widetilde{\kappa}_{s}\sqrt{D_{s}^y})\cdot \sqrt{D^X}}{\epsilon}),\quad T = O(\frac{M^2\sqrt{D^X}}{\epsilon\sum_{s=1}^S \widetilde{\kappa}_{s}\sqrt{D_{s}^y}}),
\end{displaymath}
making the total number of subgradient oracles to $F$  
\begin{displaymath}
    NT = O(\frac{M^2N^2}{\overline{r}(\sum_{s'=1}^S \widetilde{\kappa}_{s'}\sqrt{D_{s'}^y})^2})=O(\frac{\overline{r}M^2D^X}{\epsilon^2}).
\end{displaymath}
This agrees with \cite{lan_communication-efficient_2020} for the case when $r_s = 1$ for all $s$.

\textbf{Discussion on the costs.} 
To further illustrate the benefits of having different update frequencies for different duals, we analyze the ``cost'' of Algorithm \ref{alg:main-algo}. Precisely, we assume that the cost of one update to $y_s$ is $c_s$. This can be used to model the computation costs of matrix-vector multiplications ($K_s\widetilde{X}_s^{i_s}$ in \eqref{eq:update_x_tilde} and updating the $K_s^*\overline{y}_s^k$ term in \eqref{eq:primal-md-obj}). As another motivation, as will be seen in Section \ref{sec:app-do} and Algorithm \ref{alg:main-algo-distributed}, when applying our multi-timescale PDHG to distributed optimization problems, only one round of communication is needed for each update of $y_s$. Thus, the costs here can also represent the communication costs for every dual update. 

Below, we consider the case when the total cost is additive. With the above $N$, the dual variable $y_s$ is updated $O(\frac{\overline{r}}{r_s\epsilon})$ times, which is different for duals with different $r_s$. Thus, suppose one is allowed to choose the update frequencies $\{r_s\}_{s\in [S]}$, to minimize the total cost to find an $\epsilon$ suboptimal solution, the following should be (approximately) minimized
\begin{displaymath}
    O(\overline{r}\sum_{s=1}^S \frac{c_s}{r_s}) = O((\sum_{s=1}^S\rho_sr_s)\cdot(\sum_{s=1}^S \frac{c_s}{r_s})).
\end{displaymath}
With $r_s \propto \sqrt{c_s/\rho_s}$\footnote{Here and below, $\propto$ means (approximately) proportional to, i.e. there exists $r_0\in \R$ such that $r_s \approx r_0\sqrt{c_s/\rho_s}$ for all $s\in [S]$.}, the above becomes $O((\sum_{s=1}^S \sqrt{c_s\rho_s})^2)$. As a comparison, the strategy where all $r_s=r'_0$ are the same has the cost $O(\sum_{s=1}^Sc_s)$. By Cauchy–Schwarz inequality, $(\sum_{s=1}^S \sqrt{c_s\rho_s})^2\leq \sum_{s=1}^Sc_s$, and the difference can be very large when $\{c_s\rho_s\}_{s\in [S]}$ are very different, thereby showing the benefit of optimizing the updating rates $\{r_s\}_{s\in [S]}$ when $\{c_s/(\widetilde{\kappa}_s\sqrt{D^y_s})\}_{s\in [S]}$ are heterogeneous.

The additive cost is motivated by resources consumption when sending messages along each edge. In general, the total cost can be an arbitrary set function of the set of duals updated. For instance, to model time required to send messages (in parallel) where total time depends on the largest time, the cost could be $\max_{s\in \mc S}c_s$. Thus, the flexibility in choosing the updating rates allows the algorithm users to adapt the rates to the cost structures, leading to potentially lower costs.

\subsection{Accelerated convergence under strong convexity}\label{sec:convergence-acc}
With strong convexity of $F$, i.e. $\mu>0$, the convergence rate can be improved from $1/N$ to $1/N^2$, with a different set of parameters. In the following, we present these results.

\begin{theorem}\label{thm:acc_convergence}
Assume that $\eqref{eq:prop-F}$ holds with some $M,\mu>0$. Further assume that $D(X,X')\leq \frac{C}{2}\|X-X'\|^2$ for all $X,X'\in \mc X$ for some $1\leq C<\infty$. Let $\{\rho_s\}_{s\in [S]}$ be a distribution over $[S]$, $\overline{r} = \sum_{s=1}^S r_s\rho_s$ and similarly define $\overline{r^2}$ and $\overline{r^3}$. 
    
    With $\alpha_{s,i_s} = 1$, $\theta_k = k+ 2\overline{r^2}/\overline{r}$, $\eta_k = \frac{\mu }{2\overline{r}C}(k+\overline{r^2}/\overline{r})$, $\eta_{k,s} = \eta_k\rho_s$, $\tau_{s,i_s}(\sum_{k'=r_si_s}^{r_si_s+r_s-1}\theta_{k'}) = \tau_s= \frac{\widetilde{\kappa}_s^2}{\rho_s}\cdot\frac{4r_s\overline{r^2}C}{\mu}$. We have $\forall Z\in \mc Z=\mc X\times \R^n$, 
\begin{equation*}
     G(Z^N;Z) \leq \frac{2}{N(N+1)}\biggl\{\frac{5(\overline{r^2}/\overline{r})^2)}{2C}D(X,X^{init})+\sum_{k=0}^N\theta_k\delta_k(X)+ \sum_{s=1}^S\tau_{s}D(y_s,y_s^{init})  \biggr\}.
\end{equation*}

\end{theorem}

\begin{proof}[Proof of Theorem \ref{thm:acc_convergence}]
\textbf{Primal update properties.} Taking a weighted sum of \eqref{eq:primal-approx-prop} over $k$, and defining $\eta_{k,s} = 0$ for all $k< 0$ and $k\geq N+1$, we have
\begin{align*}
    &\quad \sum_{k=0}^N\theta_k\left\{\langle K^* \overline{Y}^k,\widehat{X}^k-X\rangle + F(\widehat{X}^k) - F(X) \right\}\\
    &\leq \sum_{k=0}^N\theta_k\left(\sum_{s=1}^S  \eta_{k,s} \left(D(X,X^{k-r_s})-D(\widehat{X}^k,X^{k-r_s})\right)-(\frac{\mu}{C}+\eta_k)D(X,X^k)+\delta_k(X)\right)\\
    &\leq \frac{5(\overline{r^2}/\overline{r})^2)}{2C}D(X,X^{init})-\sum_{k=0}^N\theta_k\sum_{s=1}^S  \eta_{k,s} D(\widehat{X}^k,X^{k-r_s})+\sum_{k=0}^N\theta_k\delta_k(X),
\end{align*}
where the last step is because the coefficients of the term $D(X,X^k)$ (denoting $\theta_k = \eta_{k} = 0$ for all $k\geq N+1$) for $k=0,1,\ldots,N$ is the following
\begin{align*}
    &\quad \sum_{s= 1}^S \theta_{k+r_s}\eta_{k+r_s,s} - (\frac{\mu}{C}+\eta_k)\theta_k\\
    &\leq \frac{\mu}{2\overline{r}C}\left\{\sum_{s= 1}^S (k+r_s+2\overline{r^2}/\overline{r})(k+r_s+\overline{r^2}/\overline{r})\rho_s- (k+\overline{r^2}/\overline{r}+2\overline{r})(k+2\overline{r^2}/\overline{r})\right\}\\
    & = \frac{\mu}{2\overline{r}C}\left\{ \left( k^2 + (2\overline{r} + \frac{3\overline{r^2}}{\overline{r}})k + (4\overline{r^2} + 2(\overline{r^2}/\overline{r})^2) \right)-\left(k^2 + (2\overline{r} + \frac{3\overline{r^2}}{\overline{r}})k +  (4\overline{r^2} + 2(\overline{r^2}/\overline{r})^2)\right)\right\}= 0,
\end{align*}
Since $X^{k} = X^{init}$ for all $k<0$, the coefficient for the term $D(X,X^{init})$ is 
\begin{displaymath}
    \sum_{s=1}^S \sum_{k=0}^{r_s-1} \eta_{k,s}\theta_{k}  \leq  \frac{\mu }{2\overline{r}C}\sum_{s=1}^S \rho_s\cdot r_s(r_s +\frac{\overline{r^2}}{\overline{r}})(r_s +2\frac{\overline{r^2}}{\overline{r}})=\frac{\mu(\overline{r^3}/\overline{r} + 5(\overline{r^2}/\overline{r})^2)}{2C}. 
\end{displaymath}

\textbf{Dual update properties.} Similar to \eqref{eq:dual_bound}, we get 
\begin{align}\label{eq:dual_bound2}
     &\quad \sum_{i_s=0}^{N_s-1}\left\{\langle -K_s\widetilde{X}_{s}^{i_s},y_s^{i_s}-y_s\rangle + (\sum_{k'=r_si_s}^{r_si_s+r_s-1}\theta_k')(R^*_s(y_s^{i_s})  - R^*_s(y_s))\right\}\nonumber\\
     &\leq \sum_{i_s=0}^{N_s-1} \tau_{s,i_s}(\sum_{k'=r_si_s}^{r_si_s+r_s-1}\theta_k')\left\{D(y_s,y_s^{i_s-1}) - D(y_s,y_s^{i_s})-D(y^{i_s}_s,y_s^{i_s-1})\right\}\nonumber\\
     &=  \tau_{s}\left\{D(y_s,y_s^{init}) - D(y_s,y_s^{N_s-1})-\sum_{i_s=0}^{N_s-1} D(y^{i_s}_s,y_s^{i_s-1})\right\}.
\end{align}

\textbf{Gap properties.} Notice that for each $s\in [S]$, we have
\begin{align}\label{eq:mix_terms2}
&\quad \sum_{k=0}^{N}\theta_k\left\{\langle \widehat{X}^{k},K^*_{s}y_s\rangle - \langle X,K^*_{s}\overline{y}_s^{k}\rangle\right\}\nonumber\\
    &= \sum_{i_s=0}^{N_s-1}\langle\sum_{i=0}^{r_s-1} \theta_{r_si_s+i}\widehat{X}^{r_si_s+i}-\widetilde{X}_{s}^{i_s},K^*_{s}(y_s-y_s^{i_s})\rangle  \nonumber\\
    &\quad +\sum_{k=0}^N \theta_k \langle \widehat{X}^{k}-X,K^*_{s}\overline{y}_s^{k}\rangle +\sum_{i_s=0}^{N_s-1} \langle K_{s}\widetilde{x}_{s}^{i_s}, y_s- y_s^{i_s}\rangle .
\end{align}
We first bound the first term in \eqref{eq:mix_terms}. With $\alpha_{s,i_s} = 1$ 
\begin{align*}
   &\quad \sum_{i_s=0}^{N_s-1}\langle\sum_{i=0}^{r_s-1} \theta_{r_si_s+i}\widehat{X}^{r_si_s+i}-\widetilde{X}_{s}^{i_s},K^*_{s,v}(y_s-y_s^{i_s})\rangle \\
   &= \langle \sum_{i=0}^{r_s-1}(\theta_{N-r_s+i}(\widehat{X}^{N-r_s+i}-X^{N-2r_s+i})) ,K^*_{s}(y_s-y_s^{N_s-1})\rangle\\
   &\quad +  \sum_{i_s=1}^{N_s-1}\langle \sum_{i=0}^{r_s-1} (\theta_{r_s(i_s-1)+i}(\widehat{X}^{r_s(i_s-1)+i}-X^{r_s(i_s-2)+i})),K^*_{s}(y_s^{i_s}-y_s^{i_s-1})\rangle\\
   &\leq \sum_{i=0}^{r_s-1} \theta_{r_s(N_s-1)+i}\|\widehat{X}^{r_s(N_s-1)+i}-X^{r_s(N_s-2)+i}\|\cdot \| K^*_{s}(y_s-y_s^{N_s-1})\|_{*} \\
   &\quad +  \sum_{i_s=1}^{N_s-1}\sum_{i=0}^{r_s-1}   \theta_{r_s(i_s-1)+i}\|\widehat{X}^{r_s(i_s-1)+i}-X^{r_s(i_s-2)+i}\|\cdot \|K^*_{s}(y_s^{i_s}-y_s^{i_s-1})\|_{*}
\end{align*}

\textbf{Bounding the gap.}
Putting the above together, and for convenience, denoting $y_s^{N_s} = y_s$, we have
\begin{align*}
    &\quad \sum_{k=0}^N \theta_kG(\widehat{X}^k,\overline{Y}^k;Z)\\
     &\leq \frac{5(\overline{r^2}/\overline{r})^2)}{2C}D(X,X^{init})-\sum_{k=0}^N\theta_k\sum_{s=1}^S  \eta_{k,s} D(\widehat{X}^k,X^{k-r_s})+\sum_{k=0}^N\theta_k\delta_k(X)\\
     &\quad + \sum_{s=1}^S\tau_s\left\{D(y_s,y_s^{init}) - D(y_s,y_s^{N_s-1})-\sum_{i_s=1}^{N_s-1} D(y^{i_s}_s,y_s^{i_s-1})\right\}\\
     &\quad + \sum_{s=1}^S \left\{\sum_{k=0}^N\frac{\theta_k\eta_{k,s}}{2}\|\widehat{X}^k - X^{k-r_s}\|^2 +\sum_{i_s=1}^{N_s}\frac{\widetilde{\kappa}_s^2}{2\rho_s}(\sum_{i=0}^{r_s-1}\frac{\theta_{r_s(i_s-1)+i}}{\eta_{r_s(i_s-1)+i}})\|y_s^{i_s} - y_s^{i_s-1}\|^2\right\}\\
    &\leq \frac{5(\overline{r^2}/\overline{r})^2)}{2C}D(X,X^{init})+\sum_{k=0}^N\theta_k\delta_k(X)+ \sum_{s=1}^S\tau_{s}D(y_s,y_s^{init}) 
\end{align*}
using 
\begin{displaymath}
    \tau_s = \frac{\widetilde{\kappa}_s^2}{\rho_s}\cdot\frac{4r_s\overline{r}C}{\mu}\geq \frac{\widetilde{\kappa}_s^2}{\rho_s}\cdot \max_{i_s\in [N_s]}(\sum_{i=0}^{r_s-1}\frac{\theta_{r_s(i_s-1)+i}}{\eta_{r_s(i_s-1)+i}}). 
\end{displaymath}

The result then follows from $\sum_{k=0}^N \theta_k\geq \frac{N(N+1)}{2}$ and convexity of $F$ and $R$.

\end{proof}

\begin{corollary}\label{cor:acc_convergence}
    Under the conditions in Theorem \ref{thm:acc_convergence}, and assume that $(X^k,\widehat{X}^k)$ are constructed using the generalized communication sliding \eqref{eq:update-x-cs} with $T_k/N = T/N\geq \max(\frac{5}{\sqrt{D_1}},\frac{64\overline{r}}{D_1})$ where $D_1 = \frac{\mu^2(\overline{r^2}/\overline{r})^2}{2M^2C^2}D_0$, $\lambda_t = t$ and $\beta^k_t = \frac{(t+1)\mu}{2\eta_kC} + \frac{t-1}{2}$ for $t = 1,\ldots,T_k$. Then the following holds for all $Z\in \mc Z$, 
\begin{align*}
     G(Z^N;Z) \leq \frac{2}{N(N+1)}\biggl\{&\frac{\mu(\overline{r^3}/\overline{r} + 5(\overline{r^2}/\overline{r})^2)}{2C}D(X,X^{init})\\
     &+ \frac{\mu (\overline{r^2}/\overline{r})^2}{C}D_0 + \frac{4\overline{r}C}{\mu}\sum_{s=1}^S\frac{\widetilde{\kappa}_s^2r_s}{\rho_s}D(y_s,y_s^{init}) \biggr\}.
\end{align*}
\end{corollary}

\begin{proof}[Proof of Corollary \ref{cor:acc_convergence}]
Recall that from Corollary \ref{cor:cs-prop-aux}, the following holds:
\begin{equation*}
\delta_k(X) = \frac{2M^2/\eta_k}{T_k(T_k+1)}\sum_{t=1}^{T_k} \frac{\lambda_t}{\beta_t}.
\end{equation*}
Thus, the result follows from Theorem \ref{thm:acc_convergence} and the following bound, which we show next:
\begin{equation*}
    \sum_{k=0}^N\theta_k \delta_k(X)\leq \frac{\mu (\overline{r^2}/\overline{r})^2}{C}D_0 .
\end{equation*}

The bound on the rest of the terms is since
\begin{displaymath}
    \frac{1}{\eta_k}\sum_{t=1}^{T_k} \frac{\lambda_t}{\beta_t^k} = \sum_{t=1}^{T_k} \frac{2tC/\mu}{(t+1) + (t-1)(k+\overline{r^2}/\overline{r})/2\overline{r}}\leq \frac{C}{\mu}(1+\frac{4(T_k-1)}{1+(k+\overline{r^2}/\overline{r})/2\overline{r}}),
\end{displaymath}
and notice that
\begin{displaymath}
    \sum_{k=0}^N \frac{k+2\overline{r^2}/\overline{r}}{k+\overline{r^2}/\overline{r}+2\overline{r}} \leq 2(N+1)\leq 4N,
\end{displaymath}
and since $N+1 \geq r_{\max}:=\max_{s\in [S]} r_s$, we have
\begin{displaymath}
    \overline{r^2} = \sum_{s=1}^S \rho_s r_s^2\leq \sum_{s=1}^S \rho_s r_s (N+1) = (N+1)\overline{r}\implies \overline{r^2}/\overline{r}\leq N+1\leq 2N.
\end{displaymath}
Thus, for $T_k/N = T/N\geq \max(\frac{5}{\sqrt{D_1}},\frac{64\overline{r}}{D_1})$ where $D_1 = \frac{\mu^2(\overline{r^2}/\overline{r})^2}{2M^2C^2}D_0$, we have
\begin{align*}
     \sum_{k=0}^N \frac{2M^2\theta_k}{\eta_kT_k(T_k+1)}\sum_{t=1}^{T_k}\frac{\lambda_t}{\beta_t^k} &\leq \frac{2M^2C}{\mu}\left\{\sum_{k=0}^N\frac{k+2\overline{r^2}/\overline{r}}{T_k(T_k+1)}+\sum_{k=0}^N \frac{4(k+2\overline{r^2}/\overline{r})}{(1+(k+\overline{r^2}/\overline{r})/2\overline{r})(T_k+1)}\right\}\\
     &\leq \frac{2M^2C}{\mu}\left\{\frac{10N^2}{T^2}+\frac{32N\overline{r}}{T}\right\}\leq \frac{\mu }{C}D_0(\overline{r^2}/\overline{r})^2. 
\end{align*}

\end{proof}

As a direct consequence, we have the following theorem. 
\begin{corollary}\label{cor:main_acc} 
For $\widehat{X}\in \overline{\mc X}$, assume that the following are finite:
\begin{displaymath}
  D(\widehat{X},X^{init})\leq D^X<\infty,\quad  \sup_{y_s\in \domain(R^*_s)}D(y_s,y_s^{init})\leq D_s^y<\infty. 
\end{displaymath}
Under the conditions in Corollary \ref{cor:acc_convergence}, taking $\rho_s =\frac{\widetilde{\kappa}_s\sqrt{D_s^y}}{\sum_{s'=1}^S \widetilde{\kappa}_{s'}\sqrt{D_{s'}^y}}$ and $D_0 = D^X$, 
\begin{displaymath}
     \sup_{Y'\in \R^n}G(Z^N;\widehat{X},Y')  \leq \frac{2}{N(N+1)}\left\{\frac{\mu(\overline{r^3}/\overline{r} + 7(\overline{r^2}/\overline{r})^2)}{2C}D^X+ \frac{4C(\overline{r})^2}{\mu}(\sum_{s'=1}^S \widetilde{\kappa}_{s'}\sqrt{D_{s'}^y})^2 \right\}.
\end{displaymath}
\end{corollary}

Notice that \eqref{eq:prop-F} implies that $\|x-x'\|\leq \frac{M}{\mu}$ for all $x,x'\in \mc X$. Thus, one can take $D^X = O(\frac{CM^2}{\mu^2})$. The resulting upper bound, when $\overline{r^3} = O((\overline{r})^3)$ and $\overline{r^2} = O((\overline{r})^2)$, becomes 
\begin{displaymath}
     \sup_{Y'\in \R^n}Q(Z^N;\widehat{X},Y')  =O( \frac{\overline{r}^2}{\mu N^2}\left\{M^2+ C(\sum_{s'=1}^S \widetilde{\kappa}_{s'}\sqrt{D_{s'}^y})^2 \right\}).
\end{displaymath}

\section{Application to distributed optimization}\label{sec:app-do}

For the problem \eqref{eq:objective_f}, we assume that $f_v:\overline{\mc X}\to \R$ is a convex and possibly non-smooth objective function such that for some $M_f,\mu_f\geq 0$, we have for all $v\in V$, 
\begin{equation}\label{eq:prop-f}
    \frac{\mu_f}{2}\|x-x'\|^2\leq f_v(x) - f_v(x')-\langle f' (x'),x-x')\rangle\leq M_f\|x-x'\|,\quad\forall x,x'\in \overline{\mc X},
\end{equation}
where $f_v':\mc X\to \R^d$ is a subgradient oracle, i.e. $f'_v(x)\in \partial f_v(x)$ for all $x\in \mc X$, and $f'_v$ is only available to agent $v$. For instance, when $\|f'_v\|_*\leq M_f$, $M = 2M_f$ holds.

To apply the proposed saddle point algorithms to the distributed optimization problem \eqref{eq:objective_f}, we adopt the popular approach of a lifted space reformulation\cite{lan_communication-efficient_2020,stephen_distributed_2011}, where the decision variables become $(x_v)_{v\in V}$, and $x_v$ is agent $v$'s local version of the decision variable $x$. The agents collaborate to reach \textit{consensus} on an approximate \textit{minimizer} of \eqref{eq:objective_f}. Below, we provide the detailed lifted space reformulation (Section \ref{sec:do-setup}), its connection to communication protocols (Section \ref{sec:agents}), and heterogeneity between the local objectives $f_v$ (Section \ref{sec:two_measures}). Then, in Section \ref{sec:do-pdhg}, we describe how our proposed multi-timescale PDHG can be applied, and provide the convergence guarantee performance.

\subsection{Lifted space formulation}\label{sec:do-setup}
To set the stage, we denote $\mc X = \prod_{v\in V} \overline{\mc X}\subset \R^{d}$, where $d = m\overline{d}$ is the dimension of the lifted space. it's easy to see that since $\overline{\mc X}$ is convex, so is the resulting $\mc X$. In the lifted space, to ensure that agents (approximately) reach consensus, that is $x_v\approx x_{v'}$ for all $v,v'\in V$, we need to impose the \textit{consensus constraint}, which results in the following \textit{penalized}problem: 
\begin{equation}\label{eq:objective_F}\tag{$\mc P_{d}^{lift}$}
    \min_{X=(x_v)_{v\in V}\in {\mc X}}\max_{Y = (y_1,\ldots,y_S)\in \R^n} F(X) + \sum_{s=1}^S \left(\langle K_sX,y_s\rangle - R^*_s(y_s)\right),\quad F(X):=\sum_{v\in V} f_v(x_v).
\end{equation}

Notice that the formulation \eqref{eq:objective_F} is exactly the same as the formulation in \eqref{eq:objective_pd_block}, thereby making (multi-timescale) PDHG applicable. Moreover, by \eqref{eq:prop-f}, \eqref{eq:prop-F} holds with $M = \sqrt{m}M_f$ and $\mu = \mu_f$. Below, we provide the details of $K_s$ and $R_s$ in \eqref{eq:objective_F}.

\subsubsection{Consensus constraints} 

In \eqref{eq:objective_F}, $K_s\in \R^{n_s\times d}$ is a matrix such that $\cap_{s\in S} \ker(K_s)$ is the subspace in $\R^{d}$ where all $\{x_v\}_{v\in V}$ are the same. For convenience, we denote $K_sX = \sum_{v\in V} K_{s,v}x_v$ for $K_{s,v}:\R^{\overline{d}}\to \R^{n_s}$, and abbreviate $K:\R^{d}\to \R^n$ where $(KX)_s = K_sX$. We make the following assumption. 
\begin{assumption}\label{assumption:kernel}
    $KX = \mb 0$ if and only if $x_v=x_{v'}$ for all $v,v'\in V$. 
\end{assumption}
As an example, denoting $\Pi:\R^{d}\to \R^{d}$ as the projection such that for any $X\in \R^{d}$, $\Pi(X)_v = \frac{1}{m}\sum_{v'\in V}x_{v'}$, then we can take $K = I-\Pi$. Moreover, for any $K$ satisfying Assumption \ref{assumption:kernel}, $K^*(KK^*)^{\dagger}K = I- \Pi$ holds.

We point out that the kernel condition in Assumption \ref{assumption:kernel} is the only requirement we impose on $K$. However, to make sure the resulting primal-dual algorithm can be implemented in a distributed fashion, additional sparsity requirements are needed depending on how agents communicate with each other. We provide details, examples, and the rationale behind the block-decomposable formulation in Section \ref{sec:agents}.

\subsubsection{Penalties}

In \eqref{eq:objective_F}, we assume that $R_s:
\R^{n_s}\to \overline{\R}$ is proper, convex, and lower-semicontinous, and $R_s(\mb 0) = 0$. 
Under these conditions, \eqref{eq:objective_pd_block} (and thus \eqref{eq:objective_F}) admits a primal-only formulation \eqref{eq:objective_pd_block-primal}. In this formulation, $R_s(K_sX)$ becomes a penalty term, penalizing the deviation of $K_sX$ from $\mb 0$. We further define $R:\R^n\to \overline{\R}$ as $R(y_1,\ldots,y_S) = \sum_{s=1}^SR_s(y_s)$.

As an example, if for all $s\in [S]$, $R_s$ is the characteristic function of the set $\{\mb 0\}$, i.e. $R_s(\mb 0) = 0$ and $R_s(y_s) = \infty$ for $y_s \neq \mb 0$, then \eqref{eq:objective_F} is equivalent to \eqref{eq:objective_f}, and $R_s^*(y_s) = 0$ for all $y_s\in \R^{n_s}$. As another example, $R_s$ can be any scaled norm, for instance $R_s(y_s) = \lambda \|y_s\|_p$ for some $p\geq 1$ and $\lambda>0$, then $R_s^*(y_s) = 0$ for $\|y_s\|_q\leq \lambda$ and $R_s^*(y_s) = \infty$ otherwise, where $\|\cdot \|_q $ is the dual norm of $\|\cdot\|_p$ (i.e. $p^{-1}+q^{-1}=1$). That is $R_s^*$ is the characteristic function of the dual-norm-ball of size $\lambda$. 

Notice that with penalties different from the characteristic functions of $\{\mb 0\}$, the problem \eqref{eq:objective_F} becomes a ``relaxation'' of \eqref{eq:objective_f}, thereby are not equivalent. As one can imagine, the two formulations get closer as the consensus constraints are penalized more. In fact, there are two tensions in choosing good penalties. Take $S=1$ and $R(Y) = \lambda\|Y\|$ as an example. 
\begin{itemize}
    \item On one hand, with larger penalty on the consensus constraint violation (i.e. larger $\lambda$), \eqref{eq:objective_F} will become a better proxy for \eqref{eq:objective_f}; this encourages larger $\lambda$.
    \item On the other hand, the diameter of $\domain(R^*)$ is $O(\lambda)$, and as suggested by Corollaries \ref{cor:main} and \ref{cor:acc_convergence}, the complexities of our algorithms are $O(\lambda)$; this encourages smaller $\lambda$. 
\end{itemize}
In Section \ref{sec:two_measures}, we propose a set of conditions on the penalties to achieve a balance between these two tensions, such that ``good solutions'' to \eqref{eq:objective_F} and \eqref{eq:objective_pd_block} (as measured using the duality gap \eqref{eq:def_gap}) are also ``good solutions'' to \eqref{eq:objective_f} (as measured by objective value suboptimality and constraint violation in \eqref{eq:def_eps_delta}), and the resulting algorithms have favorable dependence on the problem parameters.

\subsubsection{Performance measure} 
To measure the performance of $X\in {\mc X}$, following \cite{lan_communication-efficient_2020}, we consider the $(\epsilon,\delta)$-solution, satisfying the following conditions
\begin{equation}\label{eq:def_eps_delta}
    F(X) \leq F(X^*)+\epsilon,\quad \|(I-\Pi)X\|\leq \delta.
\end{equation}
That is, $X$ is $\epsilon$-suboptimal in terms of the objective value, while violating the consensus constraints by at most $\delta$.\footnote{In \cite{lan_communication-efficient_2020}, $\|KX\|\leq \delta$ is used instead of $\|(I-\Pi)X\|$, and $K$ is assumed to be the Laplacian matrix for the underlying graph of communication. However, we use a generic $K$ satisfying condition \ref{assumption:kernel}. In particular, for any $K$ that is a valid choice, $\lambda K$ is also valid for any $\lambda\neq 0$. Thus, it makes sense to ``normalize'' $K$, and we use $I-\Pi = K^*(KK^*)^{\dagger}K$.} Nevertheless, our algorithms have performance guarantees on the duality gap of the saddle point fromulation. To transfer such duality gap guarantee back to $(\epsilon,\delta)$-solution guarantee, in Section \ref{sec:two_measures}, we propose additional requirements for the regularization $R$.

\subsection{Agents, communication, and additional requirements on $K$}\label{sec:agents}

By distributed optimization, we mean that the objective functions $\{f_v\}_{v\in V}$ are distributed among $m$ \textit{primal agents}: for each $v\in V$, $\agent(x_v)$ has access to $f_v'$, the first order oracle for $f_v$, and is responsible for updating the variable $x_v$. In addition, we assume that there are $S$ \textit{dual agents}: for each $s\in [S]$, $\agent(y_s)$ is responsible for updating the variable $y_s$. 

We assume that for any pair $(s,v)\in [S]\times V$ such that $K_{s,v}\neq \mb 0$, $\agent(x_v)$ and $\agent(y_s)$ can communicate (in both directions). For instance, all agents might be nodes in a connected graph with vertices $[S]\cup V$ (representing $S$ dual agents and $m$ primal agents), and communication can be realized through edges (directly) or through paths (i.e. with the help of intermediate agents). In particular, since the graph is connected, any pair can communicate, but the resources consumed and/or time taken by communication between different pairs could be (significantly) different. 

At this point, we abstract away from how such communication is realized, and leave the discussion of the costs of communication to Section \ref{sec:do-pdhg}. Below, we provide two such realizations: \textit{decentralized} and \textit{hierarchical}, and provide examples in Figure \ref{fig:communication}.

\textbf{Decentralized setting.} In this setup, we assume that the dual variables are kept and updated by primal agents, respecting a graph based communication constraints. Precisely, let $\mc G = (V,E)$ denote an undirected, connected graph, and for each $s\in [S]$, we assign all tasks of $\agent(y_s)$ to $\agent(x_{v_s})$ for some $v_s\in V$, such that $\{v_s,v'\}\in E$ for each $K_{s,v'}\neq \mb 0$. 

As an example, let $W\in \R^{V\times V}$ be a doubly stochastic matrix such that $W_{v,v'}\neq 0$ only if $\{v,v'\}\in E$ or $v=v'$, and $\ker(I-W) = \setspan(\mb 1)$ (and so $K:=(I-W)\otimes I_{\overline{d}}$ satisfies Assumption \ref{assumption:kernel}). We can choose $S = m$, $n_s = d$, and decompose $K$ as $K_s:=(I-W)_s\otimes I_{\overline{d}}$, 
\begin{displaymath}
    K_sX = \sum_{v\in V} (I-W)_{s,v}x_v = x_s-\sum_{\{v, s\}\in E}W_{s,v}x_v, \quad s = 1,\ldots,m.
\end{displaymath}
Thus, $\agent(y_s)$'s tasks can be assigned to $\agent(x_s)$.

\textbf{Hierarchical setting.} In this setup, we assume that there is an underlying tree with nodes $[S]\cup V$, where all non-leaf nodes ($[S]$) correspond to dual agents and all leaf nodes ($V$) correspond to primal agents. Each non-leaf node can communicate with its child nodes directly. Precisely, for $s\in [S]$, we use $\child(s) \subset [S]\cup V$ to denote the child nodes of $\agent(y_s)$, and $\Descendant(s)\subset V$ to denote all \textit{primal agents} in the subtree rooted at $\agent(y_s)$. 

For convenience, for each $s\in [S]$, we denote the ``mean'' of all descendants of $\agent(y_s)$ as $\overline{x}_s = |\Descendant(s)|^{-1}\sum_{j\in \Descendant(s)} x_j$. Then, consider $K_s:\R^{d}\to \R^{|\child(s)|{\overline{d}}}$ defined as 
\begin{equation}\label{eq:def_K}
    (K_sX)_{i}= \overline{x}_i - \overline{x}_s = \overline{x}_i - \sum_{j\in \child(s)}\frac{|\Descendant(j)|}{|\Descendant(s)|}\overline{x}_j,\quad i\in \child(s).
\end{equation}
Then, it is easy to see that $K$ satisfies Assumption \ref{assumption:kernel}, and since $\overline{x}_j$ can be computed in a bottom up manner, $\{K_s\}_{s\in [S]}$ can be realized through this tree. In addition, the set of $\{K_s\}_{s\in [S]}$ admits the following orthogonality properties which will be useful in choosing $R$. We defer the proof to Appendix \ref{sec:proof_setup}. 
\begin{lemma}\label{lm:K_prop}
    Let $\{K_s\}_{s\in [S]}$ be as defined in \eqref{eq:def_K}. Then for $s\neq s'\in [S]$, $K_sK_{s'}^* = \mb 0$. In addition, denoting $\Pi_s := K_s^*(K_sK_s^*)^{\dagger}K_s$, we have for any $\widetilde{X},\widehat{X}\in \R^{d}$
    \begin{displaymath}
        \langle \widehat{X},\Pi_s\widetilde{X}\rangle=\langle \Pi_s\widehat{X},\Pi_s\widetilde{X}\rangle= \sum_{i\in \child(s)} |\Descendant(i)| \cdot \langle (K_s\widehat{X})_i,(K_s\widetilde{X})_i\rangle.
    \end{displaymath}
\end{lemma}

\begin{figure}[htbp]
\centering
\resizebox{.9\linewidth}{!}{
\begin{tikzpicture}[
primal/.style={circle, draw, minimum size=25},
dual/.style={rectangle, draw, minimum size=25}]
\node [primal] (x1) at (0,0) {$x_1$};
\node [primal] (x2) at (1.5,0) {$x_2$};
\node at (3,0) {$\cdots$};
\node [primal] (xm) at (4.5,0) {$x_m$};
\node [dual] (y1) at (0.75,1.5) {$y_1$};
\node at (2.25,1.5) {$\cdots$};
\node [dual] (ys) at (3.75,1.5) {$y_S$};
\draw (x1)--(y1);
\draw (x1)--(ys);
\draw (x2)--(y1);
\draw (x2)--(ys);
% \draw (xm)--(y1);
\draw (xm)--(ys);

\node [primal] (p1) at (6.5,0) {$x_1$};
\node [dual] (d1) at (6.5,2) {$y_1$};
\node [primal] (p2) at (9,0) {$x_2$};
\node [dual] (d2) at (9,2) {$y_2$};
\node [primal] (p3) at (10.25,1.5) {$x_3$};
\node [dual] (d3) at (10.25,3.5) {$y_3$};
\node [primal] (p4) at (7.75,1.5) {$x_4$};
\node [dual] (d4) at (7.75,3.5) {$y_4$};
\draw (p1)--(p3);
\draw (p1)--(p4);
\draw (p2)--(p3);
\draw [densely dotted] (d1)--(p1);
\draw [densely dotted] (d1)--(p3);
\draw [densely dotted] (d1)--(p4);
\draw [densely dotted] (d2)--(p2);
\draw [densely dotted] (d2)--(p3);
\draw [densely dotted] (d3)--(p1);
\draw [densely dotted] (d3)--(p2);
\draw [densely dotted] (d3)--(p3);
\draw [densely dotted] (d4)--(p1);
\draw [densely dotted] (d4)--(p4);

\node [dual] (hd3) at (12.75 + 3*1,3) {$y_4$};
\foreach \x in {0,...,5}
  {\pgfmathtruncatemacro{\label}{\x+1}
    \node [primal] (hp\x) at (12 + 1.5*\x,0) {$x_{\label}$};}
    % \draw (hp\x)--(hd3);}
\foreach \y in {0,...,2}
  {\pgfmathtruncatemacro{\label}{\y+1}
  \pgfmathtruncatemacro{\labell}{\y*2}
  \pgfmathtruncatemacro{\labelll}{\y*2+1}
    \node [dual] (hd\y) at (12.75 + 3*\y,1.5) {$y_{\label}$};
    \draw (hp\labell)--(hd\y);
    \draw (hp\labelll)--(hd\y);
    \draw (hd3)--(hd\y);
    }
\end{tikzpicture}}
\caption{Left: abstract setting with $m$ primal agents and $S$ dual agents. Middle: realization in the decentralized setting, where $S = m=4$, $\agent(x_s) = \agent(y_s)$, and the underlying graph is $(V,E=\{\{1,3\},\{1,4\},\{2,3\}\})$. Right: realization in the hierarchical setting.  
}\label{fig:communication}
\end{figure}

\subsection{Requirements for $R$ and function similarity}\label{sec:two_measures}

Recall that when $R$ is the characteristic function of $\{\mb 0\}$, the penalized formulation \eqref{eq:objective_F} is equivalent to \eqref{eq:objective_f}. In this section, we discuss the requirements for $R$ such that the duality gap provides upper bounds on the suboptimality of the objective value and the violation of the consensus constraints.  

First, we have the following upper bounds on the suboptimality of the objective value.  
\begin{lemma}\label{lm:eps-subopt}
For any $\widehat{X}\in {\mc X}$ such that $K\widehat{X} = \mb 0$,
\begin{displaymath}
F(X)\leq F(\widehat{X}) + \sup_{Y'\in \domain(R^*)} G(X,Y;\widehat{X},Y').
\end{displaymath}
In particular, if $\sup_{Y'\in \domain(R^*)} G(X,Y;X^*,Y')\leq \epsilon$, then $F(X)\leq F(X^*) + \epsilon$, where $X^* = (x^*)_{v\in V}$ and $x^*$ is an optimal solution to \eqref{eq:objective_f}. 
\end{lemma}
\begin{proof}[Proof of Lemma \ref{lm:eps-subopt}]
   Recall that we have
\begin{displaymath}
    \sup_{Y\in \R^n} \langle KX,Y\rangle + F(X) - R^*(Y) = F(X) + R(KX).
\end{displaymath}
In addition, since $R(\mb 0) =  0$, we have $R^*(Y) = \sup_{Y'\in \R^n} \langle Y',Y\rangle - R(Y')\geq \langle \mb 0,Y\rangle - R(\mb 0) =0$, and so
\begin{displaymath}
    \langle K\widehat{X},Y\rangle + F(\widehat{X}) - R^*(Y) \leq F(\widehat{X}),\quad \forall~\widehat{X}\in {\mc X},~K\widehat{X} = \mb 0.
\end{displaymath}
The second claim follows directly from the first since $KX^* = \mb 0$. 
\end{proof}

To connect the duality gap with the constraint violation $\|(I-\Pi)X\|$ in \eqref{eq:def_eps_delta}, or with the objective value suboptimality of $\frac{1}{m}\sum_{v\in V} x_v$, it turns out additional requirements are needed for $R$. 

% To build intuition, suppose that $\mc X = \R^d$ and $\{f_v\}_{v\in V}$ are all differentiable, and so $f'_v = \nabla f_v$. In the saddle point problem \eqref{eq:objective_pd_block}, suppose we have found a dual solution $Y$ such that $(K^*Y)_v = -\nabla f_v(x^*)$ for all $v\in V$, then $X^*=(x^*)_{v\in V}$ is an optimal solution to the primal problem, as it satisfies the first order optimality condition that $\nabla f_v(x^*) +(K^*Y)_v = \mb 0 $ for all $v\in V$. Thus, the dual problem is essentially trying to find a $Y$ such that $K^*Y \approx - \nabla F(X^*)$, and in case $\{K_s\}_{s\in [S]}$ are orthogonal, one can further decompose the problem into $S$ subproblems. This intuition suggests that one should choose $R$ (without knowing $X^*$) such that $\domain(R^*)$ contains one such solution. On the other hand, as will be seen in Theorems \ref{thm:main} and \ref{thm:main_acc}, a too large domain could slow down the convergence. Below, we provide guidelines on how to achieve the balance. 

For convenience, we denote $\sigma_{\min}^+(K_s) = \min_{X\in \R^{d},~\Pi_s X\neq \mb 0}\frac{\|K_sX\|_*}{\|\Pi_s X\|}$, where the numerator uses the dual norm to the norm in $\R^{n_s}$ and the denominator uses the norm in $\R^{d}$. As an example, when all norms are $l_2$ norms, $\sigma_{\min}^+(K_s)$ is the smallest non-zero singular value of $K_s$.

\subsubsection{Requirements on $R$ under orthogonaltiy}
Below, we show that if $K_s$ measures the constraint violation in \textit{orthogonal} subspaces, then as long as $R_s$ grows fast enough, the duality gap provides an upper bound on the constraint violation $\|(I-\Pi)X\|$ and the suboptimality of $\Pi X$.

\begin{lemma}\label{lm:delta-subopt}
   Further assume that for any $s\neq s'\in [S]$, $K_sK_{s'}^* = \mb 0$, and for each $s\in [S]$, denoting
    \begin{displaymath}
        \Pi_s = K_s^*(K_sK_s^*)^{\dagger}K_s,\quad a_s \geq \sup_{X'\in {\mc X},~KX' = \mb 0} \|\Pi_s\nabla F(X')\|_*,
    \end{displaymath}
    where $\nabla F:{\mc X}\to \R^{d}$ is an arbitrary subgradient oracle, i.e. $(\nabla F(X))_v\in \partial f_v(x_v)$. If $\sup_{Y'\in \domain(R^*)} G(X,Y;X^*,Y')\leq \epsilon$, 
    \begin{enumerate}
        \item $X$ is an $(\epsilon,\epsilon/\xi)$-solution if for each $s\in [S]$,
    \begin{equation}\label{eq:r_s_1}
        R_s(y_s) \geq R^{ccv}_s(y_s):=\frac{\xi + a_s}{\sigma_{\min}^+(K_s)}\|y_s\|_*.
    \end{equation}
    
    \item  the projected solution $\Pi X$ is an $(\epsilon(1+1/\xi),0)$-solution if for each $s\in [S]$, 
    \begin{equation}\label{eq:r_s_2}
        a_s>0,\quad R_s(y_s) \geq R^{prj}_s(y_s):=\frac{(1+\xi) a_s}{\sigma_{\min}^+(K_s)}\|y_s\|_*.
    \end{equation}
    \end{enumerate}
\end{lemma}
In Lemma \ref{lm:delta-subopt} (and Corollary \ref{cor:delta_subopt} below), the superscript $ccv$ means $\{R^{ccv}_s\}_{s\in [S]}$ ($\{\widehat{R}^{ccv}_s\}_{s\in [S]}$) are designed to provide guarantees on the consensus constraint violation, and the superscript $prj$ means $\{R^{prj}_s\}_{s\in [S]}$ ($\{\widehat{R}^{prj}_s\}_{s\in [S]}$) are designed to provide guarantees the projected solution $\Pi X$.

\begin{proof}[Proof of Lemma \ref{lm:delta-subopt}]
First, notice that by the orthogonality of $\{K_s\}_{s\in [S]}$, for any $Y\in \R^n$
\begin{displaymath}
    (KK^*Y)_{s} =K_{s}(\sum_{s'=1}^S K_{s'}^*y_{s'}) = K_sK_s^*y_s, ~\forall s\in [S].
\end{displaymath}
That is, $KK^*$ is diagonal, and so 
\begin{displaymath}
    ((KK^*)^{\dagger}Y)_s =  (K_sK_s^*)^{\dagger}y_s, ~\forall s\in [S].
\end{displaymath}
Thus, we can make the following decomposition
\begin{displaymath}
    K^*(KK^*)^{\dagger}KX = \sum_{s=1}^SK_s^*(K_sK_s^*)^{\dagger}K_sX=\sum_{s=1}^S\Pi_sX.
\end{displaymath}
For convenience, we denote $\widetilde{X}:=\Pi X$, and by Lemma \ref{lm:delta-subopt},
\begin{equation}\label{eq:bd1}
    \sup_{Y'\in \domain(R^*)} G(X,Y;X^*,Y')\leq \epsilon \implies F(X) + R(KX) \leq F(X^*) + \epsilon\leq F(\widetilde{X}) + \epsilon.
\end{equation}
In addition, using the convexity of $F$,
\begin{align}\label{eq:bd2}
   F(\widetilde{X}) - F(X) &\leq -\langle \nabla F(\widetilde{X}), (I-\Pi)X\rangle= -\sum_{s=1}^S \langle \nabla F(\widetilde{X}), \Pi_sX\rangle\nonumber\\
   &\leq \sum_{s=1}^S \|\Pi_s \nabla F(\widetilde{X})\|_*\cdot \|\Pi_sX\|\leq \sum_{s=1}^S a_s \cdot \|\Pi_sX\|.
\end{align}

For the first claim, since $\|K_sX\|_*\geq \|\Pi_sX\|\sigma_{\min}^+(K_s)$, with the first condition \eqref{eq:r_s_1} on $R_s$, we have
\begin{equation}\label{eq:bd3}
    R_s(K_sX) \geq (\xi + a_s)\cdot \|\Pi_sX\|.
\end{equation}
Combining the \eqref{eq:bd1}, \eqref{eq:bd2}, and \eqref{eq:bd3}, we get
\begin{displaymath}
    \xi\cdot \sum_{s=1}^S \|\Pi_sX\|\leq \epsilon \implies \|(I-\Pi) X\|=\|\sum_{s=1}^S \Pi_sX\|\leq \sum_{s=1}^S \|\Pi_sX\| \leq \epsilon/\xi.
\end{displaymath}
For the second claim, following a similar argument as above but with the second condition \eqref{eq:r_s_2} on $R_s$, we get
\begin{equation}\label{eq:bd3-1}
    \sum_{s=1}^S a_s\cdot \|\Pi_s X\| \leq \epsilon/\xi.
\end{equation}
Thus, using \eqref{eq:bd1}, \eqref{eq:bd2}, and \eqref{eq:bd3-1}, we have
\begin{displaymath}
        F(\widetilde{X}) \leq F(X) + \epsilon/\xi \leq F(X^*) + \epsilon/\xi + \epsilon.
\end{displaymath}
\end{proof}

We would like to point out that in \eqref{eq:bd2}, $\langle \nabla F(\widetilde{X}),\Pi_s X\rangle$ is upper bounded using $\| \nabla F(\widetilde{X})\|_*\cdot \|\Pi_s X\|$. A tighter upper bound could be obtained if one has more information about the set $\mc G_s:=\{\Pi_s \nabla F(X'),~X'\in {\mc X},~KX' = \mb 0\}$. Indeed, $\langle \nabla F(\widetilde{X}),\Pi_s X\rangle\leq \sup_{G_s\in \mc G_s} \langle G_s,\Pi_s X\rangle$, and so the inner product can be bounded using the support function of the set $\mc G_s$.  

\subsubsection{Function similarity for general convex functions}\label{sec:func_simi}
The terms $\{a_s\}_{s\in [S]}$ in Lemma \ref{lm:delta-subopt} can be viewed as a ``decomposition'' of the function variation into different subspaces spanned by (the row spaces of) $\{K_s\}_{s\in [S]}$. To be more concrete, consider the hierarchical setting presented in Section \ref{sec:agents}, which satisfies exactly the conditions in Lemma \ref{lm:delta-subopt} due to Lemma \ref{lm:K_prop}. Defining $\mu_s(i) = \frac{|\Descendant(i)|}{|\Descendant(s)|}$ for $i\in \child(s)$ as a probability measure, and assuming that all norms are $l_2$ norms, then by Lemma \ref{lm:K_prop}, 
\begin{equation}\label{eq:var_interp}
    \|\Pi_s \nabla F\|_*^2 = |\Descendant(s)|\cdot\var_{i\sim \mu_s}(\overline{f'}_i),\quad \overline{f'}_i = \frac{\sum_{j\in \Descendant(i)}f'_j}{|\Descendant(i)|},~i\in \child(s),
\end{equation}
where for a random vector $V$, we denote $\var(V):= \E[\|V-\E[V']\|_*^2]$. Thus, $\|\Pi_s \nabla F\|_*$ measures the function variation among the \textit{descendants of different child nodes} of $\agent(y_s)$, i.e. among $\left\{\sum_{j\in \Descendant(i)} f'_j\right\}_{i\in \child(s)}$. As a result, the agents closer to the root of the tree, with more descendants, take care of function variation at \textit{larger scales}, but at \textit{lower resolution}, since for all $i\in \child(s)$, the variation inside $\left\{ f'_j(x)\right\}_{j\in \Descendant(i)}$ has been taken care of by the dual agents in each sub-tree rooted at $i$.

For general but still orthogonal $\{K_s\}_{s\in [S]}$, $a_s$ measures the function variation along the span of $K_s$. With this interpretation in mind, we make the following definition regarding function similarity.
\begin{definition}\label{def:func_simi}
Assume that for all $s\neq s'\in [S]$, $K_sK_{s'}^* = \mb 0$. We say that the set of functions $\{f_v\}_{v\in V}$ is $\{(a_s,K_s)\}_{s\in [S]} $-similar if there exists a subgradient oracle $\nabla F:{\mc X}\to \R^{d}$, i.e. $(\nabla F(X))_v\in \partial f_v(x_v))$, such that for each $s\in [S]$, 
    \begin{displaymath}
        \Pi_s = K_s^*(K_sK_s^*)^{\dagger}K_s,\quad a_s \geq \sup_{X'\in \overline{\mc X},~KX' = \mb 0} \|\Pi_s\nabla F(X')\|_*.
    \end{displaymath}
If $S=1$ and $\Pi_1 = I-\Pi$, we abbreviate $\{(a_1,K_1)\}$-similar as $a_1$-similar. 
\end{definition}
For instance, if $S = 1$ and all norms are $l_2$ norms, then Assumption \ref{assumption:kernel} requires that $\Pi_1 = I-\Pi$, and one can take $a_1$ as 
\begin{displaymath}
    a_1^2 \geq \sup_{x\in \overline{\mc X}} \sum_{v\in V}\|f'_v(x) - \frac{1}{m}\sum_{v'\in V} f_{v'}'(x)\|^2.
    \end{displaymath}
Thus, if $\|f_v'(x)\|\leq M_f$ for all $v\in V,x\in \mc X$, we can also take $a_1= 2\sqrt{m}M_f$.

\textbf{Comparisons with existing notions of function similarity.} \cite{pmlr-v119-karimireddy20a} proposes the \textit{bounded gradient dissimilarity} for differentiable convex objectives, which coincides with our Definition \ref{def:func_simi} when $S=1$ and when the objectives are differentiable. For twice differentiable objectives, function similarity is also defined in terms of differences in Hessians, i.e. $\|\nabla^2 f_v -\nabla^2 f_{v'}\|$ \cite{pmlr-v151-tian22b,konecny2016federatedoptimizationdistributedmachine,ArjevaniShamir2015Communication,pmlr-v119-karimireddy20a}. For general convex functions which could be non-differentiable, \cite{ArjevaniShamir2015Communication} informally defines it (\textit{$\delta$-relatedness} in their terminology) as the condition that ``subgradients of local functions are at most $\delta$-different from each other''. Our Definition \ref{def:func_simi} formalizes this idea, and extend it to the case where $S>1$.

\subsubsection{Requirements on $R_s$ without orthogonality}
The above Lemma \ref{lm:delta-subopt} imposes orthogonality assumptions on $\{K_s\}_{s\in [S]}$. In the more general case where such assumptions do not hold, one can always view \eqref{eq:objective_pd_block} as a problem with only $1$ block, with $K$ and $R$ as the corresponding operator and regularization. Applying Lemma \ref{lm:delta-subopt}, we get the following corollary. 

\begin{corollary}\label{cor:delta_subopt}
Denoting $\widehat{a}_1 \geq \sup_{X'\in {\mc X},~KX'=\mb 0} \|(I-\Pi)\nabla F(X')\|_*$
    where $\nabla F:{\mc X}\to \R^{d}$ is an arbitrary subgradient oracle, i.e. $(\nabla F(X))_v\in \partial f_v(x_v)$. If $\sup_{Y'\in \domain(R^*)} G(X,Y;X^*,Y')\leq \epsilon$, 
    \begin{enumerate}
        \item $X$ is an $(\epsilon,\epsilon/\xi)$-solution if for each $s\in [S]$,
    \begin{equation}\label{eq:r_1}
        R_s(y_s) \geq \widehat{R}^{ccv}_s(y_s):=\frac{\xi + \widehat{a}_1 }{\sigma_{\min}^+(K)}\|y_s\|_*,
    \end{equation}
    \item  assume that $\widehat{a}_1>0$, then the projected solution $\Pi X$ is an $(\epsilon(1+1/\xi),0)$-solution if for each $s\in [S]$, 
    \begin{equation}\label{eq:r_2}
        R_s(y_s) \geq \widehat{R}^{prj}_s(y_s):=\frac{(1+\xi) \widehat{a}_1 }{\sigma_{\min}^+(K)}\|y_s\|_*.
    \end{equation}
    \end{enumerate}

\end{corollary}

% Assume that all norms are $l_2$ norms, then with orthogonality, $KK^*$ is ``diagonal'' and so $\sigma_{\min}^+(K) = \min_{s\in [S]} \sigma_{\min}^+(K_s)$. In addition, since $\|(I-\Pi)\nabla F(X)\|\geq \|\Pi_s \nabla F(X)\|$ for any $X\in {\mc X}$ and $s\in [S]$, one can always take $a_s \leq \widehat{a}_1$ for all $s\in [S]$. Thus $\frac{\xi + \widehat{a}_1 }{\sigma_{\min}^+(K)}\geq \frac{\xi + a_s}{\sigma_{\min}^+(K_s)}$. Since for the function $h(x) = \lambda \|x\|$ defined on $\R^{n_0}$ for some $\lambda >0$, the conjugate $h^*$ is the characteristic function of $\{x\in \R^{n_0}|\|x\|\leq \lambda\}$, i.e. $\domain(h^*) = \{x\in \R^{n_0}|\|x\|\leq \lambda\}$. As will be seen in Theorem \ref{thm:main}, the convergence is faster with smaller domains, suggesting that one should use the more refined decomposition when orthogonality holds. 

\textbf{Comparisons with \cite{lan_communication-efficient_2020} when $S=1$.} Assume that $\widehat{R}^{prj}_1$ in \eqref{eq:r_2} is used for some constant $\xi>0$ and $\widehat{a}_1 = 2\sqrt{m}M_f$, where $M_f$ (defined below) is an upper bound on the norm of \textit{the subgradient oracle} $f'_v\in \partial f_v$ (i.e. only one subgradient in the subdifferential for each $x\in \overline{\mc X}$, $v\in V$). Then, the diameter of $\domain(R^*)$ is $O(\frac{\sqrt{m}M_f}{\sigma^+_{\min}(K)})$. In \cite{lan_communication-efficient_2020}, it is shown that for \eqref{eq:objective_pd_block} with $R$ being the characteristic function of $\{\mb 0\}$, there exists an optimal dual solution $\|Y^*\|\leq \frac{\sqrt{m}\widehat{M}_f}{\sigma^+_{\min}(K)}$, where $\widehat{M}_f$ is an upper bound on the norms of \textit{all subgradients} $g\in \partial f_v$:
\begin{displaymath}
    \widehat{M}_f := \sup_{x\in \overline{\mc X},~v\in V,~g\in \partial f_v(x)} \|g\|_* \geq  M_f:=\sup_{x\in \overline{\mc X},~v\in V} \|f'_v(x)\|_*.
\end{displaymath}
Thus, even without function similarity, our $\widehat{R}^{prj}_1$ provides better control over the dual variables, leading to faster convergence.

\subsection{Applying (accelerated) multi-timescale PDHG to distributed optimization}\label{sec:do-pdhg}
In this section, we present the distributed implementation of Algorithm \ref{alg:main-algo} to the primal-dual formulation of the problem \eqref{eq:objective_F}. We first point it out that if the distance generating function $w_X$ in the lifted space is separable, i.e. $w_X(X) = \sum_{v\in V}w_{x_v}(x_v)$, then the Bregman divergence is also separable, i.e. $D(X,X') = \sum_{v\in V} D(x_v,x'_v)$. Now revising the updates \eqref{eq:pdhg-x} and \eqref{eq:pdhg-x-extra} as well as our proposed multi-timescale updates \eqref{eq:update_x_tilde}, \eqref{eq:primal-md-obj}, and the approximation using the generalized communication sliding \eqref{eq:update-x-cs}, we see that they are all decomposable w.r.t. the primal agents, and can be implemented locally by each $\agent(x_v)$ without any communication. More precisely, defining 
\begin{equation}\label{eq:primal-md-obj-decompose}
    \phi_v^k(x_v):=f_v(x_v) + \langle \sum_{s=1}^SK_{s,v}^* \overline{y}_s^k,x_v\rangle + \sum_{s=1}^S \eta_{k,s}D(x_v,x_v^{k-r_s}), 
\end{equation}
then we require that 
\begin{equation}\label{eq:primal-approx-prop-decompose}
    \phi_v^k(\widehat{x}_v^k) \leq \phi_v^k(x_v) -(\frac{\mu_f}{C} +\eta_k)D(x_v,x_v^k)+ \delta_{k,v}(x_v),\quad \forall x_v\in \overline{\mc X},
\end{equation}
which can be achieved through the generalized communication sliding procedure applied to each $x_v$ locally by $\agent(x_v)$. Thus, the only communication needed for the primal dual updates is to make sure the following two conditions are met: 
\begin{itemize}
    \item the dual $\agent(y_s)$ knows $K_s\widetilde{X}_s^{i_s}$ at iteration $k=r_si_s$, which can be realized if at the beginning of iteration $k=r_si_s$, each primal agent $v$ computes $\widetilde{x}_{v}^{i_s}$ and send it to the dual agent $y_s$, then the dual agent computes $\sum_{v\in V} K_{s,v}\widetilde{x}_{v}^{i_s}$;
    \item the primal $\agent(x_v)$ knows $\sum_{s=1}^S K^*_{s,v} \overline{y}_s^k$ at iteration $k$; this can be achieved if after each dual update, the dual agent $y_s$ sends $y_{s}^{i_s} - y_s^{i_s-1}$ to all primal agents, then the primal agents can update $\sum_{s=1}^S K^*_{s,v} \overline{y}_s^k$ using $K^*_{s,v}(y_{s}^{i_s} - y_s^{i_s-1})
    $.
\end{itemize}

We provide such implementation in Algorithm \ref{alg:main-algo-distributed}, and explicitly mark the steps which require communication in green. Here, we provide few remarks. Notice that in Algorithm \ref{alg:main-algo}, $\agent(y_s)$ calculates $\sum_{v\in V} K_{s,v}\widetilde{x}_{s,v}^{i_s}$ and sends $y_s^{i_s} - y_s^{i_s-1}$, a vector in $\R^{n_s}$, $\agent(x_v)$ calculates $K_{s,v}^*(\overline{y}_s^{k} -\overline{y}_s^{k-1}) $ and sends $\widetilde{x}_{s,v}^{i_s}$,
a vector in $\R^{\overline{d}}$. In fact, there are many task assignment strategies: for instance, $K_{s,v}\widetilde{x}_{s,v}^{i_s}$ can also be computed by $\agent(x_v)$, and the message from $\agent(x_v)$ to $\agent(y_s)$ will be $K_{s,v}\widetilde{x}_{s,v}^{i_s}$. This is preferable if $\agent(x_v)$ can compute matrix-vector products faster/at lower cost than $\agent(y_s)$. Due to this variability, in the cost analysis below, we take a ``modular'' perspective and assume that the cost of updating $y_s$ (including all matrix-vector multiplication and communication) is $c_s$.

\begin{algorithm}
\caption{(Accelerated) Multi-timescale PDHG for distributed optimization}\label{alg:main-algo-distributed}
\begin{algorithmic}
\Require $\{\alpha_{s,i_s}\},\{\theta_k\},\{\eta_{k,s}\},\{\tau_{s,i_s}\},\{r_s\},X^{init},Y^{init}$
\Ensure Primal dual pair $Z^N$
\State Initialize $(X^{k'},\widehat{X}^{k'},Y^{k'})\gets (X^{init},X^{init},Y^{init})$ for all $k'< 0$
\For{$k = 0,1,\ldots,N$}\Comment{implicitly $i_s = \lfloor k/r_s\rfloor$ for all $s\in [S]$}
\For{$s\in [S]$ such that $k = 0 ~(\text{ mod } r_s)$}
\For{ $v\in V$ such that $K_{s,v}\neq \mb 0$}
\State $\agent(x_v)$ computes $\widetilde{X}_{s}^{i_s}$ using \eqref{eq:update_x_tilde} and {\color{teal}\underline{sends}} it to $\agent(y_s)$
\EndFor
\State {\color{blue}{Dual update:}} $\agent(y_s)$ computes $\sum_{v\in V} K_{s,v}\widetilde{x}_{v}^{i_s}$, then updates $y_s^{i_s}$ using \eqref{eq:update_y} 
\State $\agent(y_s)$ {\color{teal}\underline{sends}} $y_s^{i_s}-y_s^{i_s-1}$ ($y_s^0$ if $i_s=0$)
to $\agent(x_v)$ for all $v\in V$ such that $K_{s,v}\neq \mb 0$
\EndFor
\For{$v\in V$}
\State {\color{red}{Primal update:}} $\agent(x_v)$ computes $K^*\overline{Y}^k$ where $\overline{y}_s^k = y_s^{\lfloor k/r_s\rfloor}$ for $s\in [S]$, updates $(x_v^k,\widehat{x}_v^k)$ satisfying \eqref{eq:primal-approx-prop-decompose} 
\EndFor 
\EndFor
\State All $\agent(x_v)$ and $\agent(y_s)$ computes their components of $Z^N$ using \eqref{eq:ergodic}. 
\end{algorithmic}
\end{algorithm}

\subsection{Convergence for general convex objectives}\label{sec:do-convergence}

With additional assumptions specific to distributed optimization, and with proper choices of $R_s$'s, the duality gap for \eqref{eq:objective_F} can be related to the suboptimality in terms of objective values $F$ and/or violation of the consensus constraint for the original problem \eqref{eq:objective_f}. Next, we establish such connection.

\begin{corollary}\label{cor:main1}
Assume that all norms are the $l_2$ norm, and take $y_s^0 = \mb 0$, $w_{y_s}(y_s) = \frac{1}{2}\|y_s\|^2$, $w_x(x) = \frac{1}{2}\|x\|^2$, and $w_X(X) = \sum_{v\in V} w_x(x_v)$. Assume that conditions of Corollary \ref{cor:main} hold, $D(X^*,X^{init})\leq D^X$, and the following holds for $A$ specified below
\begin{displaymath}
     N \geq \frac{2\sqrt{3}\overline{r}A \sqrt{D^X}}{\epsilon }.
\end{displaymath}
\begin{enumerate}
    \item If $\{f_v\}_{v\in V}$ is $\widehat{a}_1$-similar, take $\rho_s =\frac{ \|K_s\|}{\sum_{s'=1}^S\|K_{s'}\|}$
    \begin{enumerate}
    \item $\frac{1}{N+1}\sum_{k=0}^N \widehat{X}^k$ is an $(\epsilon,\epsilon/\xi)$-solution if $R_s=\widehat{R}_s^{ccv}$ as defined in \eqref{eq:r_1} (then $\sqrt{2D_s^y}=\frac{\xi + \widehat{a}_1}{\sigma_{\min}^+(K)}$) and $A = \frac{(\sum_{s=1}^S\|K_s\|)}{\sigma_{\min}^+(K)}\cdot (\xi + \widehat{a}_1)$; 
\item $\Pi(\frac{1}{N+1}\sum_{k=0}^N \widehat{X}^k)$ is an $(\epsilon(1+1/\xi),0)$-solution if $R_s = \widehat{R}_s^{prj}$ as defined in \eqref{eq:r_2} (then $ \sqrt{2D_s^y}=\frac{(1+\xi) \widehat{a}_1}{\sigma_{\min}^+(K)}$) and $A= \frac{(1+\xi) (\sum_{s=1}^S \|K_s\|)\cdot \widehat{a}_1}{\sigma_{\min}^+(K)}$.
\end{enumerate}
\item If $\{f_v\}_{v\in V}$ is $\{(a_s,K_s)\}_{s\in [S]}$-similar, 
\begin{enumerate}
    \item $\frac{1}{N+1}\sum_{k=0}^N \widehat{X}^k$ is an $(\epsilon,\epsilon/\xi)$-solution if $R_s = R_s^{ccv}$ as defined in \eqref{eq:r_s_1} (then $\sqrt{2D_s^y}=\frac{\xi + a_s}{\sigma_{\min}^+(K_s)}$), $\rho_s = (\frac{\xi + a_s}{\sigma^+_{\min}(K_s)})/(\sum_{s'=1}^S \frac{\xi + a_{s'}}{\sigma^+_{\min}(K_{s'})})$, and $A=\sum_{s=1}^S (\xi + a_s)\cdot \frac{\|K_s\|}{\sigma_{\min}^+(K_s)}$; 

\item $\Pi(\frac{1}{N+1}\sum_{k=0}^N \widehat{X}^k)$ is an $(\epsilon(1+1/\xi),0)$-solution if $R_s = R_s^{prj}$ satisfies \eqref{eq:r_s_2} (then $\sqrt{2D_s^y}=\frac{(1+\xi )a_s}{\sigma_{\min}^+(K_s)}$), $\rho_s = (\frac{ a_s}{\sigma^+_{\min}(K_s)})/(\sum_{s'=1}^S \frac{ a_{s'}}{\sigma^+_{\min}(K_{s'})})$, and $A=(1+\xi)(\sum_{s=1}^S a_s\cdot \frac{\|K_s\|}{\sigma_{\min}^+(K_s)})$. 
\end{enumerate}
\end{enumerate}

\end{corollary}

Thus, the communication round $N$ depends on $\overline{r}$, the weighted average of the rates at which the duals are updated, as well as $A$, which measures the function similarities.  

\textbf{Bounds using the Lipschitz constants.}
Consider the case where $\|f'_v\|\leq M_f$, and to guarantee that $\frac{1}{N+1}\sum_{k=0}^N \widehat{X}^k$ is an $(\epsilon,\epsilon/\xi)$-solution, in Corollary \ref{cor:main1}, with $\{f_v\}_{v\in V}$ $\widehat{a}_1$-similar, we can take $\xi = \widehat{a}_1= 2\sqrt{m}M_f$ which gives the following $N_1$, and with $\{f_v\}_{v\in V}$ $\{(a_s,K_s)\}_{s\in [S]}$-similar, we can take $\xi = a_s = 2\sqrt{m}M_f$ for all $s\in [S]$, which gives the following $N_2$:
\begin{displaymath}
    N_1 = O(\frac{\overline{r}  M_f \sqrt{mD^X}}{\epsilon }\cdot \frac{\sum_{s=1}^S \|K_s\|}{\sigma_{\min}^+(K)}),\quad N_2 =O( \frac{\overline{r} M_f\sqrt{mD^X}}{\epsilon }\cdot (\sum_{s=1}^S \frac{\|K_s\|}{\sigma_{\min}^+(K_s)})).
\end{displaymath}

Both $N_1$ and $N_2$ depend linearly in $\overline{r}$. However, when $\{K_s\}_{s\in [S]}$ are orthogonal, as discussed in Section \ref{sec:two_measures}, 
$\sigma_{\min}^+(K)\leq \sigma_{\min}^+(K_s)$ for all $s$, and so in terms of the rounds of communication $N$, it appears that orthogonality allows a more refined (i.e. $s$-dependent) control over the decomposition of the function variation and thus the dual domain size, thereby achieving better convergence. In addition, similar to the argument in Section \ref{sec:convergence-convex}, when the cost of updating $y_s$ is $c_s$ and total cost is additive, one should choose $r_s \propto \sqrt{c_s/\|K_s\|}$ under $\widehat{a}_1$-similarity, and $r_s\propto \sqrt{c_s/(\|K_s\|/\sigma_{\min}^+(K_s))}$ under $\{(a_s,K_s)\}_{s\in [S]}$-similarity. Similar results hold for $\Pi(\frac{1}{N+1}\sum_{k=0}^N \widehat{X}^k) $ to be an $(\epsilon,0)$-solution.

\textbf{Bounds using the function similarity.}
In reality, sometimes the functions $\{f_v\}_{v\in V}$ exhibit similarity. For instance, in the extreme case $f_v = f_{v'}$ for all $v,v'\in V$, and thus communication is not needed at all! In that case, the bound on $\sup_{\widetilde{X}\in {\mc X},~K\widetilde{X}=\mb 0}\|(I-\Pi) \nabla F(\widetilde{X})\|$ (and other terms using $\Pi_s$) using the Lipschitz constant $M_f$ is too loose: in fact, one can choose $\widehat{a}_1 = a_s=\epsilon_0$ for all $s$ for arbitrarily small $\epsilon_0>0$, then when $R_s$ are set according to \eqref{eq:r_2} or \eqref{eq:r_s_2} with constant $\xi$, one only needs $N = O(\frac{\epsilon_0}{\epsilon}\sum_{s=1}^S r_s)$, which can be arbitrarily small. 

More generally, choosing $\xi=1$ and setting $R_s$ according to \eqref{eq:r_2} or \eqref{eq:r_s_2}, we obtain the following bound on the rounds of communication following under $\widehat{a}_1$-similarity ($N_3$) and $\{(a_s,K_s)\}_{s\in [S]}$-similarity ($N_4$):
\begin{displaymath}
    N_3 =O( \frac{\overline{r} \sqrt{D^X}}{\epsilon}\cdot \frac{\sum_{s=1}^S \|K_s\|}{\sigma_{\min}^+(K)}\cdot \widehat{a}_1),\quad N_4=O( \frac{\overline{r}\cdot \sqrt{D^X}}{\epsilon }\cdot (\sum_{s=1}^S \frac{a_s\|K_s\|}{\sigma_{\min}^+(K_s)})).
\end{displaymath}

 Importantly, the number of rounds needed now depends on the \textit{function similarity} instead of crude quantities such as Lipschitz constants. 

 In fact, when $S=1$, such dependency is optimal. Indeed, \cite{ArjevaniShamir2015Communication} designs a pair of ``chain like'' functions $\{F_1,F_2\}$, such that for any $\gamma\geq 0$, $\{\gamma F_1,\gamma F_2\}$ is $\sqrt{1.5}\gamma$-similar. In addition, when $m/2$ agents are given $\gamma F_1$ and the rest are give $\gamma F_2$, finding an $\epsilon$ suboptimal $x$ (in terms of the objective value) in the $l_2$ unit ball requires $\Omega(\frac{\gamma}{\epsilon/m})$ rounds of communication (see Theorem 2 and the discussions after it in \cite{ArjevaniShamir2015Communication}). For our algorithm, with $\widehat{a}_1 = O(\sqrt{m}\gamma)$, $D^x = 1/2$, and 
 $K = I-\Pi$ (and so $\|K\| = \sigma_{\min}^+(K)$), we have $N_3 = O(\frac{r_1\gamma}{\epsilon/m})$. Thus, $y_1$ is updated only $N_3/r_1=O(\frac{\gamma }{\epsilon/m})$ times, which is also the number of actual communication rounds needed. This achieves the theoretical lower bound, and so is optimal.

  % However, to achieve the above function-variation-dependent bounds, the parameters $\widehat{a}_1$ and $\{a_s\}_{s\in [S]}$ need to be set correctly. It is an interesting open question how one can achieve such dependence without additional prior knowledge (such as function similarity) about $\{f_v\}_{v\in V}$. 

\textbf{The hierarchical setting and function similarity at different scales.}
In addition, we provide results when function variations could be different along the span of $K_s$ for different $s\in [S]$. As an example, consider the hierarchical setting discussed in Section \ref{sec:agents}, with the additional assumption that for each non-leaf layer of the tree, all dual variables in that layer have the same number of child nodes. Then it can be shown that $\|K_s\|=\sigma_{\min}^+(K_s) = \sqrt{|\child(s)|/|\Descendant(s)|}$ (by \eqref{eq:K_prod} in the proof of Lemma \ref{lm:K_prop}), so the above bound $N_4$ can be simplified as
\begin{displaymath}
    N_4'=O( \frac{\overline{r}\cdot (\sum_{s=1}^S a_s)\cdot \sqrt{D^X}}{\epsilon }),\quad \rho_s\propto a_s\sqrt{|\Descendant(s)|}.
\end{displaymath}
As discussed in Section \ref{sec:func_simi}, $a_s$ measures the function variation along the span of $K_s$, i.e. variation in $\{f_v\}_{v\in \Descendant(s)}$ \textit{not taken care of} by $\agent(y_{s'})$ in the subtree rooted at $\agent(y_s)$. In addition, \eqref{eq:var_interp} shows that $a_s^2 =  |\Descendant(s)| \cdot \sup_{x\in \mc X}\var_{i\sim \mu_s}(\overline{f'}_i(x))$, and so $\rho_s\propto |\Descendant(s)| \cdot \sqrt{\sup_{x\in \mc X}\var_{i\sim \mu_s}(\overline{f'}_i(x))}$. 

Thus, from the cost-minimization perspective in the discussion in Section \ref{sec:convergence-convex}, denoting the cost of updating $y_s$ as $c_s$, one should choose $r_s\propto\sqrt{\frac{c_s/|\Descendant(s)|}{\sqrt{\sup_{x\in \mc X}\var_{i\sim \mu_s}(\overline{f'}_i(x))}}}$. This corroborates the intuition that if along some $K_s$ the function does not vary by too much ($\var_{i\sim \mu_s}(\overline{f'}_i)$ is small), then $\agent(y_s)$ does not need to update $y_s$ very frequently (can use larger $r_s$).

\subsection{Convergence for strongly convex objectives}\label{sec:do-convergence-acc}

\textbf{Good initialization for \eqref{eq:objective_F}}. 
In Corollary \ref{cor:acc_convergence}, assuming that $\mc X$ is compact, then one can always use $D^X \geq \sup_{X\in {\mc X}} D_{w^X}(X,X^{init})$, suggesting that $X^{init}$ should be chosen as the ``center'' of ${\mc X}$, and $D^X$ measures the (squared) radius of ${\mc X}$. The resulting $N$, then, depends on $D^X$. However, such dependence on the size of ${\mc X}$ could be suboptimal, especially when local objectives are similar. Indeed, in the extreme case where all local functions are the same, then primal agents can optimize their local objectives without communication at all.

To take advantage of potential similarities in the local functions, we propose initializing the primal variables at (approximate) local optimal solutions, which has the following guarantee on $D(X^*,X^{init})$.
\begin{lemma}\label{lm:good_init}
   Assume that all norms are the $l_2$ norm, and for some $\epsilon_0\geq 0$, $\widehat{X} = (\widehat{x}_v)_{v\in V} \in {\mc X}$ satisfies the following condition
      \begin{displaymath}
       F(\widehat{X}) \leq \min_{X\in {\mc X}} F(X) + \epsilon_0.
   \end{displaymath}
   Assume that \eqref{eq:prop-f} holds for some $\mu>0$ and $\{f_v\}_{v\in V}$ is $\{(a_s,K_s)\}_{s\in [S]}$-similar, then 
   \begin{displaymath}
    \|\widehat{X} - X^*\| \leq \frac{(\sum_{s=1}^S a_s^2)^{1/2}}{\mu} + \sqrt{\frac{\sum_{s=1}^S a_s^2}{\mu^2} + \frac{2\epsilon_0}{\mu}}.
\end{displaymath}
\end{lemma}
\begin{proof}[Proof of Lemma \ref{lm:good_init}]
    By the suboptimality condition for $\widehat{X}$ and \eqref{eq:prop-f}, we get
    \begin{displaymath}
        \frac{\mu}{2}\|\widehat{X} - X^*\|^2\leq F(\widehat{X}) - F(X^*) - \langle \nabla F(X^*),\widehat{X} - X^*\rangle \leq - \langle \nabla F(X^*),\widehat{X} - X^*\rangle + \epsilon_0. 
    \end{displaymath}
    Notice that by the first-order optimality condition of $X^*$, we get
    \begin{displaymath}
        \langle \nabla F(X^*), \Pi (\widehat{X} -X^*)\rangle \geq 0.
    \end{displaymath}
    Combining the above two results, we get 
    \begin{align*}
        \frac{\mu}{2}\|\widehat{X} - X^*\|^2&\leq - \langle \nabla F(X^*),(I-\Pi)(\widehat{X} - X^*) \rangle + \epsilon_0\\
        &=- \sum_{s=1}^S \langle \Pi_s\nabla F(X^*),\Pi_s(\widehat{X} - X^*)\rangle + \epsilon_0\\
        &\leq \sum_{s=1}^S \|\Pi_s\nabla F(X^*)\|_*\cdot \|\Pi_s(\widehat{X} - X^*) \| + \epsilon_0\\
        &\leq (\sum_{s=1}^S \|\Pi_s\nabla F(X^*)\|^2_*)^{1/2}\cdot (\sum_{s=1}^S\|\Pi_s(\widehat{X} - X^*) \|^2)^{1/2} + \epsilon_0\\
        &\leq (\sum_{s=1}^S a_s^2)^{1/2}\cdot \|\widehat{X} - X^*\| + \epsilon_0,
    \end{align*}
where the last $\leq$ is because of the assumption that $\{f_v\}_{v\in V}$ are $\{(a_s,K_s)\}_{s\in [S]}$-similar, and all norms are $l_2$ norm. The above inequality is quadratic in $\|\widehat{X} - X^*\|$, and the result follows. 
\end{proof}

The above Lemma \ref{lm:good_init} shows that if $\{\widehat{x}_v\}_{v\in V}$ are all approximately optimal to local objectives, then $D(X^*,\widehat{X})\sim \frac{\sum_{s=1}^S a_s^2}{\mu^2}$. To find such initialization, one can apply the GS procedure.

\begin{corollary}\label{cor:good_init_CS}
Assume that all norms are $l_2$ norms and $w_x(x)=  \frac{1}{2}\|x\|^2$, that \eqref{eq:prop-f} holds with some $\mu>0$. For each $v\in V$, assume that $\agent(x_v)$ is given some $\underline{x}_v^0\in \mc X$ such that $\sup_{x\in \mc X} D(x,\underline{x}_v^0)\leq \underline{D}^x<\infty$
    \begin{displaymath}
    (\_,x_v^{init}) = GS(f_v,\mc X, D,\underline{T},\underline{\eta},\mb 0,\underline{x}_v^0,\underline{x}_v^0), 
\end{displaymath}
where the GS procedure uses $\lambda_t,\beta_t$ according to Corollary \ref{cor:lambda_beta} (for $\mu>0$), then with $\epsilon_0 = \widetilde{a}^2/\mu$, $\underline{T} \geq \frac{8CM_f^2m}{\epsilon_0\mu}$ and $\underline{\eta} = \frac{\epsilon_0/2}{m\underline{D}^x}$
% \begin{displaymath}
%     \|X^* - X^{init}\|\leq \frac{\widetilde{a}}{\mu} + \sqrt{\frac{2 m\epsilon_0}{\mu}},
% \end{displaymath}
\begin{displaymath}
    D(X^*,X^{init}) \leq \frac{4\widetilde{a}^2}{\mu^2} ,
\end{displaymath}
where $\widetilde{a} = \widehat{a}_1$ if $\{f_v\}_{v\in V}$ is $\widehat{a}_1$-similar for some $\widehat{a}_1>0$, and $\widetilde{a} = (\sum_{s=1}^S a_s^2)^{1/2}$ if $\{f_v\}_{v\in V}$ is $\{(a_s,K_s)\}_{s\in [S]}$-similar such that $(\sum_{s=1}^S a_s^2)^{1/2}>0$. 
\end{corollary}

\textbf{Complexities for \eqref{eq:objective_f}}. Combining Theorem \ref{cor:main_acc}, Corollary \ref{cor:good_init_CS}, Lemma \ref{lm:delta-subopt}, and Corollary \ref{cor:delta_subopt}, we get the following results.

\begin{corollary}\label{cor:main1_acc}
Assume that all norms are the $l_2$ norm, and take $y_s^0 = \mb 0$, $w_{y_s}(y_s) = \frac{1}{2}\|y_s\|^2$, $w_x(x) = \frac{1}{2}\|x\|^2$, and $w_X(X) = \sum_{v\in V} w_x(x_v)$. Assume that conditions of Corollaries \ref{cor:main_acc} and \ref{cor:good_init_CS} hold and $X^{init}$ is initialized according to Corollary \ref{cor:good_init_CS}, and the following holds for $A_0,A_1$ specified below
\begin{displaymath}
     N \geq \frac{2\overline{r}A}{\sqrt{\mu_f\epsilon} },\quad A = \sqrt{\overline{r^3}/(\overline{r})^3 + 7(\overline{r^2}/(\overline{r})^2)^2}\cdot A_1 +A_0.
\end{displaymath}

\begin{enumerate}
    \item If $\{f_v\}_{v\in V}$ is $\widehat{a}_1$-similar, then take $A_1 = \widehat{a}_1$ and $\rho_s =\frac{ \|K_s\|}{\sum_{s'=1}^S\|K_{s'}\|}$
    \begin{enumerate}
    \item $\frac{\sum_{k=0}^N \theta_k\widehat{X}^k}{\sum_{k=0}^N \theta_{k}}$ is an $(\epsilon,\epsilon/\xi)$-solution if $R_s=\widehat{R}_s^{ccv}$ as defined in \eqref{eq:r_1} (then $\sqrt{2D_s^y}= \frac{\xi + \widehat{a}_1}{\sigma_{\min}^+(K)}$) and $A_0 = \frac{(\sum_{s=1}^S\|K_s\|)}{\sigma_{\min}^+(K)}\cdot (\xi + \widehat{a}_1)$; 
\item $\Pi(\frac{\sum_{k=0}^N \theta_k\widehat{X}^k}{\sum_{k=0}^N \theta_{k}})$ is an $(\epsilon(1+1/\xi),0)$-solution if $R_s = \widehat{R}_s^{prj}$ as defined in \eqref{eq:r_2} (then $ \sqrt{2D_s^y}=\frac{(1+\xi) \widehat{a}_1}{\sigma_{\min}^+(K)}$) and $A_0= \frac{(1+\xi) (\sum_{s=1}^S \|K_s\|)\cdot \widehat{a}_1}{\sigma_{\min}^+(K)}$.
\end{enumerate}
\item If $\{f_v\}_{v\in V}$ is $\{(a_s,K_s)\}_{s\in [S]}$-similar where $a_s>0$ for all $s$, then take $A_1 = (\sum_{s=1}^S a_s^2)^{1/2}$,
\begin{enumerate}
    \item $\frac{\sum_{k=0}^N \theta_k\widehat{X}^k}{\sum_{k=0}^N \theta_{k}}$ is an $(\epsilon,\epsilon/\xi)$-solution if $R_s = R_s^{ccv}$ as defined in \eqref{eq:r_s_1} (then $\sqrt{2D_s^y}=\frac{\xi + a_s}{\sigma_{\min}^+(K_s)}$), $\rho_s = (\frac{\xi + a_s}{\sigma^+_{\min}(K_s)})/(\sum_{s'=1}^S \frac{\xi + a_{s'}}{\sigma^+_{\min}(K_{s'})^+})$, and $A_0=\sum_{s=1}^S (\xi + a_s)\cdot \frac{\|K_s\|}{\sigma_{\min}^+(K_s)}$; 
\item $\Pi(\frac{\sum_{k=0}^N \theta_k\widehat{X}^k}{\sum_{k=0}^N \theta_{k}})$ is an $(\epsilon(1+1/\xi),0)$-solution if $R_s = R_s^{prj}$ satisfies \eqref{eq:r_s_2} (then $\sqrt{2D_s^y}=\frac{(1+\xi )a_s}{\sigma_{\min}^+(K_s)}$), $\rho_s = (\frac{ a_s}{\sigma^+_{\min}(K_s)})/(\sum_{s'=1}^S \frac{ a_{s'}}{\sigma^+_{\min}(K_{s'})^+})$, and $A_0=(1+\xi)(\sum_{s=1}^S a_s\cdot \frac{\|K_s\|}{\sigma_{\min}^+(K_s)})$. 
\end{enumerate}
\end{enumerate}
\end{corollary}

\textbf{Subgradient oracle complexities.} With the initialization in Corollary \ref{cor:good_init_CS} and $C = 1$, the number of subgradient steps needed to find $X^{init}$ is $\underline{T}\geq \frac{8mM_f^2}{\widetilde{a}^2}$, constant in $\epsilon$.

In addition, in Corollary \ref{cor:acc_convergence}, we can take $D_0 = \frac{4\widetilde{a}^2}{\mu_f^2}$, and so $D_1 = \frac{2\widetilde{a}^2(\overline{r^2}/\overline{r})^2}{M_f^2}$. Thus, the requirement on $T$ becomes $T/N\geq \max(\frac{5}{\sqrt{D_1}},\frac{64\overline{r}}{D_1})$, i.e. $T/N = \Omega(\max(\frac{M_f/\widetilde{a}}{\overline{r^2}/\overline{r}},(\frac{M_f/\widetilde{a}}{\overline{r^2}/\overline{r}})^2\cdot \overline{r}))$, and so the total subgradient steps needed (for each agent) is
\begin{displaymath}
    N^2\cdot O(\max(\frac{M_f/\widetilde{a}}{\overline{r^2}/\overline{r}},(\frac{M_f/\widetilde{a}}{\overline{r^2}/\overline{r}})^2\cdot \overline{r})) = O(\frac{\overline{r}^2A^2}{\mu_f \epsilon}\cdot \max(\frac{M_f/\widetilde{a}}{\overline{r^2}/\overline{r}},(\frac{M_f/\widetilde{a}}{\overline{r^2}/\overline{r}})^2\cdot \overline{r}))
\end{displaymath}
In the special case where $S = 1$ and $\|K\| = O(\sigma_{\min}^+(K))$, we have $\widetilde{a} = \Omega(A)$. Further assuming that $A = O(\sqrt{m}M_f)$ (which holds when $\|f'_v\|\leq M_f$ for all $v$), the above can be simplified as $\frac{\overline{r}\sqrt{m}M_f^2}{\mu_f \epsilon}$.

\textbf{Communication rounds complexities.} From Corollary \ref{cor:main1_acc}, the communication rounds needed is $N = O(\frac{\overline{r}A}{\sqrt{\mu_f \epsilon}})$, where $A$ depends on function similarities and higher moments of $\{r_s\}_{s\in [S]}$. In terms of $N$ and $T$'s dependency on $\epsilon,\mu_f$, $O(1/\sqrt{\mu_f\epsilon})$ communication rounds and $O(1/(\mu_f\epsilon))$ gradient steps are needed.

Now, consider the special case where $\overline{r^2}=O((\overline{r})^2)$ and $\overline{r^3}=O((\overline{r})^3)$, i.e. $r_s$ has small variation, then $\sqrt{\overline{r^3}/(\overline{r})^3 + 7(\overline{r^2}/(\overline{r})^2)^2} = O(1)$, and so $A = O(\widetilde{a} + A_0)$, then $\xi=1$ with $\widehat{R}_s^{prj}$ and $R_s^{prj}$ give the following $N_1$ and $N_2$ respectively 
\begin{displaymath}
     N_1 = O\left(\frac{\overline{r}\widehat{a}_1}{\sqrt{\epsilon\mu_f}}\cdot \frac{\sum_{s=1}^S \|K_s\|}{\sigma_{\min}^+(K)}\right),\quad N_2 = O\left(\frac{\overline{r}}{\sqrt{\epsilon\mu_f}}\cdot \left(\sum_{s=1}^S a_s \cdot \frac{ \|K_s\|}{\sigma_{\min}^+(K_s)}\right)\right),
\end{displaymath}
both have linear dependence on $\overline{r}$ and function similarities. (For $N_2$, we use $(\sum_{s=1}^S a_s^2)^{1/2} \leq \sum_{s=1}^S a_s$ and $\|K_s\|\geq \sigma_{\min}^+(K_s)$.)

\textbf{Comparison with communication lower bounds.} When $S = 1$ (and so $\overline{r^2} = (\overline{r})^2$), $K = I-\Pi$ (and so $\|K\| = \sigma_{\min}^+(K)$), with $w_{x}(x) = \frac{1}{2}\|x\|^2$ and $w_{y_s}(y_s) = \frac{1}{2}\|y_s\|^2$, assuming that $\|f_v'\|\leq M_f$, $a_1^2\leq m\gamma^2$ for some $\gamma\leq M_f$, we have $N = O(\frac{\overline{r}\cdot \sqrt{m}\gamma}{\sqrt{\epsilon \mu_f}})$. Since communication is only needed when $\agent(y_1)$ updates, the total number of communication rounds is $N/r_s = O(\frac{\sqrt{m}\gamma}{\sqrt{\epsilon \mu_f}})$, which achieves the theoretical lower bound (Theorem 2 and discussion after in \cite{ArjevaniShamir2015Communication}) on the communication round complexity for $\mu$-strongly convex, $\sqrt{m}\gamma$-similar functions, and so is optimal\footnote{\cite{ArjevaniShamir2015Communication} constructs a pair of ``chain like'' functions $\{\gamma F_1,\gamma F_2\}$ which are $\Theta(\gamma )$-similar and $\mu$-strongly convex. In addition, when $m/2$ agents are given $\gamma F_1$ and the rest are given $\gamma F_2$ (and so this set of $m$ functions is $\Theta(\sqrt{m}\gamma)$-similar), the number of rounds of communication needed is $\Omega(\gamma\sqrt{\frac{1}{\mu_f\epsilon/m}})$.}.

\section{Numerical experiments}\label{sec:exp}

Below, we present numerical experiments applying MT-PDHG to linear programming problems (Section \ref{sec:exp-LP}), and (A)MT-PDHG to distributed Support Vector Machine problems (Section \ref{sec:exp-DO}). All experiments are implemented using Python and run on MacBook Air with the M3 chip and 8 cores.

\subsection{Experiment: linear programming}\label{sec:exp-LP}
In this set of experiments, we apply our MT-PDHG and the vanilla PDHG to simulated linear programming problems of the form
\begin{equation*}
    \min_{X\in \R^n} c^TX,\quad s.t.~AX =b,~X\geq \mb 0,
\end{equation*}
where $A\in \R^{m\times n}$ and the rows are divided evenly into $S = 6$ blocks, with the associated dual blocks updated in parallel: 
\begin{equation*}
    \min_{X\in \R^n}\max_{Y = (y_s)_{s\in [S]}\in \R^m} c^TX - \sum_{s=1}^S (y_s^TA_sX-y_s^Tb),\quad s.t.~X\geq \mb 0.
\end{equation*}

\textbf{Problem simulation.} We simulate $c_i\sim \mc N(0,1)$ for $i\in [n]$, $A_{i,j}\sim \text{Uniform}([0,1))$ for $i\in [m],j\in [n]$, all independently. To ensure the problem is feasible, we simulate $X'_{i}\sim \text{Uniform}([0,1))$ for $i\in [n]$ independent of $c,A$, and take $b = AX'$.

\textbf{Choice of rates $r_s$.} We consider 4 combinations of the updating rates for the dual blocks: 1. $r_s=1$ for all $s$; 2. $r_1 = r_2 = r_3 = 1$ and $r_4 = r_5 = r_6 = 10$; 3. $r_s=10$ for all $s$; 4. $r_s=50$ for all $s$. These are denoted in different colors in Figures \ref{fig:LP-exp-iter} and \ref{fig:LP-exp-time}. 

\textbf{Algorithm setups.} We use $w_X(X) = \frac{1}{2}\|X\|^2$ and $w_{y_s}(y_s) = \frac{1}{2}\|y_s\|^2$. We benchmark our MT-PDHG against the vanilla PDHG, where at each global iteration $k$, $X$ is updated as the minimizer to $\langle c- A^T\overline{Y}^k,X\rangle + \frac{\eta}{2}\|X-X^{k-1}\|^2$, with $\eta = \|A\|$, and at $k = i_sr_s$, the dual block $y_s$ is updated as the minimizer to $\langle K_s\widetilde{X}_s^{i_s}-b,y_s\rangle+ \frac{\tau_s}{2}\|y_s-y_s^{i_s-1}\|^2$, where $\widetilde{X}_s^{i_s} = 2X^{k-1} - X^{k-2}$ and $\tau_s = 2S\|A_s\|^2/\eta$. Thus, there is no mixture of Bregman divergence and multi-timescale extrapolation. For our MT-PDHG, we use $\rho_s = 1/S$ and the same $\eta,\tau_s$ as above. We initialize $X = \mb 0$ and $Y = \mb 0$.

\textbf{Results.} We present the KKT residual $(\|AX-b\|^2 + \|[A^TY-c]_+\|^2 + [c^TX - b^TY]_+)^{1/2}$ for MT-PDHG and vanilla PDHG, under different combinations of $r_s$. 

In Figure \ref{fig:LP-exp-iter}, we present the residual as a function of global iteration. As can be seen, our MT-PDHG is stable under various combinations of updating rates. Interestingly, the vanilla PDHG, even without mixture of Bregman divergence and multi-timescale extrapolation, still converges when the rates are small (rate combinations 1, 2, and 3). However, as the yellow lines suggest, when the rates ($r_s=50$) are large, our explicit control through the Bregman divergence and the extrapolation helps stabilize the performance. Moreover, comparing the green and blue curves, which have the same $\max_s r_s = 10$ but different average $r_s$, which is $5.5$ for green but $10$ for blue, we see that smaller $\overline{r}$ indeed corresponds to faster convergence rate (as a function of iteration), demonstrating that our MT-PDHG are robust to extreme values in $r_s$. 

In Figure \ref{fig:LP-exp-time}, we rescale the $x$-axis of Figure \ref{fig:LP-exp-iter} by the total runtime of each configuration. Comparing MT-PDHG with vanilla PDHG, we observe an additional overhead arising from the computation of the Bregman mixture and multi-timescale extrapolation, which involves evaluating $\sum_{s}\rho_s X^{k-r_s}$ and $\sum_{i=1}^{r_s} X^{k-i}$. However, as $m$ and $n$ grow, this overhead becomes negligible since the runtime is increasingly dominated by matrix–vector multiplications. We also observe that, approximately, the wall-clock time required to reach a target accuracy scales with $\max_{s} r_s$ (instead of $\overline{r}$ for the iteration). We note, however, that this behavior may depend on implementation details and the underlying computing hardware.

\begin{figure}[htbp]
    \centering
    \begin{subfigure}[t]{0.3\textwidth}
        \centering
        \includegraphics[width = \textwidth]{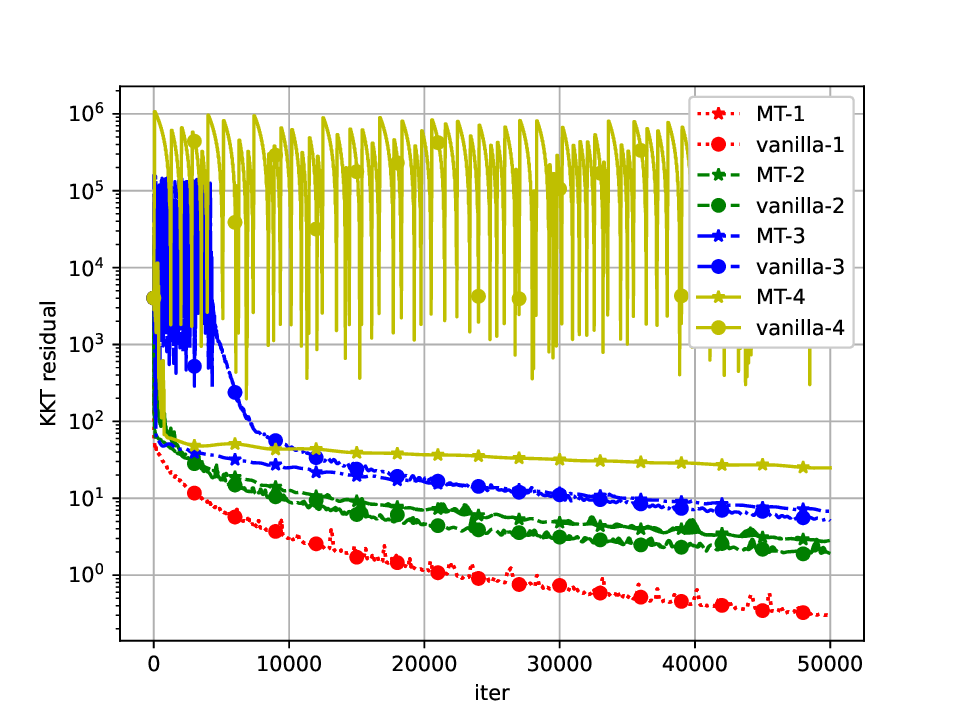}
        \caption{$(m,n) = (500,1000)$}
    \end{subfigure} 
    \begin{subfigure}[t]{0.3\textwidth}
        \centering
        \includegraphics[width = \textwidth]{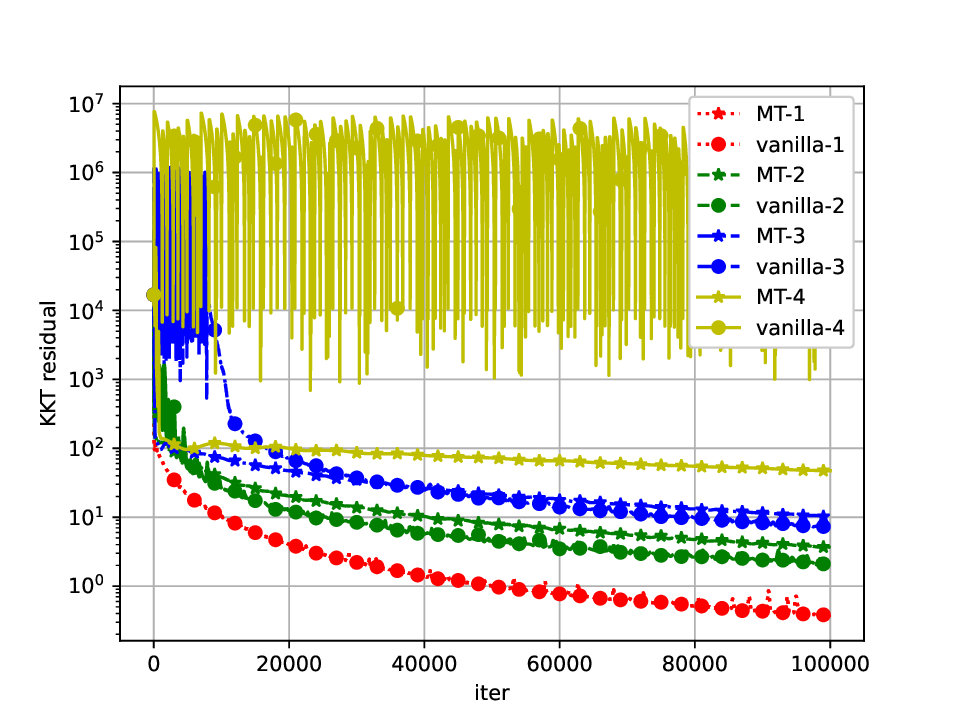}
        \caption{$(m,n) = (1000,3000)$}
    \end{subfigure}
    \begin{subfigure}[t]{0.3\textwidth}
        \centering
        \includegraphics[width = \textwidth]{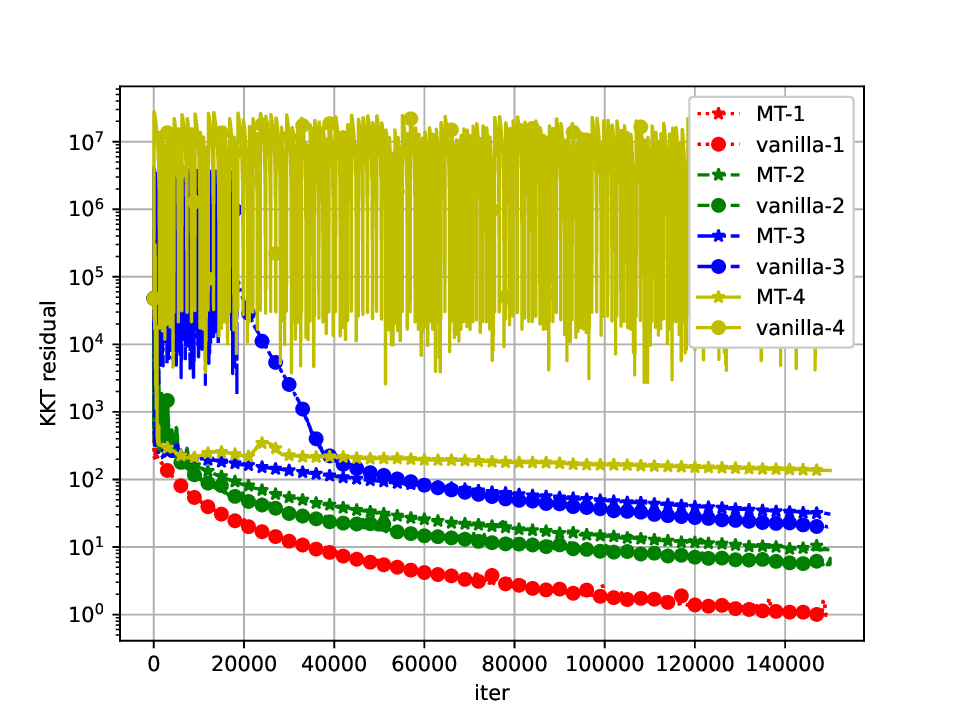}
        \caption{$(m,n) = (3000,5000)$}
    \end{subfigure}
        \caption{KKT residual as a function of global iteration for different $(m,n)$, under 4 different combinations of the updating rates for the dual blocks: 1. $r_s=1$ for all $s$; 2. $r_1 = r_2 = r_3 = 1$ and $r_4 = r_5 = r_6 = 10$; 3. $r_s=10$ for all $s$; 4. $r_s=50$ for all $s$.
        }\label{fig:LP-exp-iter}
\end{figure}

\begin{figure}[htbp]
    \centering
    \begin{subfigure}[t]{0.3\textwidth}
        \centering
        \includegraphics[width = \textwidth]{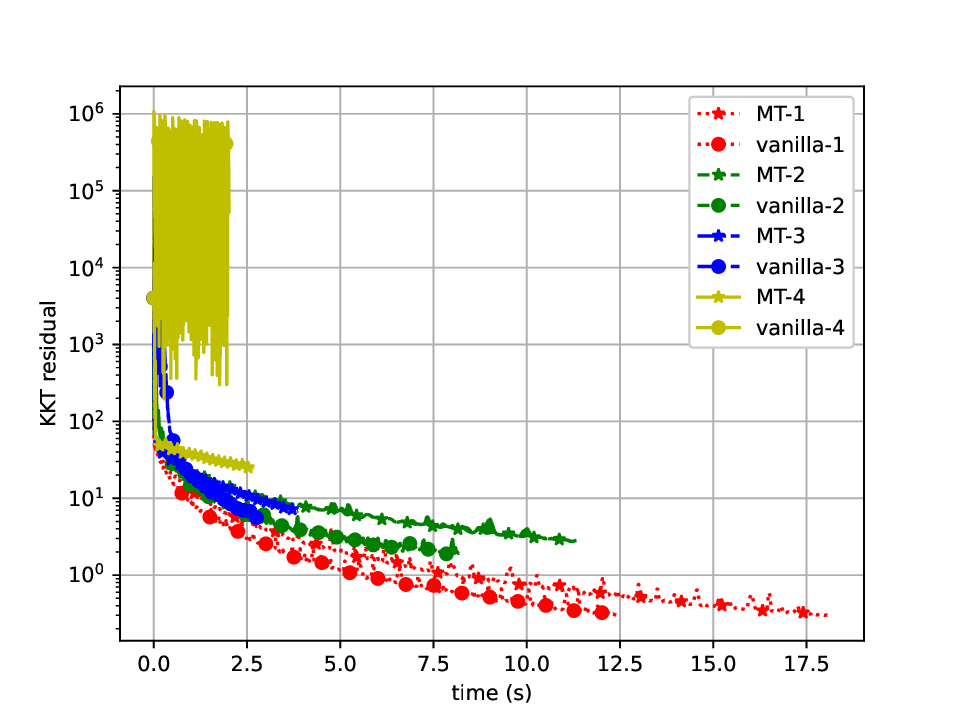}
        \caption{$(m,n) = (500,1000)$}
    \end{subfigure} 
    \begin{subfigure}[t]{0.3\textwidth}
        \centering
        \includegraphics[width = \textwidth]{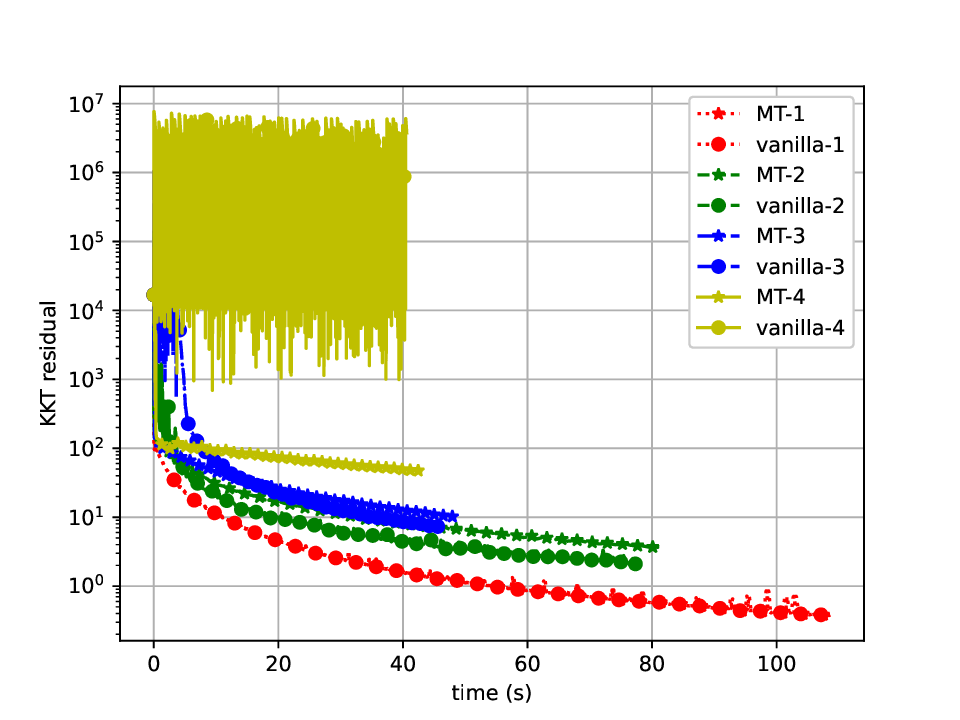}
        \caption{$(m,n) = (1000,3000)$}
    \end{subfigure}
    \begin{subfigure}[t]{0.3\textwidth}
        \centering
        \includegraphics[width = \textwidth]{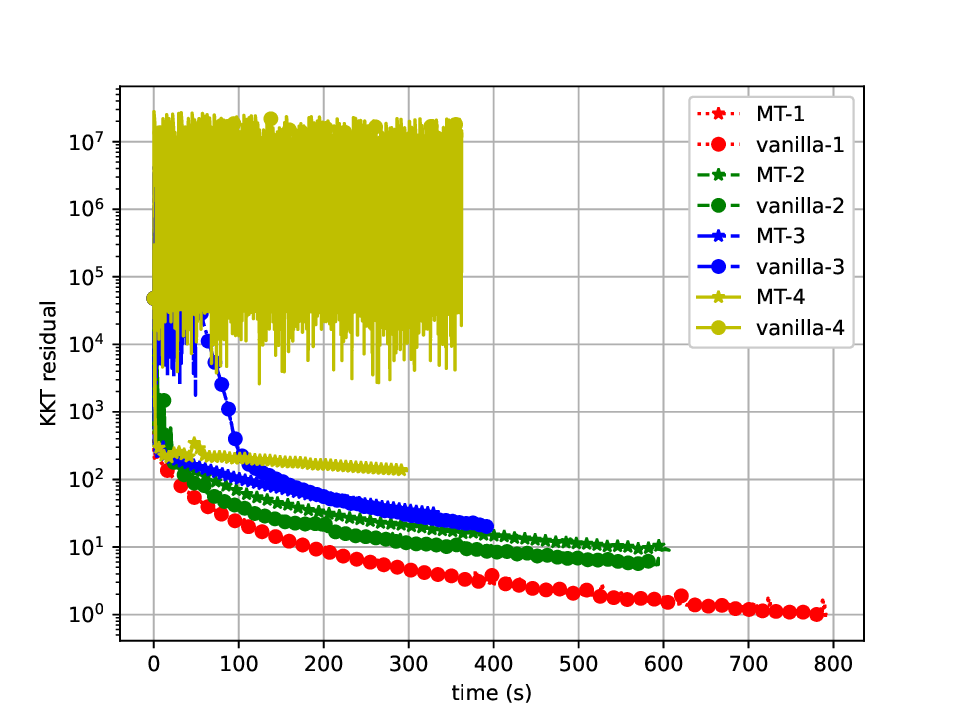}
        \caption{$(m,n) = (3000,5000)$}
    \end{subfigure}
        \caption{KKT residual as a function of running time in seconds for different $(m,n)$, under 4 different combinations of the updating rates for the dual blocks: 1. $r_s=1$ for all $s$; 2. $r_1 = r_2 = r_3 = 1$ and $r_4 = r_5 = r_6 = 10$; 3. $r_s=10$ for all $s$; 4. $r_s=50$ for all $s$.
        }\label{fig:LP-exp-time}
\end{figure}

\subsection{Experiment: distributed Support Vector Machine}\label{sec:exp-DO}

We consider the Support Vector Machine (SVM) problem with hinge loss and additional regularization. More precisely, each primal $\agent(x_v)$ is given $m_s$ pairs $\{(b_v^l,y_v^l)\}_{l\in [m_s]}$ such that $b_v^l\in \R^{\overline{d}}$ is a feature vector satisfying $\|b_v^l\|=1$, and $y_v^l\in \{\pm 1\}$ is the label. The goal of SVM is to find a weight vector $x\in \R^{\overline{d}}$ such that the linear classifier $b \to  \sign(\langle b,x\rangle)$ agrees with most pairs $(b_v^l,y_v^l)$ in the dataset. To achieve this, one common approach is to solve the following (regularized) hinge loss minimization problem (in a distributed fashion)\cite{lan_communication-efficient_2020,Duchi2012DA}: 
\begin{equation}\label{eq:svm}
    \min_{x\in \overline{\mc X}} \sum_{v\in V}f_v(x),\quad f_v(x) = \frac{1}{m_s}\sum_{l=1}^{m_s} [1-y_v^l\langle b_v^l,x\rangle]_+ + \frac{\mu}{2}\|x\|^2,~v\in V,
\end{equation}
In this experiment, we use the w8a dataset in LIBSVM \cite{CC01a}, which consists of $49749$ samples, and the feature dimension is $\overline{d} = 300$. We first normalize the features $\|b_v^l\| = 1$ for all data, and take $\overline{\mc X} = \{x\in \R^{\overline{d}}|\|x\|\leq 5\}$. When $\mu = 0$, \eqref{eq:svm} is the classical hinge loss minimization problem, and when $\mu>0$, the local objectives are $\mu$-strongly convex.

\textbf{Setup.} We take $w_x(x) = \frac{1}{2}\|x\|^2$ and $w_{y_s}(y_s) = \frac{1}{2}\|y_s\|^2$, and 
\begin{displaymath}
    f_v'(x) = -\frac{1}{m_s}\sum_{l=1}^{m_s} y_v^lb_v^l \cdot \mb 1[1>y_v^l\langle b_v^l,x\rangle] + \mu x.
\end{displaymath}
Since $\|b_v^l\|\leq 1$ and $\|x\|\leq 5$ (since $x\in \overline{\mc X}$), we have $\|f_v'\|\leq 1 + 5\mu$ and so we can take $M = 2(1 + 5\mu)$.

Below, we look at the suboptimality of $F(\Pi \underline{X}^k)$ where $\underline{X}^k:=\frac{\sum_{k'=0}^k \theta_{k'}\widehat{X}^{k'}}{\sum_{k'=0}^k \theta_{k'}}$.

\subsubsection{Suboptimality and $k, \overline{r}$}\label{sec:exp1}

In this experiment, we investigate the suboptimality of $F(\Pi \underline{X}^k)$ as a function of the iteration number $k$ and the mean updating rates for the dual $\overline{r}$.

\textbf{Communication setup.} We consider the hierarchical setup, with $3$ layers of dual agents and $1$ layer of primal agents, where each non-leaf node has $5$ child nodes. Thus, there are $m=125$ primal agents and $S = 31$ dual agents, and $|\child(s)|= 5$ for each non-leaf node. We assume that the dual agents at layer $i$ are updated with rate $r_i$ for $i=1,2,3$, and test with various $(r_1,r_2,r_3)$. We divide the data set evenly among the primal agents.

\textbf{Algorithm setup.} 
 Notice that the $\mu x$ part in $f'_v$ is common to all $v$ so we can take $a_s = 2\sqrt{|\Descendant(s)|}$ for all $s$. We use $R_s^{prj}$ as defined in \eqref{eq:r_s_2} with $\xi = 1$, and set $x_v^{init} = \mb 0$ and $y_s^{init} = \mb 0$. We test MT-PDHG for $N+1 = 3000$ and $\mu=0$. All parameters are set according to Theorem \ref{thm:main}. We test AMT-PDHG for $N+1 = 500$ and $\mu=0.01$. We set $T = N+1$ for simplicity, and set all other parameters according to Theorem \ref{thm:acc_convergence}.

\textbf{Results.} In Figures \ref{fig:exp1}, we present $F(\Pi \underline{X}^k)-F^*$ as a function of $k$ for MT-PDHG and AMT-PDHG respectively, where $F^*$ is the minimum value among all iterations of all algorithm configurations plus $0.001$ (as the $y$-axis is in the $\log$ scale). 

Different lines correspond to different $(r_1,r_2,r_3,\overline{r})$, with the line colors indicating $\overline{r}$, and the line with starred markers can serve as the benchmark: when all $r_s = 1$, our (A)MT-PDHG has the same updating rules as the classical PDHG. 

From the figures, we see that under all settings of $(r_1,r_2,r_3,\overline{r})$, our (A)MT-PDHG converge or show trend of convergence, and the convergence is faster for smaller $\overline{r}$. In addition, comparing the two figures in Figure \ref{fig:exp1}, we see that strong convexity (with AMT-PDHG) indeed accelerates the convergence. 

% \textbf{Results: $F(\Pi \underline{X}^k)$ and $\overline{r}$.} In Figure \ref{fig:exp1}(b), we present $F(\Pi \underline{X}^k)$ as a function of $\overline{r}$, taken at $k = 400,800,\ldots,2800$ for MT-PDHG. In Figure \ref{fig:exp1_acc}(b), we present $F(\Pi \underline{X}^k)$ as a function of $\overline{r}^2$, taken at $k = 60,120,\ldots,480$ for AMT-PDHG. In both figures, the line colors indicate $k$. As can be seen, as $k$ increases, the suboptimaltiy of $F(\Pi \underline{X}^k)$ is approximately linear in $\overline{r}$ for MT-PDHG and is approximately linear in $\overline{r}^2$ for AMT-PDHG, agreeing with our Theorem \ref{thm:main} and \ref{thm:acc_convergence}.

\begin{figure*}[h!]
    \centering
    \begin{subfigure}[t]{0.3\textwidth}
        \centering
        \includegraphics[width = \textwidth]{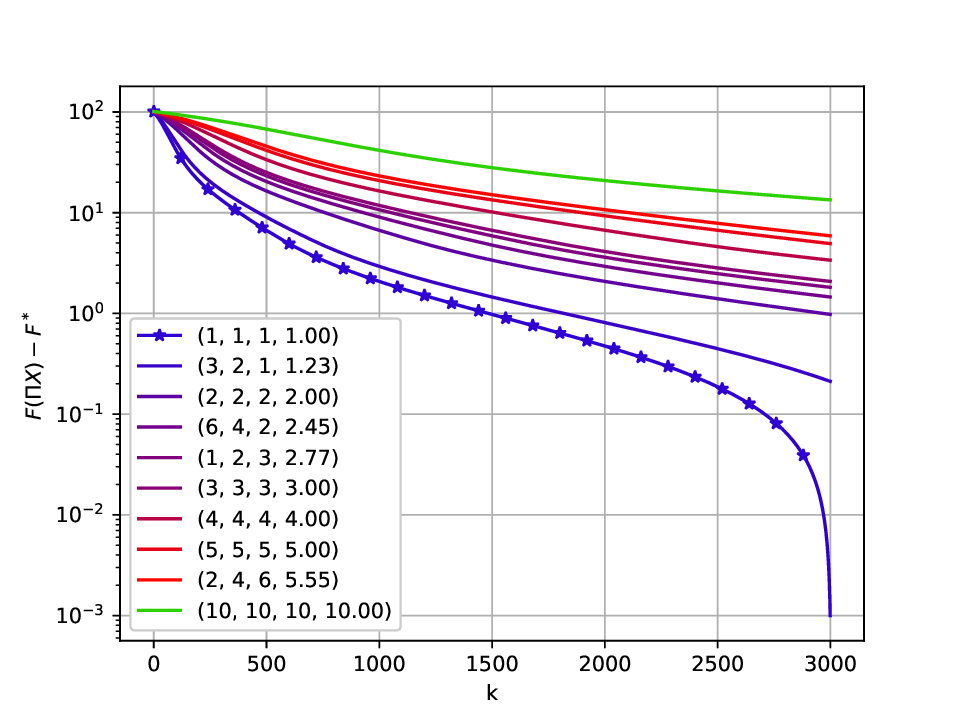}
    \end{subfigure}
    \hspace{0.1\textwidth}
    \begin{subfigure}[t]{0.3\textwidth}
        \centering
        \includegraphics[width = \textwidth]{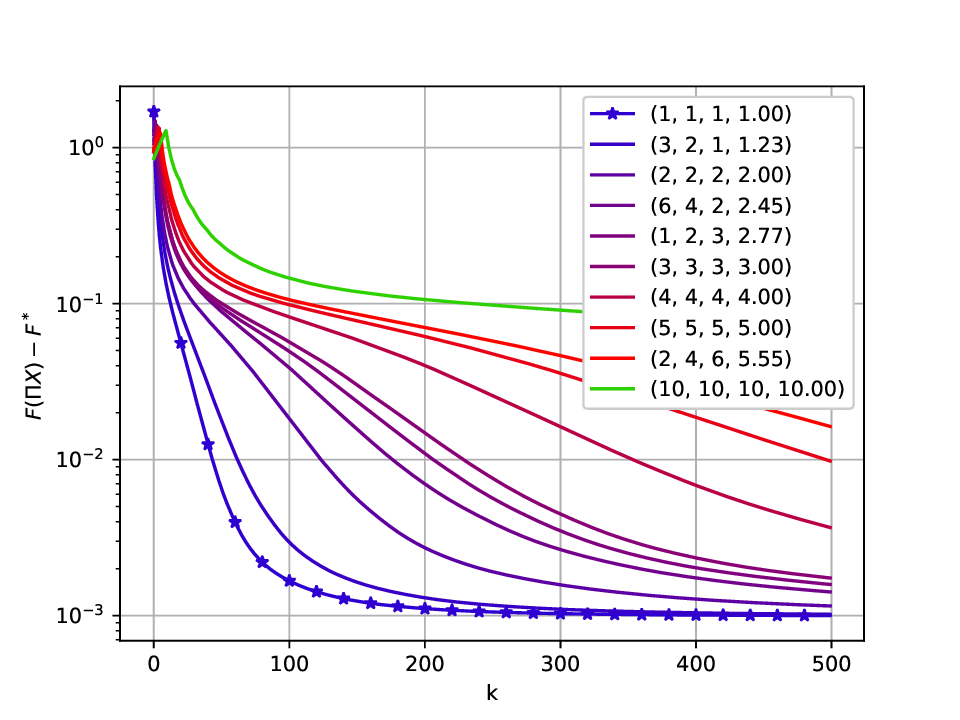}
    \end{subfigure}
    \caption{Dependence of $F(\Pi \underline{X}^k)$ on the iteration number $k$ and the mean updating rate $\overline{r}$ for MT-PDHG ($\mu=0$) (left) and AMT-PDHG ($\mu=0.01$) with communication sliding. Legends represent $(r_1,r_2,r_3,\overline{r})$ and line colors represent $\overline{r}$.}\label{fig:exp1}
\end{figure*}

\subsubsection{Suboptimality and costs of communication}
To further demonstrate the potential benefit of using different updating rates, we rescale the lines in the left plots of Figure \ref{fig:exp1} based on Amortized Costs ($AC$): since each of the $5^{i-1}$ dual agents at layer $i$ is updated every $r_i$ global iteration and each iteration costs $c_i$, the average cost per global iteration becomes $AC:=c_1/r_1 + 5c_2/r_2 + 25c_3/r_3$. In Figures \ref{fig:exp1_cost} and \ref{fig:exp1_acc_cost}, we rescale the $x$-axis in Figure \ref{fig:exp1} using $AC$ for each combinations of $r_s$, and the resulting plots reflect the objective values as a function of the costs. In addition, we use different markers to represent the ratio $r_1:r_2:r_3$: circle for $1:1:1$, triangle for $3:2:1$, and star for $1:2:3$.

From Figure \ref{fig:exp1_cost}, we see that the convergence rate for fixed $(c_1,c_2,c_3)$ depends on the ratios of $(r_1,r_2,r_3)$, as the lines with the same ratio (marker type) coincide. In addition, when the costs vary significantly, there is significant benefit of choosing the rates adaptive to the costs. For instance, in the middle two figures, with large $c_1$ or $c_2$, MT-PDHG is more cost-efficient when $r_1,r_2$ are large (as then dual agents in top two layers of the tree are updated less frequently). Indeed, lines with star markers show a faster rate of convergence. 

For AMT-PDHG, as suggested by our convergence results in Corollary \ref{cor:main1_acc}, the convergence rates have a more complicated dependency on $r_s$. From Figure \ref{fig:exp1_acc_cost}, the relative level of $(r_1,r_2,r_3)$ does not determine the convergence rates any more. However, the results still agree with the general intuition that with larger $c_1$ or $c_2$, it's more cost-efficient to choose relatively large $r_1$ or $r_2$.

\begin{figure*}[h!]
    \centering
    \begin{subfigure}[t]{0.23\textwidth}
        \centering
        \includegraphics[width = \textwidth]{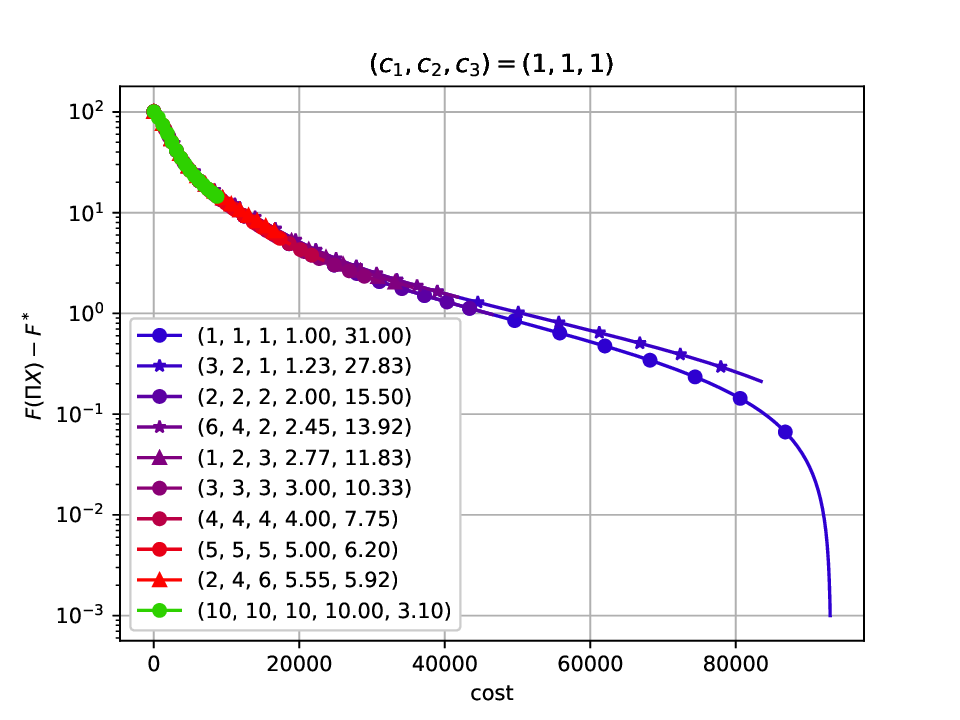}
    \end{subfigure}
    \begin{subfigure}[t]{0.23\textwidth}
        \centering
        \includegraphics[width = \textwidth]{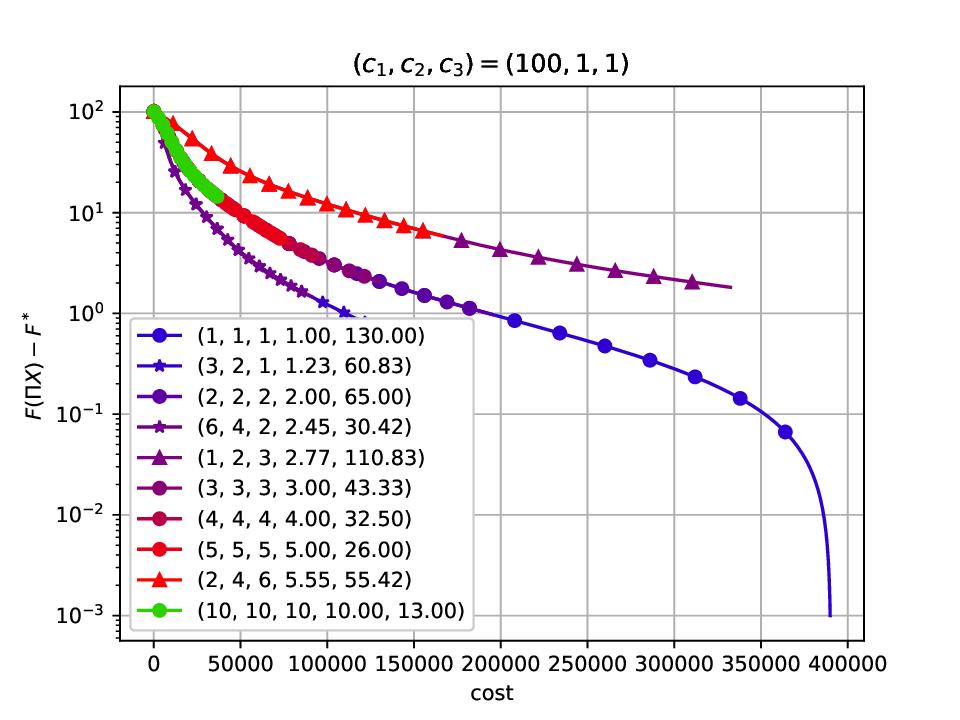}
    \end{subfigure}
    \begin{subfigure}[t]{0.23\textwidth}
        \centering
        \includegraphics[width = \textwidth]{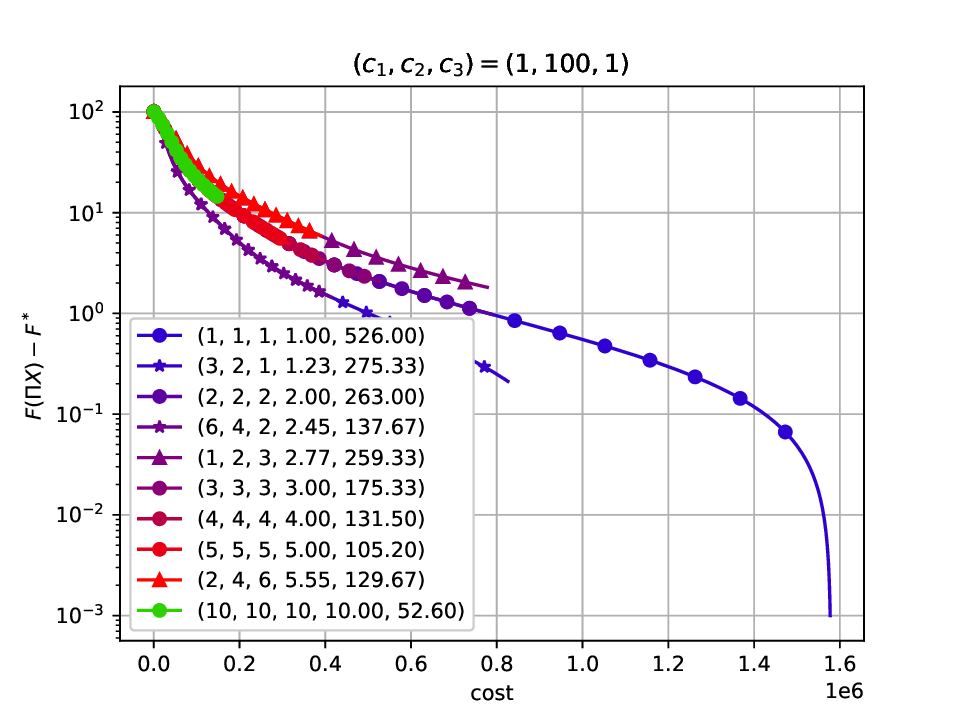}
    \end{subfigure}
    \begin{subfigure}[t]{0.23\textwidth}
        \centering
        \includegraphics[width = \textwidth]{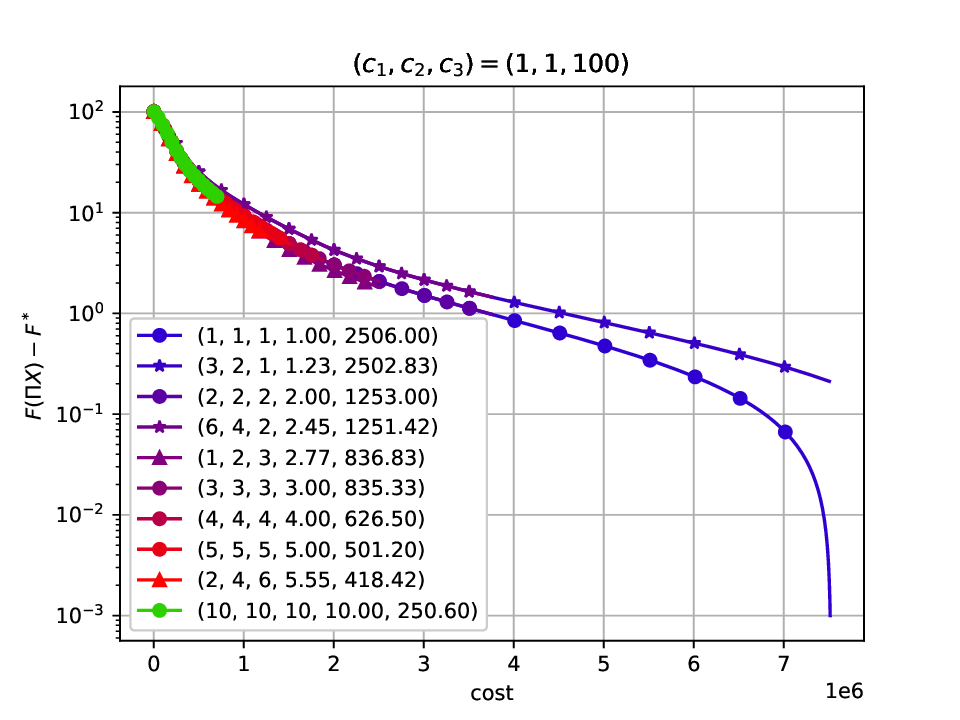}
    \end{subfigure}
    \caption{MT-PDHG ($\mu=0$), dependence of $F(\Pi \underline{X}^k)$ on the costs. Legend represents $(r_1,r_2,r_3,\overline{r},AC)$ where the amortized cost $AC:=c_1/r_1 + 5c_2/r_2 + 25c_3/r_3$. The marker indicates the ratio $r_1:r_2:r_3$: circle for $1:1:1$, triangle for $3:2:1$, and star for $1:2:3$. 
}\label{fig:exp1_cost}
\end{figure*}

\begin{figure*}[h!]
    \centering
    \begin{subfigure}[t]{0.23\textwidth}
        \centering
        \includegraphics[width = \textwidth]{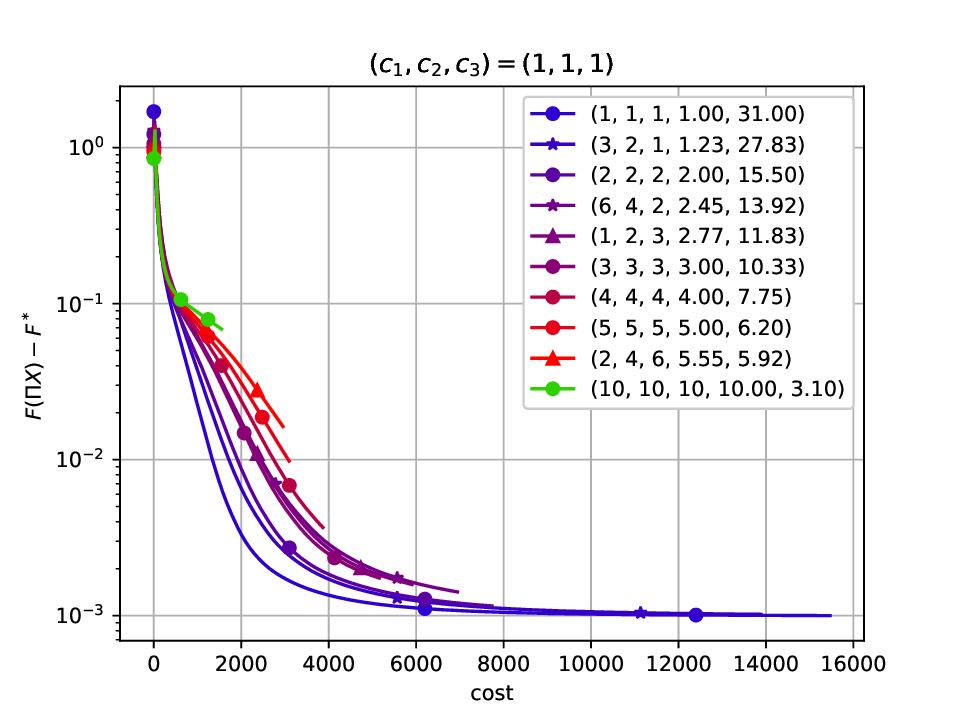}
    \end{subfigure}
    \begin{subfigure}[t]{0.23\textwidth}
        \centering
        \includegraphics[width = \textwidth]{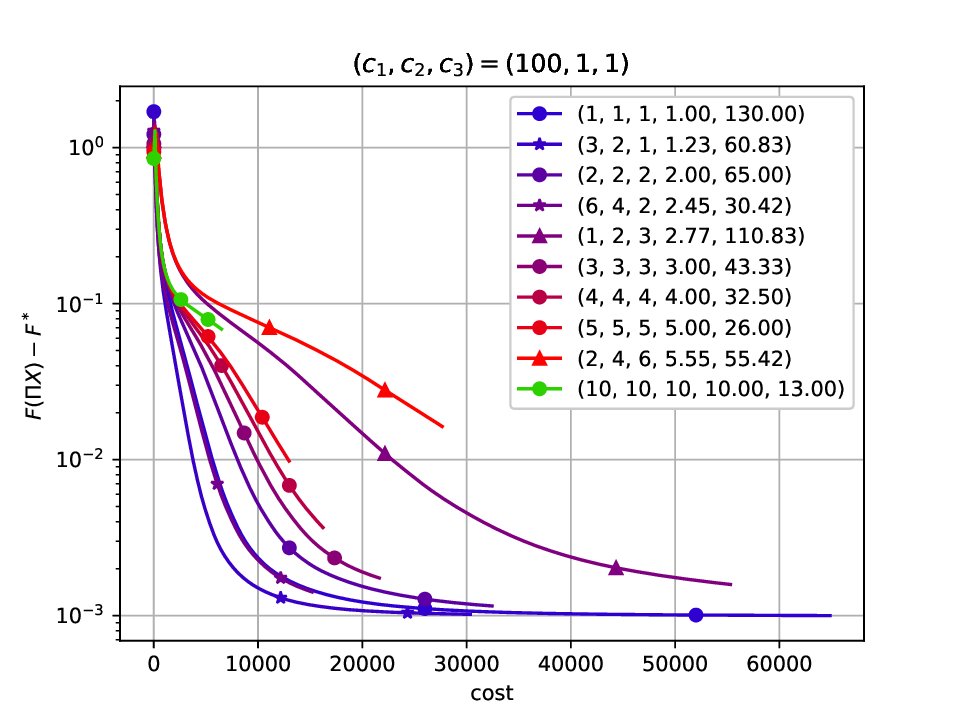}
    \end{subfigure}
    \begin{subfigure}[t]{0.23\textwidth}
        \centering
        \includegraphics[width = \textwidth]{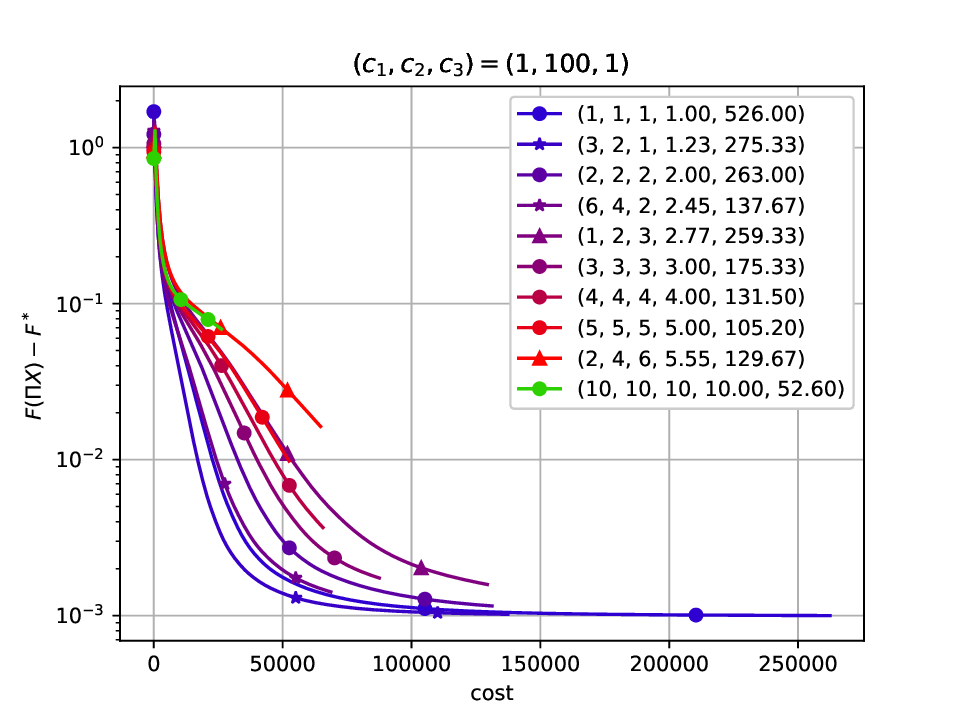}
    \end{subfigure}
    \begin{subfigure}[t]{0.23\textwidth}
        \centering
        \includegraphics[width = \textwidth]{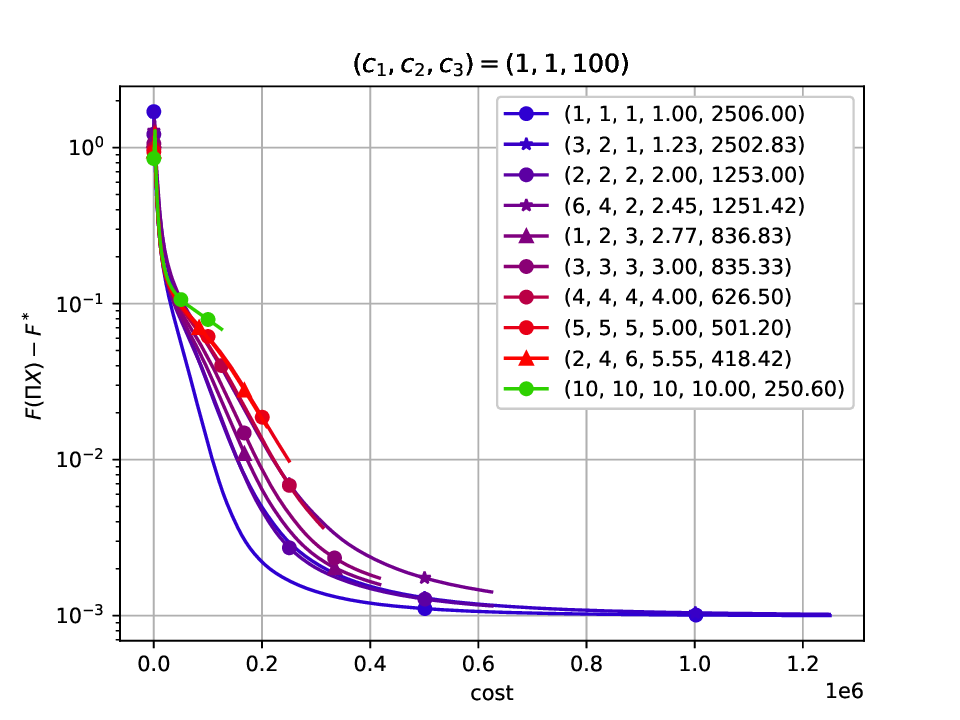}
    \end{subfigure}
    \caption{AMT-PDHG ($\mu=0.01$), dependence of $F(\Pi \underline{X}^k)$ on the costs. Legend represents $(r_1,r_2,r_3,\overline{r},AC)$ where the amortized cost $AC:=c_1/r_1 + 5c_2/r_2 + 25c_3/r_3$. The marker indicates the ratio $r_1:r_2:r_3$: circle for $1:1:1$, triangle for $3:2:1$, and star for $1:2:3$. }\label{fig:exp1_acc_cost}
\end{figure*}

\subsubsection{Suboptimality and $k,A$}\label{sec:exp2}
In this experiment, we use $1\%$ of the w8a dataset. We assume that there is only one dual agent (thus $A = \Theta(a_1)$) who updates at rate $r_1=1$, and there are $m=10$ primal agents. We focus on (normalized) $F(\Pi\underline{X}^k)$ as a function of $k$ and function similarities $a_1$.

\textbf{Dataset induced function similarities.} In this experiment, similarities between $\{f'_v\}_{v\in V}$ are inherited from similarities in the local datasets. More precisely, there is a global dataset $\{(b_{global}^l,y_{global}^l)\}_{l\in [m_{global}]}$ consisting of $m_{global}$ pairs of data. In addition, each agent has $m_{local}$ pairs of private data $\{(b_{local,v}^l,y_{local,v}^l)\}_{l\in [m_{local}]}$. Thus, $\agent(x_v)$ has access to $\{(b_{global}^l,y_{global}^l)\}_{l\in [m_{global}]}\cup\{(b_{local,v}^l,y_{local,v}^l)\}_{l\in [m_{local}]} $, a total of $m_s = m_{global} + m_{local}$ pairs of data.

More concretely, the global dataset consists of $\frac{1-\gamma}{1+(m-1)\gamma}$ fraction of the data, and the rest of the data is divided evenly among $m$ primal agents as local data. Thus $m_{local} \approx \gamma m_s  $. We test $\gamma = 0.1,0.2,\ldots,0.8$. Due to the global dataset, we have $\|f'_v-f'_{v'}\|\leq \frac{m_{local}}{m_s}$ for any $v,v'\in V$, and so we take $a_1 = 2\gamma \sqrt{m}$.

\textbf{Algorithm setup.} We consider two setups.
\begin{enumerate}
    \item Type-0 setup. We use $R_s = R_s^{prj}$ as defined in \eqref{eq:r_s_2} with $\xi = 1$, and $y_s^{init} = \mb 0$. We set the parameter $T=N+1$ and all other parameters are set according to Theorems \ref{thm:main} and \ref{thm:acc_convergence}. For the initialization, for MT-PDHG, we use $x_v^{init} = \mb 0$ and for AMT-PDHG, we use $\underline{x}_v^0 = \mb 0$ and construct $x_v^{init}$ according to Corollary \ref{cor:good_init_CS}. 
    \item Type-1 setup. In addition to the above $\gamma$-aware setup, we also test our algorithms for $R_s = 10000R_s^{prj}$, $a_1 = 2 \sqrt{m}$, and $x_v^{init} = \mb 0$, which we denote as type-1 setup. Compared to type-0, type-1 has larger dual domain size and ignores the function similarities. Thus, it can serve as an approximation to the DCS algorithms in \cite{lan_communication-efficient_2020}. 
\end{enumerate}
For both types of setups, we test MT-PDHG for $N+1 = 500$ and $\mu = 0$ and AMT-PDHG for $N+1 = 200$ and $\mu = 0.01$. Since the datasets are different for different $\gamma$, below, for each $\gamma$, we normalize $F(\Pi \underline{X}^k)$ such that $F(\mb 0) = m=500$ is normalized to $1$, and the minimum (over $k$ and two types) of $F(\Pi \underline{X}^k)$ is normalized to $0$.

\textbf{Results.} 
 In Figure \ref{fig:exp2}, we present the normalized $F(\Pi \underline{X}^k)$ as a function of $k$ for MT-PDHG and AMT-PDHG respectively. Different lines correspond to different types of setup and different $\gamma$, with the line colors indicating $\gamma$. As can be seen, our MT-PDHG and AMT-PDHG converge in all the tested settings, and strong convexity (with AMT-PDHG) accelerates the convergence. 
 
Moreover, for MT-PDHG, from the left figure in Figure \ref{fig:exp2} solid curves -- representing $\gamma$-aware setup -- are converging faster than the dotted curves, In addition, as $\gamma$ decreases -- global dataset takes a larger fraction -- the algorithm converges faster. These demonstrate that our proposed function similarity dependent penalties indeed take advantage of the function similarities to speed up the convergence. 

For the AMT-PDHG, from the right figure in Figure \ref{fig:exp2}, one can see that similarity helps speed up the convergence for both types of setups. We leave it to future works to investigate if the conditions on the penalty levels are necessary for function-similarity dependent convergence rates when the local objectives are strongly convex.

% \textbf{Results: normalized $F(\Pi \underline{X}^k)$ and $\gamma$.} In Figure \ref{fig:exp2}(b), we present normalized $F(\Pi \underline{X}^k)$ as a function of $\overline{\gamma}$, taken at $k = 100,200,300,400$ for MT-PDHG. The line colors indicate $k$. As can be seen, the normalized $F(\Pi \underline{X}^k)$ is approximately linear in $\gamma$ for MT-PDHG, which agrees with our Theorem \ref{thm:main}.

% In Figure \ref{fig:exp2_acc}(b), we present normalized $F(\Pi \underline{X}^k)$ as a function of $\gamma^2$, taken at $k = 10,15,20,25,30,40,80,120,160$ for AMT-PDHG. The line colors indicate $k$. It appears that the normalized $F(\Pi \underline{X}^k)$ is increasing in $\gamma^2$ for most $\gamma$ (the dip when $\gamma^2=0.36$ could be due to the randomness in the simulated dataset). This confirms that our AMT-PDHG can take advantage of function similarities. However, the normalized $F(\Pi \underline{X}^k)$ does not appear to be linear in $\gamma^2$: for small $\gamma^2$, it does not converge to $0$, which could be because setting $T = N+1$ (smaller than suggested)  introduces additional suboptimality; in addition, it is possible that the normalization process introduces extra $\gamma$-dependent factors. 

\begin{figure*}[h!]
    \centering
    \begin{subfigure}[t]{0.3\textwidth}
        \centering
        \includegraphics[width = \textwidth]{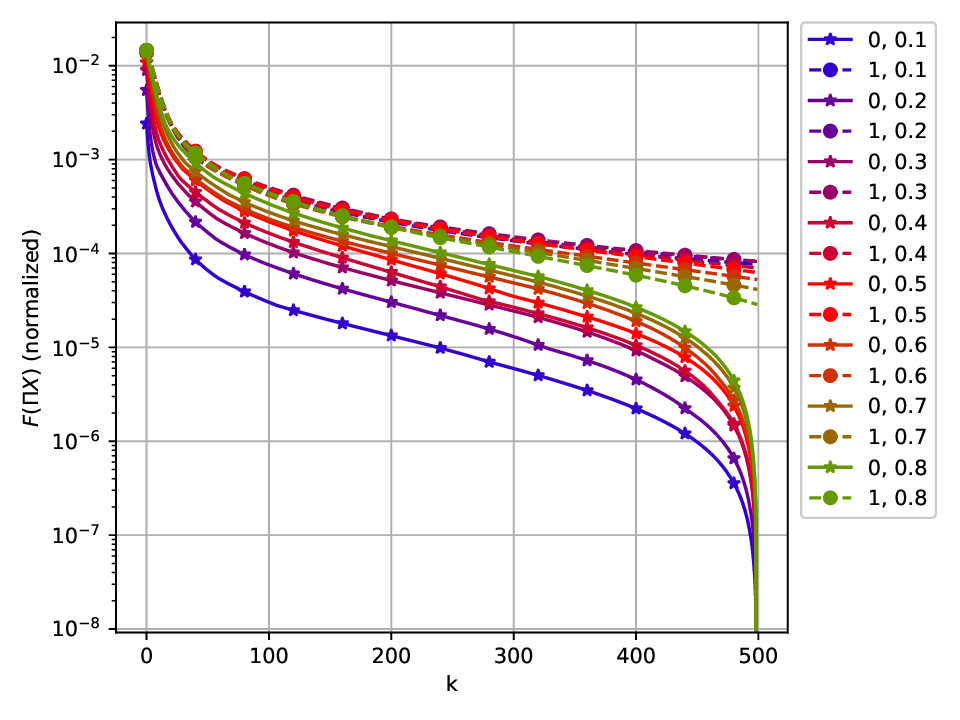}
    \end{subfigure}
    \hspace{0.1\textwidth}
    \begin{subfigure}[t]{0.3\textwidth}
        \centering
        \includegraphics[width = \textwidth]{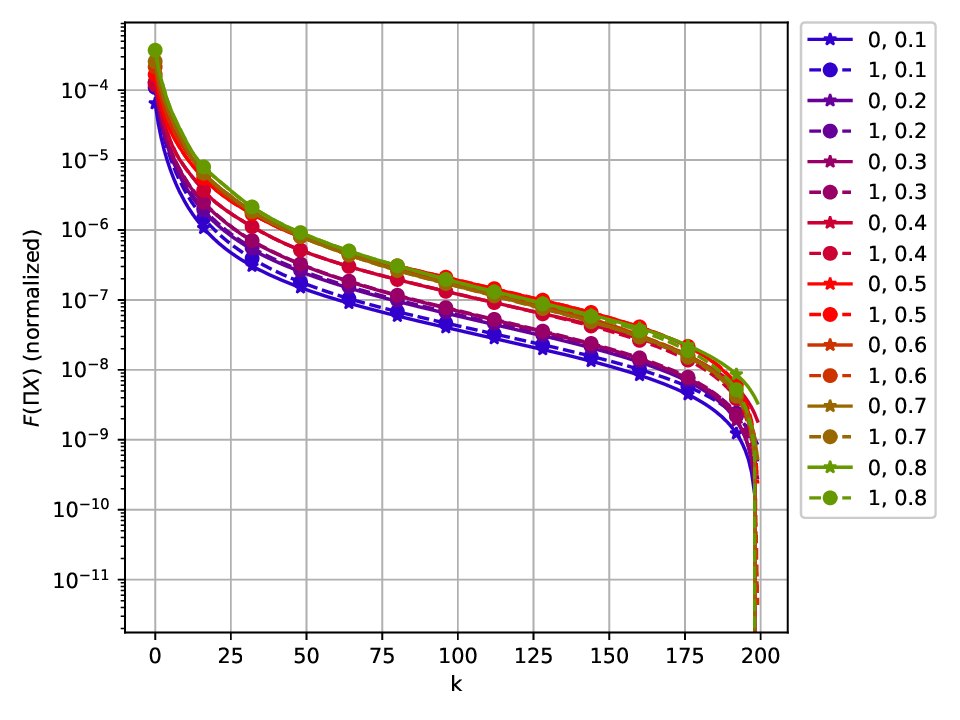}
    \end{subfigure}
    \caption{Dependence of normalized $F(\Pi \underline{X}^k)$ on the iteration number $k$ and function similarities $\gamma$. Left: MT-PDHG ($\mu=0$). Right: AMT-PDHG ($\mu=0.01$). Legends represent $(\text{type of setup},\gamma)$ and line colors represent $\gamma$. }\label{fig:exp2}
\end{figure*}

\section{Conclusion and future direction}\label{sec:conclusion}
In this work, we propose the (accelerated) multi-timescale PDHG algorithms for saddle point problems with block-decomposable duals. Our (A)MT-PDHG allows arbitrary updating rates for dual blocks while remaining fully deterministic and robust to extreme delays in dual updates. We further apply (A)MT-PDHG to distributed optimization and demonstrate how the flexibility in choosing the updating rates could help improve the overall algorithm efficiencies in heterogeneous environments.

To make the algorithms more practical, one direction of future work is to develop more space-efficient algorithms: currently each primal agents need to store $O(r_{\max})$ vectors in $\R^d$, which are used as mixtures of proximal centers and when calculating messages to dual agents. It is an interesting question whether one can achieve similar convergence guarantees with smaller memory requirement. Another promising direction is to extend the multi-timescale update mechanism to a broader class of algorithms, including those for saddle point and more general optimization problems.

\section*{Acknowledgements}
This work was funded by the Office of Naval Research grant N00014-24-1-2470.

\printbibliography

@inproceedings{Cohen2021,
 author = {Cohen, Alon and Daniely, Amit and Drori, Yoel and Koren, Tomer and Schain, Mariano},
 booktitle = {Advances in Neural Information Processing Systems},
 editor = {M. Ranzato and A. Beygelzimer and Y. Dauphin and P.S. Liang and J. Wortman Vaughan},
 pages = {9024--9035},
 publisher = {Curran Associates, Inc.},
 title = {Asynchronous Stochastic Optimization Robust to Arbitrary Delays},
 volume = {34},
 year = {2021}
}

@article{Liu2015Asynchronous,
author = {Liu, Ji and Wright, Stephen J.},
title = {Asynchronous Stochastic Coordinate Descent: Parallelism and Convergence Properties},
journal = {SIAM Journal on Optimization},
volume = {25},
number = {1},
pages = {351-376},
year = {2015},
doi = {10.1137/140961134}}

@ARTICLE{Angelia2015,
  author={Nedić, Angelia and Olshevsky, Alex},
  journal={IEEE Transactions on Automatic Control}, 
  title={Distributed Optimization Over Time-Varying Directed Graphs}, 
  year={2015},
  volume={60},
  number={3},
  pages={601-615},
  doi={10.1109/TAC.2014.2364096}}

@InProceedings{Attia2025,
  title = 	 {Faster Stochastic Optimization with Arbitrary Delays via Adaptive Asynchronous Mini-Batching},
  author =       {Attia, Amit and Gaash, Ofir and Koren, Tomer},
  booktitle = 	 {Proceedings of the 42nd International Conference on Machine Learning},
  pages = 	 {1931--1949},
  year = 	 {2025},
  editor = 	 {Singh, Aarti and Fazel, Maryam and Hsu, Daniel and Lacoste-Julien, Simon and Berkenkamp, Felix and Maharaj, Tegan and Wagstaff, Kiri and Zhu, Jerry},
  volume = 	 {267},
  series = 	 {Proceedings of Machine Learning Research},
  month = 	 {13--19 Jul},
  publisher =    {PMLR},
}

@article{lan_communication-efficient_2020,
	title = {Communication-efficient algorithms for decentralized and stochastic optimization},
	volume = {180},
	issn = {1436-4646},
	number = {1},
	journal = {Mathematical Programming},
	author = {Lan, Guanghui and Lee, Soomin and Zhou, Yi},
	month = mar,
	year = {2020},
	pages = {237--284},
}

@book{rockafellarVariational,
    author = {R. Tyrrell Rockafellar, Roger J. B. Wets},
    title = {Variational Analysis},
    publisher = {Springer Science \& Business Media},
    year = {2009},
address = {Berlin, Germany}
}

@article{Meyer1973Generalized,
 ISSN = {00361399},
 author = {Carl D. Meyer},
 journal = {SIAM Journal on Applied Mathematics},
 number = {3},
 pages = {315--323},
 publisher = {Society for Industrial and Applied Mathematics},
 title = {Generalized Inversion of Modified Matrices},
 volume = {24},
 year = {1973}
}

@article{stephen_distributed_2011,
	title = {Distributed {Optimization} and {Statistical} {Learning} via the {Alternating} {Direction} {Method} of {Multipliers}},
	volume = {3},
	issn = {1554-0669},
	url = {https://doi.org/10.1561/2200000016},
	doi = {10.1561/2200000016},
	number = {1},
	journal = {Foundations and Trends in Information Retrieval},
	author = {Stephen, Boyd and Neal, Parikh and Eric, Chu and Borja, Peleato and Jonathan, Eckstein},
	month = jul,
	year = {2011},
	pages = {1--122},
}

@article{CC01a,
 author = {Chang, Chih-Chung and Lin, Chih-Jen},
 title = {{LIBSVM}: A library for support vector machines},
 journal = {ACM Transactions on Intelligent Systems and Technology},
 volume = {2},
 issue = {3},
 year = {2011},
 pages = {27:1--27:27},
 note =	 {Software available at \url{http://www.csie.ntu.edu.tw/~cjlin/libsvm}}
}

@inproceedings{ArjevaniShamir2015Communication,
author = {Arjevani, Yossi and Shamir, Ohad},
title = {Communication complexity of distributed convex learning and optimization},
year = {2015},
publisher = {MIT Press},
address = {Cambridge, MA, USA},
booktitle = {Proceedings of the 29th International Conference on Neural Information Processing Systems - Volume 1},
pages = {1756–1764},
numpages = {9},
location = {Montreal, Canada},
series = {NIPS'15}
}

@ARTICLE{NedicOlshevskyRabbat2018,
  author={Nedić, Angelia and Olshevsky, Alex and Rabbat, Michael G.},
  journal={Proceedings of the IEEE}, 
  title={Network Topology and Communication-Computation Tradeoffs in Decentralized Optimization}, 
  year={2018},
  volume={106},
  number={5},
  pages={953-976}}

@ARTICLE{AssranAytekinFeyzmahdavianJohanssonRabbat2020,
  author={Assran, By Mahmoud and Aytekin, Arda and Feyzmahdavian, Hamid Reza and Johansson, Mikael and Rabbat, Michael G.},
  journal={Proceedings of the IEEE}, 
  title={Advances in Asynchronous Parallel and Distributed Optimization}, 
  year={2020},
  volume={108},
  number={11},
  pages={2013-2031}}

@misc{konecny2016federatedoptimizationdistributedmachine,
      title={Federated Optimization: Distributed Machine Learning for On-Device Intelligence}, 
      author={Jakub Konečný and H. Brendan McMahan and Daniel Ramage and Peter Richtárik},
      year={2016},
      eprint={1610.02527},
      archivePrefix={arXiv},
      primaryClass={cs.LG},
}

@ARTICLE{TsitsiklisBertsekasAthans1986,
  author={Tsitsiklis, J. and Bertsekas, D. and Athans, M.},
  journal={IEEE Transactions on Automatic Control}, 
  title={Distributed asynchronous deterministic and stochastic gradient optimization algorithms}, 
  year={1986},
  volume={31},
  number={9},
  pages={803-812}}

@book{BertsekasTsitsiklis1989,
author = {Bertsekas, Dimitri P. and Tsitsiklis, John N.},
title = {Parallel and distributed computation: numerical methods},
year = {1989},
isbn = {0136487009},
publisher = {Prentice-Hall, Inc.},
address = {USA}
}

@inproceedings{NiuRechtReWright2011,
author = {Niu, Feng and Recht, Benjamin and Re, Christopher and Wright, Stephen J.},
title = {HOGWILD! a lock-free approach to parallelizing stochastic gradient descent},
year = {2011},
isbn = {9781618395993},
publisher = {Curran Associates Inc.},
address = {Red Hook, NY, USA},
booktitle = {Proceedings of the 25th International Conference on Neural Information Processing Systems},
pages = {693–701},
numpages = {9},
location = {Granada, Spain},
series = {NIPS'11}
}

@article{ShaiTong2013,
  author  = {Shai Shalev-Shwartz and Tong Zhang},
  title   = {Stochastic Dual Coordinate Ascent Methods for Regularized Loss Minimization},
  journal = {Journal of Machine Learning Research},
  year    = {2013},
  volume  = {14},
  number  = {16},
  pages   = {567--599},
}

@ARTICLE{NedicOzdaglar2009,
  author={Nedic, Angelia and Ozdaglar, Asuman},
  journal={IEEE Transactions on Automatic Control}, 
  title={Distributed Subgradient Methods for Multi-Agent Optimization}, 
  year={2009},
  volume={54},
  number={1},
  pages={48-61}}

@ARTICLE{Duchi2012DA,
  author={Duchi, John C. and Agarwal, Alekh and Wainwright, Martin J.},
  journal={IEEE Transactions on Automatic Control}, 
  title={Dual Averaging for Distributed Optimization: Convergence Analysis and Network Scaling}, 
  year={2012},
  volume={57},
  number={3},
  pages={592-606}}

@article{Scaman2019,
  author  = {Kevin Scaman and Francis Bach and S{{\'e}}bastien Bubeck and Yin Tat Lee and Laurent Massouli{{\'e}}},
  title   = {Optimal Convergence Rates for Convex Distributed Optimization in Networks},
  journal = {Journal of Machine Learning Research},
  year    = {2019},
  volume  = {20},
  number  = {159},
  pages   = {1--31},
}

@inproceedings{Tyurin2024,
 author = {Tyurin, Alexander and Richt\'{a}rik, Peter},
 booktitle = {Advances in Neural Information Processing Systems},
 editor = {A. Globerson and L. Mackey and D. Belgrave and A. Fan and U. Paquet and J. Tomczak and C. Zhang},
 pages = {122652--122705},
 publisher = {Curran Associates, Inc.},
 title = {On the Optimal Time Complexities in Decentralized Stochastic Asynchronous Optimization},
 volume = {37},
 year = {2024},
address = {Red Hook,NY}
}

@InProceedings{pmlr-v32-shamir14,
  title = 	 {Communication-Efficient Distributed Optimization using an Approximate Newton-type Method},
  author = 	 {Shamir, Ohad and Srebro, Nati and Zhang, Tong},
  booktitle = 	 {Proceedings of the 31st International Conference on Machine Learning},
  pages = 	 {1000--1008},
  year = 	 {2014},
  editor = 	 {Xing, Eric P. and Jebara, Tony},
  volume = 	 {32},
  number =       {2},
  series = 	 {Proceedings of Machine Learning Research},
  address = 	 {Bejing, China},
  month = 	 {22--24 Jun},
  publisher =    {PMLR},
  
}

@article{chambolle_first-order_2011,
	title = {A {First}-{Order} {Primal}-{Dual} {Algorithm} for {Convex} {Problems} with {Applications} to {Imaging}},
	volume = {40},
	issn = {1573-7683},
	number = {1},
	journal = {Journal of Mathematical Imaging and Vision},
	author = {Chambolle, Antonin and Pock, Thomas},
	month = may,
	year = {2011},
	pages = {120--145},
}

@article{Chambolle2018Stochastic,
author = {Chambolle, Antonin and Ehrhardt, Matthias J. and Richt\'{a}rik, Peter and Sch\"{o}nlieb, Carola-Bibiane},
title = {Stochastic Primal-Dual Hybrid Gradient Algorithm with Arbitrary Sampling and Imaging Applications},
journal = {SIAM Journal on Optimization},
volume = {28},
number = {4},
pages = {2783-2808},
year = {2018}}

@inproceedings{Zhu2016,
author = {Allen-Zhu, Zeyuan and Yuan, Yang and Sridharan, Karthik},
title = {Exploiting the structure: stochastic Gradient methods using raw clusters},
year = {2016},
isbn = {9781510838819},
publisher = {Curran Associates Inc.},
address = {Red Hook, NY, USA},
booktitle = {Proceedings of the 30th International Conference on Neural Information Processing Systems},
pages = {1650–1658},
numpages = {9},
location = {Barcelona, Spain},
series = {NIPS'16}
}

@article{Alacaoglu2022,
author = {Alacaoglu, Ahmet and Fercoq, Olivier and Cevher, Volkan},
title = {On the Convergence of Stochastic Primal-Dual Hybrid Gradient},
journal = {SIAM Journal on Optimization},
volume = {32},
number = {2},
pages = {1288-1318},
year = {2022}}

@article{wright_coordinate_2015,
	title = {Coordinate descent algorithms},
	volume = {151},
	issn = {1436-4646},
	number = {1},
	journal = {Mathematical Programming},
	author = {Wright, Stephen J.},
	month = jun,
	year = {2015},
	pages = {3--34},
}

@inproceedings{Sun2017,
author = {Sun, Tao and Hannah, Robert and Yin, Wotao},
title = {Asynchronous coordinate descent under more realistic assumption},
year = {2017},
isbn = {9781510860964},
publisher = {Curran Associates Inc.},
address = {Red Hook, NY, USA},
booktitle = {Proceedings of the 31st International Conference on Neural Information Processing Systems},
pages = {6183–6191},
numpages = {9},
location = {Long Beach, California, USA},
series = {NIPS'17}
}

@ARTICLE{Aybat2018,
  author={Aybat, N. S. and Wang, Z. and Lin, T. and Ma, S.},
  journal={IEEE Transactions on Automatic Control}, 
  title={Distributed Linearized Alternating Direction Method of Multipliers for Composite Convex Consensus Optimization}, 
  year={2018},
  volume={63},
  number={1},
  pages={5-20}}

@misc{wei2013o1kconvergenceasynchronousdistributed,
      title={On the O(1/k) Convergence of Asynchronous Distributed Alternating Direction Method of Multipliers}, 
      author={Ermin Wei and Asuman Ozdaglar},
      year={2013},
      eprint={1307.8254},
      archivePrefix={arXiv},
      primaryClass={math.OC},
}

@inbook{Santosh2020,
author = {Santosh S. Vempala and Ruosong Wang and David P. Woodruff},
title = {The Communication Complexity of Optimization},
booktitle = {Proceedings of the 2020 ACM-SIAM Symposium on Discrete Algorithms (SODA)},
chapter = {},
pages = {1733-1752},
year = {2020},
address = {Philadelphia, PA},
publisher = {Society for Industrial and Applied Mathematics}}

@InProceedings{pmlr-v37-zhangb15,
  title = 	 {DiSCO: Distributed Optimization for Self-Concordant Empirical Loss},
  author = 	 {Zhang, Yuchen and Lin, Xiao},
  booktitle = 	 {Proceedings of the 32nd International Conference on Machine Learning},
  pages = 	 {362--370},
  year = 	 {2015},
  editor = 	 {Bach, Francis and Blei, David},
  volume = 	 {37},
  series = 	 {Proceedings of Machine Learning Research},
  address = 	 {Lille, France},
  month = 	 {07--09 Jul},
  publisher =    {PMLR},
}

@inproceedings{Kovalev2022,
author = {Kovalev, Dmitry and Beznosikov, Aleksandr and Borodich, Ekaterina and Gasnikov, Alexander and Scutari, Gesualdo},
title = {Optimal gradient sliding and its application to distributed optimization under similarity},
year = {2022},
isbn = {9781713871088},
publisher = {Curran Associates Inc.},
address = {Red Hook, NY, USA},
booktitle = {Proceedings of the 36th International Conference on Neural Information Processing Systems},
articleno = {2427},
numpages = {14},
location = {New Orleans, LA, USA},
series = {NIPS '22}
}

@InProceedings{pmlr-v151-tian22b,
  title = 	 { Acceleration in Distributed Optimization under Similarity },
  author =       {Tian, Ye and Scutari, Gesualdo and Cao, Tianyu and Gasnikov, Alexander},
  booktitle = 	 {Proceedings of The 25th International Conference on Artificial Intelligence and Statistics},
  pages = 	 {5721--5756},
  year = 	 {2022},
  editor = 	 {Camps-Valls, Gustau and Ruiz, Francisco J. R. and Valera, Isabel},
  volume = 	 {151},
  series = 	 {Proceedings of Machine Learning Research},
  month = 	 {28--30 Mar},
  publisher =    {PMLR},
address = {virtual conference}
}

@article{Ghadimi2012,
author = {Ghadimi, Saeed and Lan, Guanghui},
title = {Optimal Stochastic Approximation Algorithms for Strongly Convex Stochastic Composite Optimization I: A Generic Algorithmic Framework},
journal = {SIAM Journal on Optimization},
volume = {22},
number = {4},
pages = {1469-1492},
year = {2012}}

@article{bertsekas_incremental_2011,
	title = {Incremental proximal methods for large scale convex optimization},
	volume = {129},
	issn = {1436-4646},
	number = {2},
	journal = {Mathematical Programming},
	author = {Bertsekas, Dimitri P.},
	month = oct,
	year = {2011},
	pages = {163--195},
}

@InProceedings{pmlr-v37-zhanga15,
  title = 	 {Stochastic Primal-Dual Coordinate Method for Regularized Empirical Risk Minimization},
  author = 	 {Zhang, Yuchen and Lin, Xiao},
  booktitle = 	 {Proceedings of the 32nd International Conference on Machine Learning},
  pages = 	 {353--361},
  year = 	 {2015},
  editor = 	 {Bach, Francis and Blei, David},
  volume = 	 {37},
  series = 	 {Proceedings of Machine Learning Research},
  address = 	 {Lille, France},
  month = 	 {07--09 Jul},
  publisher =    {PMLR},
}

@article{chambolle_ergodic_2016,
	title = {On the ergodic convergence rates of a first-order primal–dual algorithm},
	volume = {159},
	issn = {1436-4646},
	number = {1},
	journal = {Mathematical Programming},
	author = {Chambolle, Antonin and Pock, Thomas},
	month = sep,
	year = {2016},
	pages = {253--287},
}

@article{lan_gradient_2016,
	title = {Gradient sliding for composite optimization},
	volume = {159},
	issn = {1436-4646},
	number = {1},
	journal = {Mathematical Programming},
	author = {Lan, Guanghui},
	month = sep,
	year = {2016},
	pages = {201--235},
}

@book{Boyd2011,
  author={Boyd, Stephen and Parikh, Neal and Chu, Eric and Peleato, Borja and Eckstein, Jonathan},
  title={Distributed Optimization and Statistical Learning via the Alternating Direction Method of Multipliers},
  year={2011},
  volume={},
  number={},
  pages={}}

@ARTICLE{Molzahn2017,
  author={Molzahn, Daniel K. and Dörfler, Florian and Sandberg, Henrik and Low, Steven H. and Chakrabarti, Sambuddha and Baldick, Ross and Lavaei, Javad},
  journal={IEEE Transactions on Smart Grid}, 
  title={A Survey of Distributed Optimization and Control Algorithms for Electric Power Systems}, 
  year={2017},
  volume={8},
  number={6},
  pages={2941-2962}}

@article{Nedic2018,
   author = "Nedić, Angelia and Liu, Ji",
   title = "Distributed Optimization for Control", 
   journal= "Annual Review of Control, Robotics, and Autonomous Systems",
   year = "2018",
   volume = "1",
   number = "Volume 1, 2018",
   pages = "77-103",
   publisher = "Annual Reviews",
   issn = "2573-5144",
   type = "Journal Article"}

@ARTICLE{Berahas2019,
  author={Berahas, Albert S. and Bollapragada, Raghu and Keskar, Nitish Shirish and Wei, Ermin},
  journal={IEEE Transactions on Automatic Control}, 
  title={Balancing Communication and Computation in Distributed Optimization}, 
  year={2019},
  volume={64},
  number={8},
  pages={3141-3155}}

@inproceedings{Jakub2016,title	= {Federated Learning: Strategies for Improving Communication Efficiency},author	= {Jakub Konečný and H. Brendan McMahan and Felix X. Yu and Peter Richtarik and Ananda Theertha Suresh and Dave Bacon},year	= {2016},booktitle	= {NIPS Workshop on Private Multi-Party Machine Learning}}

@INPROCEEDINGS{Rabbat2004,
  author={Rabbat, M. and Nowak, R.},
  booktitle={Third International Symposium on Information Processing in Sensor Networks, 2004. IPSN 2004}, 
  title={Distributed optimization in sensor networks}, 
  year={2004},
  volume={},
  number={},
  pages={20-27}}

@ARTICLE{Dan2002,
  author={Dan Li and Wong, K.D. and Yu Hen Hu and Sayeed, A.M.},
  journal={IEEE Signal Processing Magazine}, 
  title={Detection, classification, and tracking of targets}, 
  year={2002},
  volume={19},
  number={2},
  pages={17-29}}

@book{Bullo2009,
  author={Bullo, Francesco and Cortés, Jorge and Martínez, Sonia},
  title={Distributed Control of Robotic Networks: A Mathematical Approach to Motion Coordination Algorithms},
  year={2009},
  volume={},
  number={},
  pages={}}

@ARTICLE{Jaleel2020,
  author={Jaleel, Hassan and Shamma, Jeff S.},
  journal={Proceedings of the IEEE}, 
  title={Distributed Optimization for Robot Networks: From Real-Time Convex Optimization to Game-Theoretic Self-Organization}, 
  year={2020},
  volume={108},
  number={11},
  pages={1953-1967}}

@ARTICLE{Shorinwa2024,
  author={Shorinwa, Ola and Halsted, Trevor and Yu, Javier and Schwager, Mac},
  journal={IEEE Robotics \& Automation Magazine}, 
  title={Distributed Optimization Methods for Multi-Robot Systems: Part 1—A Tutorial [Tutorial]}, 
  year={2024},
  volume={31},
  number={3},
  pages={121-138}}

@article{Csiba2018,
  author  = {Dominik Csiba and Peter Richt{{\'a}}rik},
  title   = {Importance Sampling for Minibatches},
  journal = {Journal of Machine Learning Research},
  year    = {2018},
  volume  = {19},
  number  = {27},
  pages   = {1--21},
}

@InProceedings{pmlr-v119-karimireddy20a,
  title = 	 {{SCAFFOLD}: Stochastic Controlled Averaging for Federated Learning},
  author =       {Karimireddy, Sai Praneeth and Kale, Satyen and Mohri, Mehryar and Reddi, Sashank and Stich, Sebastian and Suresh, Ananda Theertha},
  booktitle = 	 {Proceedings of the 37th International Conference on Machine Learning},
  pages = 	 {5132--5143},
  year = 	 {2020},
  editor = 	 {III, Hal Daumé and Singh, Aarti},
  volume = 	 {119},
  series = 	 {Proceedings of Machine Learning Research},
  month = 	 {13--18 Jul},
  publisher =    {PMLR},
address = {virtual conference}
}

@InProceedings{pmlr-v119-koloskova20a,
  title = 	 {A Unified Theory of Decentralized {SGD} with Changing Topology and Local Updates},
  author =       {Koloskova, Anastasia and Loizou, Nicolas and Boreiri, Sadra and Jaggi, Martin and Stich, Sebastian},
  booktitle = 	 {Proceedings of the 37th International Conference on Machine Learning},
  pages = 	 {5381--5393},
  year = 	 {2020},
  editor = 	 {III, Hal Daumé and Singh, Aarti},
  volume = 	 {119},
  series = 	 {Proceedings of Machine Learning Research},
  month = 	 {13--18 Jul},
  publisher =    {PMLR},
address = {virtual conference}
}

@ARTICLE{Zeng2021,
  author={Zeng, Qunsong and Du, Yuqing and Huang, Kaibin and Leung, Kin K.},
  journal={IEEE Transactions on Wireless Communications}, 
  title={Energy-Efficient Resource Management for Federated Edge Learning With CPU-GPU Heterogeneous Computing}, 
  year={2021},
  volume={20},
  number={12},
  pages={7947-7962},
  doi={10.1109/TWC.2021.3088910}}

@article{applegate_faster_2023,
	title = {Faster first-order primal-dual methods for linear programming using restarts and sharpness},
	volume = {201},
	issn = {1436-4646},
	number = {1},
	journal = {Mathematical Programming},
	author = {Applegate, David and Hinder, Oliver and Lu, Haihao and Lubin, Miles},
	month = sep,
	year = {2023},
	pages = {133--184},
}

@article{Peng2016ARock,
author = {Peng, Zhimin and Xu, Yangyang and Yan, Ming and Yin, Wotao},
title = {ARock: An Algorithmic Framework for Asynchronous Parallel Coordinate Updates},
journal = {SIAM Journal on Scientific Computing},
volume = {38},
number = {5},
pages = {A2851-A2879},
year = {2016},
doi = {10.1137/15M1024950}}

@inproceedings{Woodworth2020,
 author = {Woodworth, Blake E and Patel, Kumar Kshitij and Srebro, Nati},
 booktitle = {Advances in Neural Information Processing Systems},
 editor = {H. Larochelle and M. Ranzato and R. Hadsell and M.F. Balcan and H. Lin},
 pages = {6281--6292},
 publisher = {Curran Associates, Inc.},
 title = {Minibatch vs Local SGD for Heterogeneous Distributed Learning},
 volume = {33},
 year = {2020},
address = {Red Hook, NY}
}

@inproceedings{stich2019local,
  title={Local SGD Converges Fast and Communicates Little},
  author={Stich, Sebastian Urban},
  booktitle={ICLR 2019-International Conference on Learning Representations},
  year={2019}
}

@inproceedings{Zinkevich2010,
author = {Zinkevich, Martin A. and Weimer, Markus and Smola, Alex and Li, Lihong},
title = {Parallelized stochastic gradient descent},
year = {2010},
publisher = {Curran Associates Inc.},
address = {Red Hook, NY, USA},
booktitle = {Proceedings of the 24th International Conference on Neural Information Processing Systems - Volume 2},
pages = {2595–2603},
numpages = {9},
location = {Vancouver, British Columbia, Canada},
series = {NIPS'10}
}

\appendix
\section{Generalized gradient sliding procedure}\label{sec:generalized-CS}
In Algorithm \ref{alg:CS-procedure}, we provide the generalized gradient sliding procedure. Then we provide the proofs of its properties. 

\begin{algorithm}
\caption{Generalized gradient sliding procedure}\label{alg:CS-procedure}
\begin{algorithmic}
\Require The sequences $\{\beta_t\}$ and $\{\lambda_t\}$, $\phi':U\to \R^{d_0}$ a subgradient oracle for $\phi$. 
\Ensure $(u^T,\widehat{u}^T) = GS(\phi,U,D,T,(\eta_i)_{i\in \mc I},v,(x_i)_{i\in \mc I},x^{init})$, an approximate solution to
\begin{displaymath}
    \min_{u\in U} \Phi(u):=\langle v,u\rangle + \phi(u) + \sum_{i\in \mc I} \eta_i D(u,x_i) 
\end{displaymath}
\State $(u^0,\widehat{u}^0)\gets (x^{init},x^{init})$, $\eta\gets \sum_{i\in \mc I} \eta_i$
\For{$t=1,\ldots,T$} 
\State \begin{displaymath}
    u^t = \argmin_{u\in U}\langle v+  \phi'(u^{t-1}),u\rangle + \sum_{i\in \mc I} \eta_i D(u,x_i) + \eta \beta_tD(u,u^{t-1})
\end{displaymath}
\EndFor
\State \begin{displaymath}
    \widehat{u}^T = (\sum_{t=1}^T \lambda_t)^{-1} \sum_{t=1}^T \lambda_t u^t.
\end{displaymath}
\end{algorithmic}
\end{algorithm}

\begin{lemma}[generalized lemma 5 in \cite{lan_communication-efficient_2020}]\label{lm:generalized_lm5}
    Let the convex function $q:U\to \R$, and $\mc I$ an arbitrary finite index set. Assume that the points $x_i\in U$ and the numbers $\eta_i\geq 0$ for $i\in \mc I$. Let $w:U\to \R$ be a distance generating function and 
    \begin{displaymath}
        u^* \in \argmin_{u\in U} q(u) + \sum_{i\in \mc I} \eta_i D(u,x_i).
    \end{displaymath}
    Then for any $u\in U$, we have
    \begin{displaymath}
        q(u^*) + \sum_{i\in \mc I} \eta_i D(u^*,x_i) \leq q(u) + \sum_{i\in \mc I} \eta_i D(u,x_i) - \sum_{i\in \mc I} \eta_i D(u,u^*).
    \end{displaymath}
\end{lemma}
\begin{proof}[Proof of Lemma \ref{lm:generalized_lm5}]
    First, by the optimality condition for $u^*$, there exists $q'(u^*)\in \partial q(u^*) $ such that  
    \begin{displaymath}
        \langle q'(u^*) + \sum_{i\in \mc I} \eta_i \nabla D(u^*,x_i), u-u^*\rangle \geq 0,\quad \forall u\in U. 
    \end{displaymath}
    By definition, we have for each $i\in \mc I$ that 
    \begin{displaymath}
        D(u,x_i)- D(u^*,x_i) - D(u,u^*) = \langle \nabla w(x_i) - \nabla w(u^*),u-u^*\rangle =  -\langle \nabla D(u^*,x_i) ,u-u^*\rangle 
    \end{displaymath}
Thus, we have for any $u\in U$, 
\begin{align*}
        &\quad q(u) + \sum_{i\in \mc I} \eta_i D(u,x_i)\\
        &\geq q(u^*) + \langle q'(u^*),u-u^*\rangle + \sum_{i\in \mc I} \eta_i \left( D(u^*,x_i) + D_w(u,u^*) -  \langle \nabla D(u^*,x_i),u-u^*\rangle \right)\\
        &\geq q(u^*)  + \sum_{i\in \mc I} \eta_i  D(u^*,x_i) +\sum_{i\in \mc I} \eta_i  D(u,u^*).
    \end{align*}
\end{proof}

\begin{lemma}\label{lm:modified_prop2}
Assume that $U\subset \R^{d_0}$ is a convex set, and $\phi:U\to \R$ is a convex function such that 
\begin{displaymath}
    \frac{\mu}{2}\|x-y\|^2\leq \phi(x) - \phi(y)-\langle \phi'(y),x-y\rangle\leq M\|x-y\|,\quad\forall x,y\in U,
\end{displaymath}
where $\phi':U\to \R^{d_0}$ is a subgradient oracle, i.e. for each $y\in U$, $\phi'(y)\in \partial \phi(y)$ is a subgradient. In addition, $D_{w^x}(x,x')\leq \frac{C}{2}\|x-x'\|^2$ for some $C\in [0,\infty]$. If $\{\beta_t\}$ and $\{\lambda_t\}$ in Algorithm \ref{alg:CS-procedure} satisfies that 
    \begin{displaymath}
        \lambda_{t+1} (\eta\beta_{t+1}-\mu/C) \leq \lambda_t(1+\beta_t)\eta,\quad \forall t\geq 1,
    \end{displaymath}
    then for $t\geq 1$ and $u\in U$
    \begin{displaymath}
        (\sum_{t=1}^T \lambda_t)\cdot(\Phi(\widehat{u}^T) - \Phi(u))\leq (\eta\beta_1-\mu/C)\lambda_1D(u,u^0)  - \eta(1+\beta_T)\lambda_TD(u,u^T)+\sum_{t=1}^T \frac{M^2\lambda_t}{2\eta\beta_t}.
    \end{displaymath}
\end{lemma}
\begin{proof}[Proof of Lemma \ref{lm:modified_prop2}]
    
Applying Lemma \ref{lm:generalized_lm5}, and using $\sum_{i\in \mc I}\eta_i = \eta$, we have
\begin{align*}
    &\quad \langle v+\phi'(u^{t-1}),u^{t}-u\rangle + \sum_{i\in \mc I}\eta_iD(u^t,x_i)  - \sum_{i\in \mc I}\eta_iD(u,x_i) \\
    &\leq \eta \beta_tD(u,u^{t-1}) -\eta \beta_tD(u^t,u^{t-1})-(1+\beta_t)\eta D(u,u^t)
\end{align*}

The rest follows a similar argument as in the proof of Proposition 2 \cite{lan_communication-efficient_2020}. 

\end{proof}

As a corollary to Lemma \ref{lm:modified_prop2}, we have the following performance guarantee.
\begin{corollary}\label{cor:lambda_beta}
Assume that $U\subset \R^{d_0}$ is a convex set, and $\phi:U\to \R$ is a convex function such that 
\begin{displaymath}
    \frac{\mu}{2}\|x-y\|^2\leq \phi(x) - \phi(y)-\langle \phi'(y),x-y)\rangle\leq M\|x-y\|,\quad\forall x,y\in U,
\end{displaymath}
where $\phi':U\to \R^{d_0}$ is a subgradient oracle, i.e. for each $y\in U$, $\phi'(y)\in \partial \phi(y)$ is a subgradient. With $\lambda_t = t+1$ and $\beta_t =\frac{t}{2}$ for $t\geq 1$, we have for any $u\in U$
\begin{align*}
     \langle v,\widehat{u}^T-u\rangle + \phi(\widehat{u}^T) - \phi(u) &\leq  \frac{2\eta}{T(T+3)}D(u,x^{init})  + \sum_{i\in \mc I} \eta_i D({u},x_i)\\
     & \quad- \frac{(T+1)(T+2)}{T(T+3)}\eta D(u,u^T) -\sum_{i\in \mc I} \eta_i D(\widehat{u}^T,x_i)+ \frac{4M^2}{\eta(T+3)}.
\end{align*}
Further, if $\mu >0$, and $D_{w^{x}}(x,x')\leq \frac{C}{2}\|x-x'\|^2$ for some $C<\infty$, then denoting $\eta = \sum_{i\in \mc I}\eta_i$, setting $\lambda_t = t$ and $\beta_t = \frac{(t+1)\mu}{2\eta C} +\frac{t-1}{2}$, we have for any $u\in U$,
\begin{align*}
    \langle v,\widehat{u}^T-u\rangle + \phi(\widehat{u}^T) - \phi(u)  &\leq  \sum_{i\in \mc I} \eta_i D({u},x_i) - \sum_{i\in \mc I} \eta_i D(\widehat{u}^T,x_i) \\
    &\quad - (\frac{\mu}{ C} +\eta)D(u,u^T) + \frac{2M^2/\eta}{T(T+1)}\sum_{t=1}^T \frac{\lambda_t}{\beta_t},
\end{align*}
and $\frac{2M^2/\eta}{T(T+1)}\sum_{t=1}^T \frac{\lambda_t}{\beta_t} \leq \frac{4CM^2}{\mu (T+1)}$. 
\end{corollary}

\section{Proof for Section \ref{sec:agents}}\label{sec:proof_setup}

\begin{proof}[Proof of Lemma \ref{lm:K_prop}]
    For $\overline{d}=1$, the matrix representation of $K_s\in \R^{|\child(s)|\times m}$ is $K_s = (I - \mb 1(\frac{|\Descendant(j)|}{|\Descendant(s)|})_{j\in \child(s)}^T)P_s$ where $P_s\in \R^{|\child(s)|\times m}$, and $P_s(i,j) = |\Descendant(i)|^{-1}$ if $j\in \Descendant(i)$ and $P_s(i,j)=0$ otherwise. Notice that
\begin{displaymath}
    K_sK_{s'}^* = (I - \mb 1(\frac{|\Descendant(j)|}{|\Descendant(s)|})_{j\in \child(s)}^T)P_sP_{s'}^T(I - \mb 1(\frac{|\Descendant(j)|}{|\Descendant(s')|})_{j\in \child(s')}^T)^T.
\end{displaymath}
If $s$ is not in the subtree rooted at $s'$ and $s'$ is not in the subtree rooted at $s$, then $\Descendant(s)\cap \Descendant(s') = \emptyset$, and so $P_sP_{s'}^T = \mb 0$. If $s$ is in the subtree rooted at $s'$, then $s$ is in the subtree rooted at some $\widehat{s}\in \child(s')$. In particular, $(P_sP_{s'}^T)(i,j) = 0$ for all $j\neq \widehat{s}$ and $(P_sP_{s'}^T)(i,\widehat{s}) = |\Descendant(\widehat{s})|^{-1}$, and thus $(I - \mb 1(\frac{|\Descendant(j)|}{|\Descendant(s)|})_{j\in \child(s)}^T)P_sP_{s'}^T = \mb 0$. Similarly for the case when $s'$ is in the subtree rooted at $s$. The case when $\overline{d}>1$ follows by applying the above argument coordinate-wise.

For the second claim, consider the case $\overline{d}=1$, denoting 
\begin{displaymath}
    D = \diag((|\Descendant(j)|)_{j\in \child(s)}),\quad v = (|\Descendant(j)|)_{j\in \child(s)}/\|(|\Descendant(j)|)_{j\in \child(s)}\|_2,
\end{displaymath}
where the norm in the denominator in the definition of $v$ is the $l_2$ norm. Then when $\overline{d}=1$, we have (applying Theorem 6 in \cite{Meyer1973Generalized})
\begin{equation}\label{eq:K_prod}
    K_sK_{s}^* = D^{-1} - \frac{1}{|\Descendant(s)|}\mb 1\mb 1^T,\quad (K_sK_{s}^*)^{\dagger} = D - vv^TD - Dvv^T + (v^TDv)vv^T.
\end{equation}
Thus, noticing that $v^TK_s = \mb 0$, we have 
\begin{displaymath}
    \Pi_s = K_s^*(K_sK_s^*)^{\dagger}K_s = K_s^TDK_s.
\end{displaymath}

Thus, for any $\widetilde{X},\widehat{X}\in \R^{m}$
    \begin{displaymath}
        \langle \widehat{X},\Pi_s\widetilde{X}\rangle=\langle \Pi_s\widehat{X},\Pi_s\widetilde{X}\rangle= (K_s\widehat{X})^TD(K_s\widetilde{X})= \sum_{i\in \child(s)} |\Descendant(i)| \cdot \langle (K_s\widehat{X})_i,(K_s\widetilde{X})_i\rangle.
    \end{displaymath}

The above argument can be applied coordinate-wise, and so extend to $\overline{d}\geq 1$.
\end{proof}

\end{document}